\documentclass[11pt,reqno]{amsart}
\usepackage{a4wide}
\usepackage{color}
\usepackage[colorlinks, linkcolor=blue, citecolor=blue]{hyperref}
\usepackage{amssymb,latexsym}
\usepackage{amsmath}
\usepackage{amsthm}
\usepackage{graphicx}
\usepackage{hyperref}
\usepackage{titletoc}
\usepackage{pdfsync}
\usepackage{color}
\numberwithin{equation}{section}
\newtheorem{theorem}{Theorem}[section]

\newtheorem{lemma}[theorem]{Lemma}

\newtheorem{remark}[theorem]{Remark}

\usepackage{cite}

\theoremstyle{definition}

\newtheorem{Rem}{Remark}[section]

\def\R{{\mathfrak R}}

\def\begeq{\begin{equation}}
\def\endeq{\end{equation}}

\def\R{\Bbb R}

\allowdisplaybreaks[4]

\begin{document}

\title[the Hartree type Br\'ezis-Nirenberg problem ]{Qualitative analysis to an eigenvalue problem of\\ the Hartree type Br\'ezis-Nirenberg problem  }

\author[K. Pan, S. Wen, J. Yang]{Kefan Pan, Shixin Wen and Jing Yang}

\address[Kefan Pan]{School of Mathematics and Statistics, Central China Normal University, Wuhan 430079, China} \email{kfpan@mails.ccnu.edu.cn}

\address[Shixin Wen]{School of Mathematics and Statistics, Central China Normal University, Wuhan 430079, China} \email{sxwen@mails.ccnu.edu.cn}

\address[Jing Yang]{Department of Mathematics, Zhejiang University of Technology, Zhejiang 310014, China} \email{yyangecho@163.com}

\begin{abstract}
In this paper, we are concerned with the critical Hartree equation
\begin{equation*}
\begin{cases}
-\Delta u=\left(\displaystyle{\displaystyle{\int_{\Omega}}}\frac{u^{2^{*}_{\mu}}(y)}{|x-y|^{\mu}}dy\right)u^{2^{*}_{\mu}-1}+\varepsilon u,\quad u>0,\quad  &\text{in $\Omega$,}\\
u=0,\quad &\text{on $\partial\Omega$,}
\end{cases}
\end{equation*}
where $\Omega\subset \R^{N}$ ($N\geq 5$) is a smooth bounded domain, $\mu\in (0,4)$ and $2^{*}_{\mu}=\frac{2N-\mu}{N-2}$ is the upper critical exponent in the sense of the Hardy-Littlewood-Sobolev inequality. Under a non-degeneracy condition on the critical point $x_0\in\Omega$ of the Robin function $R(x)$, we perform that for $\varepsilon>0$ sufficiently small, the Morse index of the blow-up solutions $u_\varepsilon$ concentrating at $x_0$ can be computed in terms of the negative eigenvalues of the Hessian matrix $D^{2}R(x)$ at $x_0$. Compared with the usual local cases, our problem is non-local due to the nonlinearity with Hartree-type, and several difficulties arise and new estimates of the eigenpairs $\{\left(\lambda_{i,\varepsilon},v_{i,\varepsilon}\right)\}$ to the associated linearized problem at $u_{\varepsilon}$ should be introduced. To our knowledge, this seems to be the first paper to consider the qualitative analysis of a Hartree type Br\'ezis-Nirenberg problem and
our results extend the works established by M. Grossi et al in \cite{GP}
and F. Takahashi in \cite{Ta3} to the non-local case.
\end{abstract}

\date{}\maketitle

\section{\label{Int}Introduction}
\setcounter{equation}{0}
\vspace{0.2cm}

This paper is devoted to the qualitative analysis of solutions to the non-local problem
\begin{equation}\label{eq 1.1}
\begin{cases}
-\Delta u=\left(\displaystyle{\displaystyle{\int_{\Omega}}}\frac{u^{2^{*}_{\mu}}(y)}{|x-y|^{\mu}}dy\right)u^{2^{*}_{\mu}-1}+\varepsilon u,\quad u>0,\quad  &\text{in $\Omega$,}\\
u=0,\quad &\text{on $\partial\Omega$,}
\end{cases}
\end{equation}
 where $\Omega\subset \R^{N}$ ($N\geq 5$) is a smooth bounded domain, $\mu\in (0,4)$, $\varepsilon>0$ is a small parameter and $2^{*}_{\mu}=\frac{2N-\mu}{N-2}$ is the upper critical exponent in the sense of the Hardy-Littlewood-Sobolev inequality. Different from the Br\'ezis-Nirenberg problem driven by the fractional Laplacian operator, \eqref{eq 1.1} is non-local since the Hartree-type nonlinearity, which makes it more complicated to handle than the  local semilinear elliptic problems.

In a celebrated paper \cite{BN}, Br\'ezis and Nirenberg considered the problem
\begin{equation}\label{eq1.2}
    \begin{cases}
        -\Delta u=u^{2^{*}-1}+\varepsilon u,\quad u>0,\quad  &\text{in $\Omega$,}\\[1mm]
u=0,\quad &\text{on $\partial\Omega$,}
    \end{cases}
\end{equation}
where $\Omega\subset \R^{N}$ ($N\geq3$) is a smooth bounded domain, $\varepsilon>0$ and $2^{*}=\frac{2N}{N-2}$ is the critical Sobolev exponent. In \cite{BN}, The authors overcame the difficulties due to the lack of compactness of the embedding $H^{1}_{0}(\Omega)
\hookrightarrow L^{2^{*}}(\Omega)$ and obtained the existence results of positive solutions to \eqref{eq1.2}. Since then, many results have been derived concerning the multiplicity, existence and nonexistence of nodal solutions, and qualitative properties of solutions (also see \cite{BLR,CFP,CLP,CSS1,CW,G,H,IP,MP1,SZ,Ta3,Rey1,Rey2}). Among the plenty of work dealt with \eqref{eq1.2}, we would like to recall the papers of Han \cite{H} and Rey \cite{Rey1,Rey2}, where the authors utilized different methods to study the asymptotic behavior of the positive solution $u_\varepsilon$ to \eqref{eq1.2} verifying
\begin{align}\label{eq1.3}
 \frac{\displaystyle{\int_{\Omega}}\left|\nabla u_{\varepsilon}\right|^2 dx}{\left(\displaystyle{\int_{\Omega}}\left|u_{\varepsilon}\right|^{2^{*}} dx\right)^{\frac{2}{2^{*}}}}\longrightarrow S,  \quad \text{as}\quad \varepsilon\rightarrow0,
\end{align}
where $S$ is the best Sobolev constant. They proved that, for $\varepsilon>0$ small, the solution $u_\varepsilon$ blows up at a critical point of Robin function $R(x)$ (see Definition B). Moreover, via the Pohozaev identity, they obtained the exact blowing-up rate of the solution $u_\varepsilon$ under $\|\cdot\|_{L^{\infty}(\Omega)}$. Later, for such a blow-up solution $u_\varepsilon$, Takahashi in \cite{Ta3} derived the precise asymptotic behavior of the first $(N+2)$-eigenvalues and the corresponding eigenfunctions of the following eigenvalue problem
\begin{equation}\label{eq1.4}
    \begin{cases}-\Delta v=\lambda\Big(\left(2^{*}-1\right) u_{\varepsilon}^{2^{*}-2}+\varepsilon\Big) v & \text { in } \Omega, \\
    v=0 & \text { on } \partial \Omega, \\ \|v\|_{L^{\infty}(\Omega)}=1. & \end{cases}
\end{equation}

An analogue problem to \eqref{eq1.2} is the following one
\begin{equation}\label{eq1.5}
    \begin{cases}
        -\Delta u=u^{2^{*}-1-\varepsilon},\quad u>0,\quad  &\text{in $\Omega$,}\\[1mm]
u=0,\quad &\text{on $\partial\Omega$.}
    \end{cases}
\end{equation}
Obviously, for $\varepsilon>0$, the nearly critical problem \eqref{eq1.5} admits a solution $u_\varepsilon$ by variational methods. But, due to the appearance of the critical Sobolev exponent when $\varepsilon=0$, \eqref{eq1.5} becomes delicate and only has trivial solutions provided $\Omega$ is star-shaped, which makes it interesting to study the asymptotic behavior of the solution $u_\varepsilon$ to \eqref{eq1.5} satisfying \eqref{eq1.3}. In particular, for $\varepsilon>0$ small, a similar phenomenon occurs as that of solutions to \eqref{eq1.2} with \eqref{eq1.3}, see \cite{H,Rey1,Rey2}. Later, under a non-degeneracy condition on the critical point $x_0\in\Omega$ of Robin function, Grossi et al in \cite{GP} stated  a connection between the eigenvalues of the associated linearized problem at $u_\varepsilon$ concentrating at $x_0$ and those of the Hessian matrix $D^{2}R(x_0)$, which implies that the Morse index of $u_\varepsilon$ can be computed as a by-product (see Definition A). Subsequently, such a type of results has been extended to multi-bubble solutions of \eqref{eq1.5} in \cite{CKL}.

Recently, much attention has been focused on the study of the non-local equation \eqref{eq 1.1}, which is closely related to the nonlinear Choquard equation
\begin{align}\label{eq1.6}
  -\Delta u+V(x) u=\left(\displaystyle{\int_{\R^N}}\frac{|u(y)|^{p}}{|x-y|^{\mu}}dy\right)|u|^{p-2} u, \quad\text {in $\R^N$.}
\end{align}
For the case that $N=3,\,p=2,\,\mu=1$ and $V(x)\equiv1$, \eqref{eq1.6} reduces into
\begin{align}\label{eq1.7}
    -\Delta u+u=\left(\displaystyle{\int_{\R^3}}\frac{u^{2}(y)}{|x-y|}dy\right) u, \quad\text {in $\R^3$,}
\end{align}
which arises in various fields of mathematical physics, such as the description of the quantum theory of a polaron at rest (see \cite{PS}). By variational methods, the existence and uniqueness of ground states to \eqref{eq1.7} were obtained by Lieb \cite{L1}. Later, Wei et al in \cite{WW} showed the non-degeneracy of ground states to \eqref{eq1.7}. Regarding the problem \eqref{eq1.6} with different assumptions on the potential $V(x)$ and the exponent $p$, several results have been obtained in the past decades, we may refer the interested readers to \cite{A,ANY,MV1,MV2,DY,GHPS} and the reference therein. In particular, Moroz et al in \cite{MV1} showed that, for $N\geq3$ and $\mu\in(0,N)$, the lower critical exponent $2_{*}^{\mu}:=\frac{2N-\mu}{N}$ and the upper critical exponent $2^{*}_{\mu}:=\frac{2N-\mu}{N-2}$ play a crucial role in the existence and nonexistence of solutions to \eqref{eq1.6}.

Concerning the upper critical case, the first result was obtained in \cite{GY}, where the authors utilized the variational methods to clarify the existence, multiplicity and nonexistence of solutions to \eqref{eq 1.1}. We would like to point out that their results perfectly agree with the local problem \eqref{eq1.2}.
Very recently, great attention has been paid to studying \eqref{eq 1.1} as the parameter $\varepsilon$ goes to zero. Similar to the arguments about the local problems \eqref{eq1.2} or \eqref{eq1.5}, the stability assumptions on the critical points of Robin function are essential for the existence and local uniqueness of the solutions to \eqref{eq 1.1}. In \cite{YYZ}, via the reduction arguments, Yang et al constructed a concentrated solution $u_\varepsilon$ of \eqref{eq 1.1} blowing-up at $x_0$ with $x_0\in\Omega$  a non-degenerate critical point of $R(x)$. Later, Squassina et al in \cite{SYZ} utilized the local Pohozaev identity and blow-up analysis to derive the local uniqueness of the blow-up solutions.

Conversely, in \cite{YZ}, Yang et al studied the asymptotic behavior of solutions
 $u_\varepsilon$ of \eqref{eq 1.1} satisfying
\begin{align}\label{eq 1.2}
    \frac{\displaystyle{\int_{\Omega}}{|\nabla u_{\varepsilon}|^{2}}dx}{\left(\displaystyle{\displaystyle{\int_{\Omega}}}\displaystyle{\displaystyle{\int_{\Omega}}}\frac{|u_{\varepsilon}(x)|^{2^{*}_{\mu}}|u_{\varepsilon}(y)|^{2^{*}_{\mu}}}{|x-y|^{\mu}}dxdy\right)^{\frac{1}{2^{*}_{\mu}}}}\longrightarrow S_{HL},\quad \text{as}\quad \varepsilon\rightarrow0,
\end{align}
where $S_{HL}$ is the best constant defined by
\begin{align}\label{eq 1.3}
S_{HL}:=\underset{u\in \mathcal{D}^{1,2}(\R^{N})\backslash\{0\}}{\inf} \frac{\displaystyle{\int_{\R^{N}}}{|\nabla u|^{2}}dx}{\left(\displaystyle{\displaystyle{\int_{\R^{N}}}}\displaystyle{\displaystyle{\int_{\R^{N}}}}\frac{|u(x)|^{2^{*}_{\mu}}|u(y)|^{2^{*}_{\mu}}}{|x-y|^{\mu}}dxdy\right)^{\frac{1}{2^{*}_{\mu}}}}.
\end{align}
In this case, the solutions $u_\varepsilon$ are called blow-up solutions since it is possible to verify that
$$\left|\nabla u_{\varepsilon}\right|^{2}\rightharpoonup S_{HL}^{\frac{2N-\mu}{N-\mu+2}}\delta_{x_0},\quad\text{weakly in the sense of measures,}$$
 where $x_0$ is the blow-up point of $u_\varepsilon$ and $\delta_{x_0}$ is the Dirac function centered at $x_0$, and they showed that the point $x_0\in\Omega$ is the critical point of Robin function $R(x)$. Note that the results of \eqref{eq 1.1} we mentioned above in \cite{YYZ,SYZ,YZ} can be regarded as the counterpart of the local case \eqref{eq1.2} in \cite{Rey2,G} respectively.

Based on the work \cite{YZ}, the main purpose of this paper is to consider the qualitative analysis the blow-up solutions of
\eqref{eq 1.1} verifying \eqref{eq 1.2}.
Precisely,  we will study the asymptotic behavior of the eigenvalues and the
corresponding eigenfunctions to the associated linearized problem at $u_\varepsilon$, and then a strong correspondence between the Morse index of blow-up solutions and the negative eigenvalues of the Hessian matrix of Robin function at the concentration point will be performed. We would like to point out that Morse index has a wide range of applications in mathematics, such as the derivation of symmetries, uniqueness or bifurcation results of solutions, see \cite{BLP,DP,GGPS,PW,DGIP} and the references therein. To state the
results in this direction, we first introduce some notations and definitions as follows.

\vskip 0.2cm

\noindent\textbf{Definition A.}\label{def A}
\emph{The Morse and augmented
Morse index of a solution $u_\varepsilon$ to problem \eqref{eq 1.1} can be defined as
\begin{equation*}
\begin{cases}
m\left(u_\varepsilon\right):=\sharp\left\{i \in \mathbb{N}:\lambda_{i, \varepsilon}<1\right\}, \\[1mm]
m_0\left(u_\varepsilon\right):=\sharp\left\{i \in \mathbb{N}: \lambda_{i, \varepsilon} \leq 1\right\},
\end{cases}
\end{equation*}
where $\lambda_{1,\varepsilon}<\lambda_{2,\varepsilon}\leq\lambda_{3,\varepsilon}\leq\cdots$ is the sequence of eigenvalues for the linearized problem
\begin{equation}\label{eq 1.12}
\begin{cases}
-\Delta v_{i,\varepsilon}(x)=\lambda_{i,\varepsilon}\Bigg\{(2^{*}_\mu-1)u^{2^{*}_\mu-2}_{\varepsilon}(x)v_{i,\varepsilon}(x)\left(\displaystyle{\displaystyle{\int_{\Omega}}}\frac{u_{\varepsilon}^{2^{*}_\mu}(y)}{|x-y|^{\mu}}dy\right)+\varepsilon v_{i,\varepsilon}(x)\\
\quad\,\,\,\quad\,\,\,\,\,\quad\,\,\,\,\,\,\,\,\,\,\quad\,+2^{*}_\mu u^{2^{*}_\mu-1}_{\varepsilon}(x)\left(\displaystyle{\displaystyle{\int_{\Omega}}}\frac{u^{2^{*}_\mu-1}_{\varepsilon}(y)v_{i,\varepsilon}(y)}{|x-y|^{\mu}}dy\right)\Bigg\},\quad&\text{in $\Omega$,}\\
v_{i,\varepsilon}(x)=0,\quad &\text{on $\partial\Omega$,}\\[2mm]
\|v_{i,\varepsilon}\|_{L^{\infty}(\Omega)}=1,
\end{cases}
\end{equation}
here $v_{i,\varepsilon}$ is the $i$-th eigenfunction corresponding to $\lambda_{i,\varepsilon}$.
  For any function $f\in C^{2}(\Omega)$, we denote by $m(x)$ and $m_{0}(x)$ the Morse and augmented Morse index of $x$, as a critical point of $f$, that is,
 \begin{equation*}
 \begin{cases}
  m(x):=\sharp\left\{l\in\left\{1,2,\cdots,N\right\} :\lambda_{l}<0\right\}, \\[1mm]
 m_{0}(x):=\sharp\left\{l\in\left\{1,2,\cdots,N\right\} :\lambda_{l}\leq 0\right\},
 \end{cases}
 \end{equation*}
 where $\lambda_{1}\leq\lambda_{2}\leq\cdots\leq\lambda_{N}$ are the eigenvalues of the Hessian matrix $D^{2}f(x)$.
}

\vskip 0.2cm

Note that when $\lambda_{i,\varepsilon}=1$ in \eqref{eq 1.12}, the space generated by the solutions $\{v_{i,\varepsilon}\}$ is the kernel of the
linearized operator at $u_\varepsilon$. Recalling the definition of non-degeneracy of solutions (see \cite{G1}), we can conclude that if each
eigenvalue of \eqref{eq 1.12} is not equal to 1, then the solution $u_{\varepsilon}$ of \eqref{eq 1.1} is non-degenerate, especially $m(u_\varepsilon)=m_{0}(u_\varepsilon)$ in this case. This method greatly reduces
the tedious calculations compared with the utilization of blow-up analysis and local Pohozaev identities to draw the same conclusion. In addition, the solutions of \eqref{eq 1.12} verify the orthogonal relation, that is, for any $i\neq j$,
 \begin{equation}\label{eqq1-11}
     \begin{split}
     0= &(2^{*}_\mu-1)\displaystyle{\int_{\Omega}}u^{2^{*}_\mu-2}_{\varepsilon}v_{i,\varepsilon}v_{j,\varepsilon}\left(\displaystyle{\int_{\Omega}}\frac{u_{\varepsilon}^{2^{*}_\mu}(y)}{|x-y|^{\mu}}dy\right)\,dx+\varepsilon\displaystyle{\int_{\Omega}}v_{i,\varepsilon}v_{j,\varepsilon}dx\\
    &+2^{*}_\mu\displaystyle{\int_{\Omega}}u^{2^{*}_\mu-1}_{\varepsilon}v_{j,\varepsilon}\left(\displaystyle{\int_{\Omega}}\frac{u^{2^{*}_\mu-1}_{\varepsilon}(y)v_{i,\varepsilon}(y)}{|x-y|^{\mu}}dy\right)\,dx .
     \end{split}
 \end{equation}

It is known that the computation of the Morse index gives qualitative information on solutions. Except some results, we have already mentioned, concerning the Morse index of blow-up solutions to \eqref{eq1.2} or \eqref{eq1.5} in \cite{GP,Ta3,CKL}, we present other results in this subject. In \cite{GG}, Gladiali et al studied the two-dimensional Gel'fand problem
\begin{equation}\label{1-2}
    \begin{cases}
        -\Delta u=\varepsilon e^{u},\quad&\text{in $\Omega$,}\\[1mm]
        u=0,\quad&\text{on $\partial\Omega$,}
    \end{cases}
\end{equation}
where $\Omega\subset \R^2$ is a smooth bounded domain and $\varepsilon>0$ is a small parameter. For the given one-bubble solution $u_\varepsilon$ satisfying $$\varepsilon\int_{\Omega}e^{u_\varepsilon}\rightarrow 8\pi~~~as~~~\varepsilon\rightarrow0,$$ they obtained the asymptotic behavior of
eigenvalues and eigenfunctions to the linearized Gel'fand problem and the Morse
index of $u_\varepsilon$ as a product. Later, such results have been extended to multi-bubble solutions of \eqref{1-2} in
\cite{GGOS}, which require a careful analysis of the decay estimates of each bubble. Different from the problems posed in bounded domains $\Omega$, Grossi et al in \cite{GS} considered single-peak solutions of the following nonlinear Schr\"odinger equation
\begin{equation}\label{1-1}
    \begin{cases}
        -\varepsilon^{2}\Delta u+V(x)u=u^{p},\quad u>0,\quad\text{in $\R^N$},\\[2mm]
        u\in H^{1}(\R^N),
    \end{cases}
\end{equation}
where $\varepsilon>0$, $p\in(1,\frac{N+2}{N-2})$ if $N\geq3$, $p>1$ if $N=2$ and $V(x)\in C^{\infty}(\R^N)$ is a bounded positive admissible potential. Under the assumption that the point $P$ is the critical point of $V(x)$, the Morse index of single-peak solutions concentrating at $P$ can be computed, even in the case that $P$ is a degenerate critical point of $V(x)$. Very recently, Luo et al in \cite{LPP} derived the Morse index of multi-peak solutions to \eqref{1-1} when $V(x)$ has non-isolated critical points with different degenerate rates along different directions.

Before presenting our main results, for the convenience of readers, we recall some definitions and known results of \eqref{eq 1.1} needed in this paper.
\vskip 0.2cm
\noindent\textbf{Definition B.}\label{def B}
\emph{Let $G(x,y)$ be the Green function of $-\Delta$ in $H_{0}^{1}(\Omega)$ satisfying
 \begin{equation*}
     \begin{cases}
       -\Delta  G(\cdot,y)=\delta_{y},\quad&\text{in $\Omega$},\\[2mm]
       G(\cdot,y)=0,\quad&\text{on $\partial\Omega$},
     \end{cases}
 \end{equation*}
and $H(x,y)$ be its regular part, i.e.,
$$H(x,y)=\frac{1}{(N-2)\sigma_{N}|x-y|^{N-2}}-G(x,y),\quad \forall\, x,y\in\Omega,$$
where $\sigma_{N}$ is the measure of the unit sphere in $\R^N$. Then, we denote $R(x):=H(x,x)>0$ as the Robin function.}

 \vskip 0.2cm
Getting back to problem \eqref{eq 1.1}, 
as for the constant $S_{HL}$ given by \eqref{eq 1.3}, we have the following conclusions.
\smallskip

\noindent\textbf{Theorem A.(c.f.\cite{GY, GLMY})}\label{thm A} \emph{ $(i)$ The constant $S_{HL}$ defined in \eqref{eq 1.3} is achieved if and only if  by $
\bar{U}_{\xi,\tau}(x)$, where
\begin{equation}\label{eq 1.4}
    \bar{U}_{\xi,\tau}(x)=S^{\frac{(N-\mu)(2-N)}{4(N-\mu+2)}}\left(C_{N,\mu}\right)^{\frac{2-N}{2(N-\mu+2)}}\left(N(N-2)\right)^{\frac{N-2}{4}}U_{\xi,\tau}(x):=\bar{C}\,U_{\xi,\tau}(x).
\end{equation}
Here, we denote $U_{\xi,\tau}(x)$ as the Talenti bubble given by
$$U_{\xi,\tau}(x)=\frac{\tau^{\frac{N-2}{2}}}{(1+\tau^{2}|x-\xi|^{2})^{\frac{N-2}{2}}},\quad \tau\in \R^{+},\,x,\,\xi \in \R^{N},$$
which is the unique minimizer for the best Sobolev constant $S$ and satisfies
\begin{align*}
-\Delta u=N(N-2)u^{2^{*}-1},\quad \text{in $\R^{N}$}.
\end{align*}
Furthermore,
$$S_{HL}=\frac{S}{\left(C_{N,\mu}\right)^{\frac{1}{2^{*}_{\mu}}}}.$$
\\
$(ii)$ The function $\bar{U}_{\xi,\tau}(x)$ is the unique family of positive solutions of
\begin{equation}\label{eq 1.5}
-\Delta u(x)=\left(\displaystyle{\int_{\R^{N}}}\frac{u^{2^{*}_{\mu}}(y)}{|x-y|^{\mu}}dy\right){u}^{2^{*}_{\mu}-1}(x),\quad  \text{in}\ \R^N.
\end{equation}
Moreover, from \eqref{eq 1.4}, a direct calculation yields that $U_{\xi,\tau}(x)$ satisfies
\begin{equation}\label{eq 1.6}
-\Delta u={A}_{HL}\left(\displaystyle{\int_{\R^{N}}}\frac{u^{2^{*}_{\mu}}(y)}{|x-y|^{\mu}}dy\right)u^{2^{*}_{\mu}-1}(x),\quad  \text{in}\ \R^N,\\
\end{equation}
where ${A}_{HL}:=\left(N(N-2)\right)^{\frac{N-\mu+2}{2}}S^{\frac{\mu-N}{2}}{\left(C_{N,\mu}\right)}^{-1}$.
\\[1mm]
$(iii)$ As a consequence of the moving sphere method, the Talenti
bubbles satisfy
\begin{align}\label{eq 1.7}
\displaystyle{\int_{\R^N}}\frac{U_{\xi,\tau}^{2^{*}_{\mu}}(y)}{|x-y|^{\mu}}dy=\frac{N(N-2)}{{A}_{HL}}U_{\xi,\tau}^{2^{*}-2^{*}_{\mu}}(x).
\end{align}
}

\vskip 0.2cm
In the present paper, the non-degeneracy of the solution $\bar{U}_{\xi,\tau}$ to \eqref{eq 1.5} plays a crucial role. We summarize the non-degeneracy results as follows.
\smallskip

\noindent\textbf{Theorem B.(c.f.\cite{DY,GMYZ,GWY})}\label{thm 1.C}
  \emph{ Assume that $N,\,\mu$ satisfy the assumptions:
    \begin{equation*}
       (*) \begin{cases}
            \mu\in (0,N), \quad \mu\,\text{is sufficiently close to $0$ or $N$},\,\quad&\text{if $N=3$ or $4$},\\[1mm]
            \mu\in (0,4), \quad \mu\,\text{is sufficiently close to $0$},\,\quad&\text{if $N\geq 5$ but $N \neq 6$},\\[1mm]
            \mu\in (0,4], \quad \mu\,\text{is sufficiently close to $0$, $4$ or $=4$},\,\quad&\text{if $N=6$}.
        \end{cases}
    \end{equation*}
Then the linearized operator of \eqref{eq 1.5} at $\bar{U}_{\xi,\tau}$ defined by
\begin{equation*}
       L\phi=-\Delta\phi-(2^{*}_\mu-1)\bar{U}_{\xi,\tau}^{2^{*}_\mu-2}\phi\left(\displaystyle{\int_{\R^N}}\frac{\bar{U}_{\xi,\tau}^{2^{*}_\mu}(y)}{|x-y|^{\mu}}dy\right)-2^{*}_{\mu} \bar{U}_{\xi,\tau}^{2^{*}_\mu-1}\left(\displaystyle{\int_{\R^N}}\frac{\bar{U}_{\xi,\tau}^{2^{*}_{\mu}-1}(y)\phi(y)}{|x-y|^{\mu}}dy\right)
\end{equation*}
 only admits solutions in $\mathcal{D}^{1,2}(\R^N)$ of the form
$$\phi=\bar{a} D_{\tau}\bar{U}_{\xi,\tau}+\vec{b}\cdot\nabla\bar{U}_{\xi,\tau},$$
where $\bar{a}\in\R$ and $\vec{b}\in\R^N$.
}
\begin{Rem} Note that in \cite{LLTX}, Li et al generalized the above non-degeneracy results of $\bar{U}_{\xi,\tau}$ to the whole range $\mu\in(0,N)$. But, our work is based on the results in \cite{YZ}, which requires the exponent $\mu$ satisfies the assumptions of Theorem B.
\end{Rem}

\vskip 0.2cm
 Now, we state some conclusions about the asymptotic behavior of the solutions $u_\varepsilon$ to problem \eqref{eq 1.1} verifying \eqref{eq 1.2}.
 \smallskip

\noindent\textbf{Theorem C.(c.f.\cite{YZ})}\label{thm 1.D}\emph{
 Assume that $N\geq 4$ and $\mu\in(0,4]$. Let $u_{\varepsilon}$ be a solution of \eqref{eq 1.1} satisfying \eqref{eq 1.2}, then for $\varepsilon>0$ sufficiently small,
$$u_\varepsilon=\xi_\varepsilon PU_{x_\varepsilon,\tau_\varepsilon}+\zeta_{\varepsilon},$$
where $x_\varepsilon\rightarrow x_0$, $\|\zeta_\varepsilon\|_{H_{0}^{1}(\Omega)}\rightarrow 0$ and $\xi_\varepsilon\rightarrow \bar{C}$ as $\varepsilon\rightarrow 0$, with $\bar{C}$ a fixed constant given by \eqref{eq 1.4}. Moreover, $x_0\in\Omega$ is a critical point of the Robin function $R(x)$.
}
\begin{Rem}\label{Rem1.2} Under the assumptions of Theorem C, we can deduce that $u_\varepsilon$ is uniformly bounded in $H_{0}^{1}(\Omega)$, that is,
\begin{equation}\label{eq 1.10}
\displaystyle{\int_{\Omega}}|\nabla u_{\varepsilon}|^{2}dx\leq S_{HL}^{\frac{2N-\mu}{N+2-\mu}}.
\end{equation}
On the other hand, since the solution $u_\varepsilon$ must blow up at the point $x_0$ provided $\varepsilon>0$ small, i.e.,
\begin{equation}\label{eq 1.8}
    \|u_{\varepsilon}\|_{L^{\infty}(\Omega)}\longrightarrow +\infty,\quad\text{as $\varepsilon\rightarrow0$},
  \end{equation}
we can assume that there exists $\tau_{\varepsilon}\in \R^{+}$ such that $\|u_\varepsilon\|_{L^{\infty}(\Omega)}=u_\varepsilon(x_\varepsilon)=\tau_{\varepsilon}^{\frac{N-2}{2}}\rightarrow+\infty$ as $\varepsilon\rightarrow 0$. Then, a family of rescaled functions can be given by$$\tilde{u}_\varepsilon(x)=\tau_{\varepsilon}^{-\frac{N-2}{2}}u_\varepsilon(\tau_{\varepsilon}^{-1}x+x_{\varepsilon}),\quad \forall\,x\in\Omega_{\varepsilon}:=\left\{x\in\R^N:\tau_{\varepsilon}^{-1}x+x_{\varepsilon}\in \Omega \right\}.$$
And a direct calculation yields that
\begin{equation}\label{eq 1.11}
\begin{cases}
-\Delta \tilde{u}_\varepsilon=\left(\displaystyle{\displaystyle{\int_{\Omega_{\varepsilon}}}}\frac{\tilde{u}_\varepsilon^{2^{*}_{\mu}}(y)}{|x-y|^{\mu}}dy\right)\tilde{u}_\varepsilon^{2^{*}_{\mu}-1}+\frac{\varepsilon}{\tau_{\varepsilon}^{2}} \tilde{u}_\varepsilon,\quad\tilde{u}_\varepsilon>0,\quad  &\text{in}\ \Omega_{\varepsilon},\\
\tilde{u}_\varepsilon=0,\quad &\text{on}\ \partial\Omega_{\varepsilon},\\
\|\tilde{u}_{\varepsilon}\|_{L^{\infty}(\Omega_{\varepsilon})}=1.
\end{cases}
\end{equation}
Using Arzela-Ascoli Theorem, we obtain that for $\varepsilon>0$ small enough, there exists a subsequence $\{\tilde{u}_{\varepsilon}\}$ converging to some $u\not\equiv 0$ uniformly on any compact set in $\R^N$, and the limit function $u\in\mathcal{D}^{1,2}(\R^N)$ satisfies \eqref{eq 1.5}. Hence, from {Theorem A-$(ii)$}, we conclude that \begin{equation}\label{eq 1.9}
   \tilde{u}_{\varepsilon}(x)\longrightarrow \bar{U}_{0,1}(x),\quad\text{in $C^{1}_{loc}(\R^N)$.}
\end{equation}
\end{Rem}
\vskip 0.2cm
 Let us denote by $v_{i,\varepsilon}$ the eigenfunction of \eqref{eq 1.12} corresponding to $\lambda_{i,\varepsilon}$ and by $\tilde{v}_{i,\varepsilon}$ the rescaled eigenfunction defined as
$$\tilde{v}_{i,\varepsilon}(x)=v_{i,\varepsilon}(\tau_{\varepsilon}^{-1}x+x_{\varepsilon}),\quad \forall\, x\in \Omega_{\varepsilon}:=\left\{x\in\R^N:\tau_{\varepsilon}^{-1}x+x_{\varepsilon}\in \Omega \right\}.$$
Then, we have
\begin{equation}\label{eq 1.13}
\begin{cases}
    \,-\Delta \tilde{v}_{i,\varepsilon}=\lambda_{i,\varepsilon}\Bigg\{(2^{*}_\mu-1)\tilde{u}_{\varepsilon}^{2^{*}_\mu-2}\tilde{v}_{i,\varepsilon}\left(\displaystyle{\displaystyle{\int_{\Omega_{\varepsilon}}}}\frac{\tilde{u}_{\varepsilon}^{2^{*}_\mu}(y)}{|x-y|^{\mu}}dy\right)+\frac{\varepsilon}{\tau_{\varepsilon}^{2}} \tilde{v}_{i,\varepsilon}\\
\quad\,\,\,\quad\,\,\,\,\,\quad\,\,\,\,\,\,\,\,\,\,\quad\,+2^{*}_{\mu} \tilde{u}_{\varepsilon}^{2^{*}_\mu-1}\left(\displaystyle{\displaystyle{\int_{\Omega_{\varepsilon}}}}\frac{\tilde{u}^{2^{*}_\mu-1}_{\varepsilon}(y)\tilde{v}_{i,\varepsilon}(y)}{|x-y|^{\mu}}dy\right)\Bigg\},\quad&\text{in}\,\Omega_{\varepsilon},\\
\,\tilde{v}_{i,\varepsilon}=0,\quad &\text{on}\ \partial\Omega_{\varepsilon},\\
\,\|\tilde{v}_{i,\varepsilon}\|_{L^{\infty}(\Omega_\varepsilon)}=1.
\end{cases}
\end{equation}
\vskip 0.2cm
Now, we state precisely our results.
\begin{theorem}\label{th1.1}
   Assume $N\geq5$ and $\mu\in(0,4)$. As $\varepsilon\rightarrow0$, we have
  \begin{align}
\lambda_{1,\varepsilon}&\rightarrow\frac{1}{2\cdot2^{*}_\mu-1},\label{eq 1.16}\\
\tilde{v}_{1,\varepsilon}&\rightarrow U_{0,1},\quad \text{in $C^{1}_{loc}\left(\R^N\right)$,} \label{eq 1.17}\\
\|u_\varepsilon\|_{L^{\infty}(\Omega)}^{2}v_{1,\varepsilon}(x)&\rightarrow\frac{\sigma_{N}}{N}G(x,x_0),\quad \text{in $C^{1}_{loc}\left(\bar{\Omega}\backslash\{x_0\}\right)$.}\label{eq 1.18}
  \end{align}
\end{theorem}
\vskip 0.2cm
The next Theorem states some qualitative properties of the eigenfunctions $v_{i,\varepsilon}$ to \eqref{eq 1.12} and its rescaled function $\tilde{v}_{i,\varepsilon}$ to \eqref{eq 1.13}, for $i=2,\cdots,N+1$. In particular, a strong correspondence between the eigenvalues $\lambda_{i,\varepsilon}$ and the eigenvalues of the Hessian matrix $D^{2}R(x)$ at $x_0$ are performed.
\begin{theorem}\label{th1.2} Assume $N\geq6$  and $\mu\in(0,4)$. For $i=2,3,\cdots,N+1$, we obtain that  as $\varepsilon\rightarrow 0$,
\begin{align}
    \tilde{v}_{i,\varepsilon}(x)&\rightarrow\sum_{k=1}^{N}
    \frac{a_{i,k}\,x_k}{\big(1+|x|^{2}\big)^{\frac{N}{2}}},\quad \text{in $C^{1}_{loc}\left(\R^N\right)$,}\label{eq 1.21}\\
\|u_{\varepsilon}\|_{L^{\infty}(\Omega)}^{2+\frac{2}{N-2}}v_{i,\varepsilon}(x)&\rightarrow \frac{\sigma_{N}}{N(N-2)}\sum_{k=1}^{N}a_{i,k}\,\frac{\partial\, G(x,x_{0})}{\partial y_{k}},\quad \text{in $C^{1}_{loc}\left(\bar{\Omega}\backslash\{x_0\}\right)$,}\label{eq 1.22}
\end{align}
for some $\vec{a}_{i}=(a_{i,1},a_{i,2},\cdots,a_{i,N})\neq\vec{0}$. In addition,
\begin{align}\label{eq 1.20}
\|u_{\varepsilon}\|_{L^{\infty}(\Omega)}^{\frac{2N}{N-2}}\left(\lambda_{i,\varepsilon}-1\right)\rightarrow M\gamma_{i-1},\quad\text{ as $\varepsilon\rightarrow 0$,}
\end{align}
where $\gamma_{1}\leq\gamma_{2}\leq\cdots\leq\gamma_{N}$ are eigenvalues of the Hessian matrix $D^{2}R(x_0)$ and
$$M:=\frac{(N-2)\,\sigma_{N}^{2}}{2N(N+2)\left(\displaystyle{\displaystyle{\int_{\R^N}}}U_{0,1}^{2^{*}-2}(x)|\nabla U_{0,1}(x)|^{2}dx\right)}>0.$$
Furthermore, $\vec{a}_i$ is an eigenvector of $D^{2}R(x_0)$ corresponding to $\gamma_{i-1}$ and $\vec{a}_i$ is perpendicular to $\vec{a}_j$ in $\R^N$ if $i\neq j$.
\end{theorem}
\vskip 0.2cm
Finally, we present some estimates on the $(N+2)$-th eigenvalue of \eqref{eq 1.12} and the asymptotic behavior of the corresponding eigenfunction.
\begin{theorem}\label{th1.3} Assume $N\geq6$  and $\mu\in(0,4)$. As $\varepsilon\rightarrow0$, we have
\begin{align}
\tilde{v}_{N+2,\varepsilon}(x)&\rightarrow b_{N+2}
    \frac{1-|x|^{2}}{\big(1+|x|^{2}\big)^{\frac{N}{2}}},\quad \text{in $C^{1}_{loc}\left(\R^N\right)$,}\label{eq 1.24}
\end{align}
for some $b_{N+2}\neq 0$. Moreover, for $\varepsilon>0$ small,
\begin{align}
\|u_{\varepsilon}\|_{L^{\infty}(\Omega)}^{2}\left(\lambda_{N+2,\varepsilon}-1\right)\rightarrow \kappa,\label{eq 1.23}
\end{align}
where
$$\kappa:=\frac{2\sigma_{N}^{2}(N-4)R(x_0)}{(N+2)N^{2}\left(\displaystyle{\displaystyle{\int_{\R^N}}}\frac{(1-|x|^{2})^{2}}{(1+x^{2})^{N+2}}dx\right)}=\left(N-2\right)\left(N-4\right)MR(x_0)>0.$$
\end{theorem}
\vskip 0.2cm
As a consequence of all above results, we obtain the following result.
\begin{theorem}\label{cor1}
Let $N\geq 6$, $\mu\in(0,4)$, $u_\varepsilon$ be a solution to \eqref{eq 1.1} with \eqref{eq 1.2} and $x_0$ be the limit point of the maximum point $x_\varepsilon$ of $u_\varepsilon$. Then if $x_0\in\Omega$ is a non-degenerate critical point of the Robin function $R(x)$, we obtain that the Morse index of $u_\varepsilon$ is equal to $m(x_0)+1$, where $m(x_0)$ denotes the Morse index of the Hessian matrix $D^{2}R(x_0)$.
\end{theorem}

\begin{remark}
We would like to point that there are several crucial difficulties in our main proofs.
\vskip 0.1cm

\noindent\textup{($i$)} In \cite{YZ}, the authors investigated the asymptotic behavior of solutions $u_\varepsilon$ to \eqref{eq 1.1} verifying \eqref{eq 1.2} and gave the explicit expansion of $u_\varepsilon$, see Theorem C. Yet, some essential estimates of $u_\varepsilon$ are unknown, such as the exact blowing-up rate of $u_\varepsilon$ under $\|\cdot\|_{L^{\infty}(\Omega)}$. We will fill in this gap by using blow-up analysis in \cite{H}. In this process, some new arguments about the non-local term in \eqref{eq 1.1} will be performed and we can find the detailed results in Lemma \ref{lema 2.4}.

\vskip 0.1cm

\noindent\textup{($ii$)} Since there exist two non-local terms in \eqref{eq 1.12}, our problem is quite different from \eqref{eq1.4}. This leads to many difficulties while analyzing the asymptotic behavior of the eigenfunctions $v_{i,\varepsilon}$ to \eqref{eq 1.12}, which require more additional precise estimates. To overcome this difficulty, we will use some ideas in \cite{LTX,SYZ} and develop some new Pohozaev identities.

\vskip 0.1cm

\noindent\textup{($iii$)}
As aforementioned, Grossi et al in \cite{GP} and Takahashi in \cite{Ta3}  studied the qualitative behavior of blow-up solutions to the local problems \eqref{eq1.5} and \eqref{eq1.2} respectively, our results can be seen as an extension to the non-local case.
However, we can not utilize the same discussions as in \cite{GP,Ta3} directly due to the Hartree-type nonlinearity in \eqref{eq 1.1}, which may bring obstacles to the estimates of the eigenpairs $\{(\lambda_{i,\varepsilon},v_{i,\varepsilon})\}$ to \eqref{eq 1.12}.

\vskip 0.1cm

 For example, by selecting appropriate test functions defined in \eqref{eq 4.2} and using the min-max principle, we derive the inequality \eqref{eq4.18} concerning the eigenvalue $\lambda_{i,\varepsilon}$, for $i=2,\cdots,N+1$. Then, some new estimates should be introduced to evaluate each term in the numerator and the denominator, especially the terms $D_{2,\varepsilon}$ and $N_{3,\varepsilon}$. We would like to point out that the symmetry property of double integrals, the application of Hardy-Littlewood-Sobolev inequality and the decay estimates of solutions outside the concentrated point play a crucial role.
Similar computational techniques can be used in the estimates of $\lambda_{N+2,\varepsilon}$ in \eqref{eq5-2}, where we should overcome the difficulties from the new choice of test functions given by \eqref{eq 4.3}.

\vskip 0.1cm

\noindent\textup{($iv$)} The restriction $N\geq6$ is needed to deduce the asymptotic behavior of eigenfunctions $v_{i,\varepsilon}\,(i=2,\cdots,N+1)$ stated in the proof of \eqref{eq 1.22}. Indeed, the key steps in proving \eqref{eq 1.22} are \eqref{4-1}, \eqref{4-2} and \eqref{I-3}, which indicates that $N\geq6$ is an essential condition.

\end{remark}

Our paper will be organized as follows: In Section 2, we state the precise estimates of solutions $u_\varepsilon$ to \eqref{eq 1.1} with \eqref{eq 1.2}, eigenfunctions $v_{i,\varepsilon}$ to \eqref{eq 1.12} and its rescaled function $\tilde{v}_{i,\varepsilon}$ to \eqref{eq 1.13}, which will be used throughout the paper. Concerning the properties of the eigenpairs $\{(\lambda_{i,\varepsilon},v_{i,\varepsilon})\}$, since the information of the previous pair is crucial for the study of the next pair, we first consider the case $i=1$ in Section 3, then $2\leq i\leq N+1$ in Section 4, and finally $i=N+2$ in Section 5. In each case, the procedure consists of the estimates of the eigenvalue $\lambda_{i,\varepsilon}$, the limit characterization of the eigenfunction $v_{i,\varepsilon}$ and its rescaled function $\tilde{v}_{i,\varepsilon}$. Thus, by analyzing the asymptotic behavior of eigenpairs to \eqref{eq 1.12}, we complete the proof of Theorem \ref{cor1}. In addition, we recall some well-known results and give some computational techniques in appendixes, which are useful in our proofs.

In this paper, we denote $\|\cdot\|_{p}$ $\left(\,p\in[1,\infty]\,\right)$ as the usual $L^{p}(\Omega)$ norm and ``$\rightarrow,\rightharpoonup$'' as the strong and weak convergence respectively in the corresponding space. For simplicity, we will omit the constants $\bar{C}$ in \eqref{eq 1.4}, ${A}_{HL}$ in \eqref{eq 1.6} and $\frac{N(N-2)}{{A}_{HL}}$ in \eqref{eq 1.7}. Our results are still correct in terms of the difference of a coefficient. Moreover, we will use $C$, $C_{i}$, $i\in\mathbb{N}$, to denote various positive constants from line to line.

\noindent
\section{\label{Pre}Preliminaries}
\setcounter{equation}{0}
\vspace{0.2cm}
 In this section, we present some results that are essential for the proofs in the subsequent sections. Firstly, we give some estimates of blow-up solutions $u_\varepsilon$ to \eqref{eq 1.1} with \eqref{eq 1.2}, especially the blowing-up rate of $u_\varepsilon$ under $\|\cdot\|_{\infty}$.
\begin{lemma}\label{lema 2.4}
    Let $u_\varepsilon$ be a solution of \eqref{eq 1.1} with \eqref{eq 1.2}. Then the following estimates hold true
      \begin{equation}\label{eq 2.2}
        u_{\varepsilon}(x)\leq CU_{x_{\varepsilon,\tau_{\varepsilon}}}(x),\quad  x\in \Omega,
    \end{equation}
\begin{align}\label{eq 2.3}
   \|u_\varepsilon\|_{\infty} u_\varepsilon(x)\rightarrow \frac{\sigma_{N}}{N}G(x,x_0),\quad \text{in}\,\,C^{1}_{loc}\left(\bar{\Omega}\backslash\{x_0\}\right)\,\,\text{as}\,\,\varepsilon\rightarrow 0.
\end{align}
\begin{align}\label{eq 2.4}
\underset{\varepsilon\rightarrow0}{\lim}\,\varepsilon\|u_\varepsilon\|_{\infty}^{\frac{2(N-4)}{N-2}}=\frac{\sigma_N}{2a_N}\frac{N-2}{N^2}R(x_0),\quad\text{for $N\geq5$,}
\end{align}
where $\sigma_{N}$ is the measure of the unit sphere in $\R^N$ and $a_N$ is given by
$$a_N=\int_{0}^{\infty}\frac{r^{N-1}
}{(1+r^{2})^{N-2}}\,dr=\frac{\Gamma\left(\frac{N}{2}\right)\Gamma\left(\frac{N}{2}-2\right)}{2\Gamma\left(N-2\right)}.$$
\begin{proof}
    Note that \eqref{eq 2.2} is equivalent to
\begin{equation}\label{eq 2.5}
        \tilde{u}_{\varepsilon}(x)\leq C U_{0,1}(x),\quad x\in\Omega_{\varepsilon}.
    \end{equation}
Denote the Kelvin transform of $\tilde{u}_\varepsilon$ as
$$w_{\varepsilon}(x):=\frac{1}{|x|^{N-2}}\tilde{u}_\varepsilon\left(\frac{x}{|x|^{2}}\right),\quad x\in \Omega_{\varepsilon}^{*},$$
where $\Omega_{\varepsilon}^{*}=T(\Omega_{\varepsilon})\subseteq \{x\in \R^N:|x|\geq \frac{1}{\alpha \tau_{\varepsilon}}\}$ for some $\alpha>0$ with the inversion map
$T:\R^N\setminus\{0\}\rightarrow\R^N\setminus\{0\}:T(x)=\frac{x}{|x|^2}$
. Then a direct calculation yields that
\begin{align*}
    -\Delta w_{\varepsilon}(x)=w_{\varepsilon}^{2^{*}_{\mu}-1}(x)\left(\displaystyle{\displaystyle{\int_{\Omega_{\varepsilon}^{*}}}}\frac{w_\varepsilon^{2^{*}_{\mu}}(z)}{|x-z|^{\mu}}dz\right)+\frac{\varepsilon}{\tau_{\varepsilon}^{2}|x|^{4}} w_\varepsilon(x)
    :=a_{\varepsilon}(x)w_\varepsilon(x),\quad  &\text{in}\,\,\Omega_{\varepsilon}^{*},
\end{align*}
where
$$a_{\varepsilon}(x):=w_{\varepsilon}^{2^{*}_{\mu}-2}(x)\left(\displaystyle{\int_{\Omega_{\varepsilon}^{*}}}\frac{w_\varepsilon^{2^{*}_{\mu}}(z)}{|x-z|^{\mu}}dz\right)+\frac{\varepsilon}{\tau_{\varepsilon}^{2}|x|^{4}}.$$
From Lemma \ref{lema 2.1}, it is enough to clarify that $a_{\varepsilon}(x)\in L^{\frac{N}{2}}(\Omega_{\varepsilon}^{*}\cap B_{R}(0))$. This tells that there exist $C>0$ and $R>0$ such that
\begin{align}\label{eq 2.6}
    \sup_{x\in \Omega_{\varepsilon}^{*}\cap B_{{R}/{2}}(0)}w_\varepsilon(x)\leq C,
\end{align}
which implies \eqref{eq 2.5} holds.  To this end, firstly using similar arguments in \cite[Lemma 3.1]{G1}, we obtain that there exists $\varepsilon_{0}>0$ such that
\begin{equation}\label{eq 2.7}
    \begin{split}
      &\frac{\varepsilon}{\tau_{\varepsilon}^{2}}\left(\displaystyle{\displaystyle{\int_{\Omega_{\varepsilon}^{*}\cap B_{R}(0)}}}\frac{1}{|x|^{4\cdot{\frac{N}{2}}}}dx\right)^{\frac{2}{N}}\leq \frac{\varepsilon}{\tau_{\varepsilon}^{2}}\left(\displaystyle{\displaystyle{\int_{{B^{c}_{{1}/{\alpha \tau_{\varepsilon}}}(0)}\bigcap B_{R}(0)}}}\frac{1}{|x|^{2N}}dx\right)^{\frac{2}{N}} \\
    =&\frac{\varepsilon}{\tau_{\varepsilon}^{2}}\left(\displaystyle{\displaystyle{\int_{\frac{1}{\alpha \tau_{\varepsilon}}}^{R}}}\frac{\sigma_{N}\rho^{N-1}}{\rho^{2N}}d\rho\right)^{\frac{2}{N}}<\frac{\varepsilon_{0}}{2},\quad \text{for $\varepsilon>0$ small and $R>0$}.
    \end{split}
\end{equation}
For the other term in $a_\varepsilon(x)$, since
$$\displaystyle{\int_{\Omega_{\varepsilon}^{*}}}\frac{w_\varepsilon^{2^{*}_{\mu}}(y)}{|x-y|^{\mu}}dy=\frac{1}{|x|^{\mu}}\displaystyle{\int_{\Omega_{\varepsilon}}}\frac{\tilde{u}_{\varepsilon}^{2^{*}_{\mu}}(y)}{|\frac{x}{|x|^{2}}-y|^{\mu}}dy,\quad \forall\, x\in \Omega_{\varepsilon}^{*}\cap B_{R}(0),$$
 a direct computation yields that
\begin{align}\label{eq 2.8}
&\left(\displaystyle{\int_{\Omega_{\varepsilon}^{*}\cap B_{R}(0)}}\left|w_{\varepsilon}^{2^{*}_{\mu}-2}(x)\left(\displaystyle{\displaystyle{
\int_{\Omega_{\varepsilon}^{*}}}}\frac{w_\varepsilon^{2^{*}_{\mu}}(y)}{|x-y|^{\mu}}dy\right)
\right|^{\frac{N}{2}}dx\right)^{\frac{2}{N}}\notag\\
=&\left(\displaystyle{\int_{\Omega_{\varepsilon}\cap B^{c}_{{1}/{R}}(0)}}\left|\tilde{u}_{\varepsilon}^{2^{*}_{\mu}-2}(x)
\left(\displaystyle{\int_{\Omega_{\varepsilon}}}\frac{\tilde{u}_{\varepsilon}^{2^{*}_{\mu}}
(y)}{|x-y|^{\mu}}dy\right)\right|^{\frac{N}{2}}dx\right)^{\frac{2}{N}}\notag\\
=&\left(\displaystyle{\int_{\Omega\cap B^{c}_{{1}/{R\tau_{\varepsilon}}}(x_{\varepsilon})}}\left|u_{\varepsilon}^{2^{*}_{\mu}-2}
(x)\left(\displaystyle{\int_{\Omega}}\frac{u_{\varepsilon}^{2^{*}_{\mu}}(y)}{|x-y|^{\mu}}
dy\right)\right|^{\frac{N}{2}}dx\right)^{\frac{2}{N}}\notag\\
\underset{(1)}{\leq }&~C\left(\displaystyle{\int_{\Omega\cap B^{c}_{{1}/{R\tau_{\varepsilon}}}(x_{\varepsilon})}}|u_{\varepsilon}|^{\left(2^{*}_{\mu}-2\right)
\frac{N}{2}}dx\right)^{\frac{2}{N}}\underset{(2)}{\leq} C\left(\displaystyle{\displaystyle{\int_{\Omega\cap B^{c}_{{1}/{R\tau_{\varepsilon}}}(x_{\varepsilon})}}}U_{x_\varepsilon,\tau_{\varepsilon}
}^{\left(2^{*}_{\mu}-2\right)\frac{N}{2}}(x)dx\right)^{\frac{2}{N}}\notag\\
{\leq}~&
\begin{cases}
     \,\frac{C}{\tau_{\varepsilon}^{\frac{\mu}{2}}}\left(\displaystyle{\int_{\R^N}}U_{0,1}^{\left(2^{*}_{\mu}-2\right)\frac{N}{2}}(x)dx\right)^{\frac{2}{N}},\quad &\text{if $\mu\in(0,2)$,}\\
    \,\frac{C}{\tau_{\varepsilon}^{\frac{4-\mu}{2}}}\left(\displaystyle{\int_{\Omega}}U_{x_{\varepsilon},\tau_\varepsilon}^{\left(2^{*}_{\mu}-2\right)\frac{N}{2}}(x)dx\right)^{\frac{2}{N}},\quad &\text{if $\mu\in(2,4)$,}
\end{cases}\notag\\
<&~\frac{\varepsilon_{0}}{2},\quad\, \quad\,\text{for $\varepsilon>0$ small and $R>0$}.
\end{align}
Here (1) and (2) hold since the following fact  founded in \cite[Lemma 6.1]{GY}
$$\displaystyle{\int_{\Omega}}\frac{u_{\varepsilon}^{2^{*}_{\mu}}(y)}{|x-y|^{\mu}}dy\in L^{\infty}(\Omega),$$
and the blow-up solution $u_{\varepsilon}$ can be written as $u_{\varepsilon}=\xi_{\varepsilon}PU_{x_{\varepsilon,\tau_{\varepsilon}}}+\zeta_{\varepsilon}$, with $PU_{x_{\varepsilon,\tau_{\varepsilon}}}\leq U_{x_{\varepsilon,\tau_{\varepsilon}}}$ and $\|\zeta_{\varepsilon}\|_{H^{1}_{0}(\Omega)}\ll1$.
WLOG, take $R=1$. Hence, from \eqref{eq 2.7}, \eqref{eq 2.8}, \eqref{eq 1.10}, Theorem A-$(i)$ and Lemma \ref{lema 2.1}, we can derive \eqref{eq 2.6} since
\begin{align*}
&\displaystyle{\int_{\Omega_{\varepsilon}^{*}\cap B_{1}(0)}}\left|w_{\varepsilon}\right|^{2^{*}}dx= \displaystyle{\int_{\Omega_{\varepsilon}\backslash B_{1}(0)}}\left|\tilde{u}_{\varepsilon}\right|^{2^{*}}dx=\displaystyle{\int_{\Omega\cap B^{c}_{{1}/{\tau_\varepsilon}}(x_\varepsilon)}}\left|{u}_{\varepsilon}\right|^{2^{*}}dx\\
\leq &\displaystyle{\int_{\Omega}}\left|{u}_{\varepsilon}\right|^{2^{*}}dx\leq S^{-1}\left(\displaystyle{\int_{\Omega}}|\nabla u_{\varepsilon}|^{2}dx\right)^{\frac{2^{*}}{2}}\leq S^{-1}\left(S_{HL}^{\frac{2N-\mu}{N+2-\mu}}\right)^{\frac{2^{*}}{2}} =C S^{\frac{N^{2}-2\mu+4}{(N-2)(N+2-\mu)}}<+\infty.
\end{align*}

\smallskip

Now, we prove \eqref{eq 2.3}.  Note that the function $\|u_\varepsilon\|_{\infty}u_\varepsilon$ satisfies
\begin{align*}
-\Delta\big(\|u_\varepsilon\|_{\infty}u_\varepsilon(x)\big)=\|u_\varepsilon\|_{\infty}\left(\varepsilon u_{\varepsilon}(x)+\left(\displaystyle{\int_{\Omega}}\frac{u_{\varepsilon}^{2^{*}_{\mu}}(y)}{|x-y|^{\mu}}dy\right)u_{\varepsilon}^{2^{*}_{\mu}-1}(x)\right):=f_{\varepsilon}(x),\quad x\in \Omega.
\end{align*}
By  \eqref{eq 1.7} and \eqref{eq 2.2}, we derive that for any $x\in\omega$, where $\omega$ is a neighborhood of $\partial\Omega$,
\begin{equation*}
\begin{split}
|f_{\varepsilon}(x)|\leq& ~C\|u_\varepsilon\|_{\infty}\left(\displaystyle{\int_{\R^N}}\frac{U_{x_\varepsilon,\tau_{\varepsilon}}^{2^{*}_{\mu}}(y)}{|x-y|^{\mu}}dy\right)U_{x_\varepsilon,\tau_{\varepsilon}}^{2^{*}_{\mu}-1}(x)+C\varepsilon \|u_\varepsilon\|_{\infty} U_{x_\varepsilon,\tau_{\varepsilon}}(x)\\
=&C
\|u_\varepsilon\|_{\infty}U^{2^{*}-1}_{x_\varepsilon,\lambda_{\varepsilon}}(x)+C\varepsilon \|u_\varepsilon\|_{\infty} U_{x_\varepsilon,\tau_{\varepsilon}}(x)
\\
\leq& ~C\left(\frac{1}{\tau_{\varepsilon}^{2}}\frac{1}{|x-x_{\varepsilon}|^{N+2}}+\varepsilon\frac{1}{|x-x_{\varepsilon}|^{N-2}}\right)\longrightarrow 0,\quad\text{as $\varepsilon\rightarrow0$}.
\end{split}
\end{equation*}
This implies that for $\varepsilon>0$ small, $\|f_{\varepsilon}\|_{L^{\infty}(\omega)}\rightarrow 0$. 
Also we have
\begin{align*}
    \|f_{\varepsilon}\|_{L^{1}(\Omega)} 
\leq&\displaystyle{\int_{\Omega_{\varepsilon}}}\left|\left(\displaystyle{\int_{\Omega_{\varepsilon}}}\frac{\tilde{u}_{\varepsilon}^{2^{*}_{\mu}}(y)}{|x-y|^{\mu}}dy\right)\tilde{u}_{\varepsilon}^{2^{*}_{\mu}-1}(x)\right|dx+ \frac{\varepsilon}{\tau_{\varepsilon}^{2}}\displaystyle{\int_{\Omega_{\varepsilon}}}\left|\tilde{u}_{\varepsilon}(x)\right|dx\\
\leq&~C\displaystyle{\int_{\R^N}}\left(\displaystyle{\int_{\R^N}}\frac{U_{0,1}^{2^{*}_{\mu}}(y)}{|x-y|^{\mu}}dy\right)U_{0,1}^{2^{*}_{\mu}-1}(x)dx+ C\frac{\varepsilon}{\tau_{\varepsilon}^{2}}\left(\displaystyle{\int_{B_{1}(0)}}U_{0,1}(x)dx+\displaystyle{\int_{\Omega_\varepsilon\backslash B_{1}(0)}}U_{0,1}(x)dx\right)\\
\leq&~C\displaystyle{\int_{\R^N}}U_{0,1}^{2^{*}-1}(x)dx+C\varepsilon=C\frac{\sigma_{N}}{N}+C\varepsilon\leq~C\frac{\sigma_{N}}{N}.
\end{align*}
Hence, $f_{\varepsilon}\in L^{1}(\Omega)$. Moreover, for $\varepsilon>0$ sufficiently small, we claim that
\begin{align}\label{eq 2.9}
    f_{\varepsilon}(x)\longrightarrow \frac{\sigma_{N}}{N}\delta_{x_{0}}(x) \quad\text{in the sense of distributions}.
\end{align}
Indeed, for any $\psi\in C_{0}^{\infty}(\Omega)$, we have
\begin{align*}
\displaystyle{\int_{\Omega}}f_{\varepsilon}(x)\psi(x)dx
=& \displaystyle{\int_{\Omega_{\varepsilon}}}
\displaystyle{\int_{\Omega_{\varepsilon}}}\frac{\tilde{u}_{\varepsilon}^{2^{*}_{\mu}}(y)
\,\tilde{u}_{\varepsilon}^{2^{*}_{\mu}-1}(x)\,{\psi}\left(\tau_{\varepsilon}^{-1}x
+x_\varepsilon\right)}{|x-y|^{\mu}}\,dxdy+\frac{\varepsilon}{\tau_{\varepsilon}^{2}}
\displaystyle{\int_{\Omega_{\varepsilon}}}\tilde{u}_{\varepsilon}(x){\psi}
\left(\tau_{\varepsilon}^{-1}x+x_\varepsilon\right)dx\\
=&\displaystyle{\int_{\Omega_{\varepsilon}}}
\displaystyle{\int_{\Omega_{\varepsilon}}}\frac{\tilde{u}_{\varepsilon}^{2^{*}_{\mu}}(y)
\,\tilde{u}_{\varepsilon}^{2^{*}_{\mu}-1}(x)\,{\psi}(\tau_{\varepsilon}^{-1}x+x_\varepsilon)}
{|x-y|^{\mu}}\,dxdy+O\left(\varepsilon\right)\\
\to&\psi(x_{0})\displaystyle{\int_{\R^N}}
\left(\displaystyle{\int_{\R^N}}\frac{U_{0,1}^{2^{*}_{\mu}}(y)}{|x-y|^{\mu}}dy\right)
U_{0,1}^{2^{*}_{\mu}-1}(x)dx\\=&\psi(x_{0})\displaystyle{\int_{\R^N}}U_{0,1}^{2^{*}-1}(x)dx=\frac{\sigma_{N}}{N}\psi(x_{0}),
\end{align*}
which implies that \eqref{eq 2.9} holds true. Therefore, according to \cite[Remark 2.2]{C}, since $\{f_{\varepsilon}\}$ is a bounded sequence both in $L^{1}(\Omega)$ and $L^{\infty}(\omega)$,
  \eqref{eq 2.3} follows by Lemma \ref{lema 2.3}.

\vskip 0.1cm

Next, we prove \eqref{eq 2.4}. In fact, from the Pohozaev identity \eqref{eq 2.1} with $y=x_0$ and \eqref{eq 2.3}, we have
\begin{align*}
    \text{LHS of \eqref{eq 2.1}}=&\varepsilon\left(\displaystyle{\int_{\Omega}}U_{x_{\varepsilon},\tau_{\varepsilon}}^{2}(x)dx+o(1)\right)=\frac{\varepsilon}{\tau_{\varepsilon}^{2}}\left(\displaystyle{\int_{\Omega_{\varepsilon}}}U_{0,1}^{2}(x)dx+o(1)\right)\\
    =&\frac{\varepsilon}{\tau_{\varepsilon}^{2}}a_{N}\sigma_{N}+o\left(\frac{\varepsilon}{\tau_{\varepsilon}^{2}}\right),\\
    \text{RHS of \eqref{eq 2.1}}=&\frac{\sigma_{N}^{2}}{2N^{2}\tau_{\varepsilon}^{N-2}}\left(\displaystyle{\int_{\partial\Omega}}\left(\frac{\partial G(x,x_{0})}{\partial \nu }\right)^{2}(x-x_0)\cdot {\nu}\,dS_{x}+o(1)\right)\\
    =&\frac{(N-2)\sigma_{N}^{2}R(x_{0})}{2N^{2}\tau_{\varepsilon}^{N-2}}+o\left(\frac{1}{\tau_{\varepsilon}^{N-2}}\right),
\end{align*}
where we use the fact that
$$\displaystyle{\int_{\partial\Omega}}\left(\frac{\partial G(x,x_{0})}{\partial \nu }\right)^{2}(x-x_0)\cdot {\nu}\,dS_{x}=(N-2)R(x_0).$$
Then, a simple computation yields \eqref{eq 2.4}.
\end{proof}
\end{lemma}

Now, we give a description of the eigenvalues and eigenfunctions to the associated linearized problem of \eqref{eq 1.5} at $U_{0,1}$.
\begin{lemma}\label{lma2.5} The eigenvalue problem
\begin{equation}\label{eq 1.14}
    \begin{cases}
         \!-\Delta \! V_{i}=\lambda_{i}\!\left(\!(2^{*}_\mu-1){U}_{0,1}^{2^{*}_\mu-2}V_{i}\!
         \displaystyle{\displaystyle{\int_{\R^N}}}\!\frac{{U}_{0,1}^{2^{*}_\mu}(y)}
         {|x-y|^{\mu}}dy+2^{*}_{\mu}{U}_{0,1}^{2^{*}_\mu-1}\!
        \displaystyle{\displaystyle{\int_{\R^N}}}\!\frac{{U}_{0,1}^{2^{*}_{\mu}-1}
         (y)V_{i}(y)}
         {|x-y|^{\mu}}\!dy\right),\\[6mm]
    V_{i}\in\mathcal{D}^{1,2}(\R^N),
    \end{cases}
\end{equation}
has eigenvalues
$$\lambda_{1}=\frac{1}{2\cdot2^{*}_\mu-1}<1=\lambda_{2}=\cdots=\lambda_{N+2}\leq \lambda_{N+3}\leq\cdots.$$
Moreover, ${U}_{0,1}(x)$ is the eigenfunction corresponding to $\lambda_{1}$, and the eigenfunctions related to $\lambda_{2},\cdots,\lambda_{N+2}$ are given by
$$\frac{\partial {U}_{\xi,1}(x)}{\partial \xi_{i}}\bigg|_{\xi=0}~~(\,\text{for $i=1,\cdots,N$}\,),\quad \frac{\partial {U}_{0,\tau}(x)}{\partial \tau}\bigg|_{\tau=1}.$$
  \begin{proof}
    The results can be directly obtained as a consequence of \eqref{eq 1.6} in Theorem A and the non-degeneracy result in Theorem B.
  \end{proof}
\end{lemma}

In the following, we show a fine estimate of the eigenfunctions ${v}_{i,\epsilon}$ to \eqref{eq 1.12}.
\begin{lemma}\label{lema 2.7}
For any $i\in \mathbb{N}$, let $v_{i,\varepsilon}$ be a solution of \eqref{eq 1.12}. Then, there exists a constant $C>0$ independent of $\varepsilon$ such that
    \begin{align}\label{eq 2.10}
    |v_{i,\varepsilon}(x)|\leq \frac{C}{(1+\tau_{\varepsilon}|x-x_{\varepsilon}|)^{N-2}},\quad \forall\, x\in \Omega.
\end{align}
\begin{proof}
Since $v_{i,\varepsilon}$ is a solution of \eqref{eq 1.12}, by Green's representation formula, we have
\begin{align*}
    v_{i,\varepsilon}(x)=&\underbrace{(2^{*}_\mu-1)\,\lambda_{i,\varepsilon}\displaystyle{\int_{\Omega}}G(x,z)u_{\varepsilon}^{2^{*}_\mu-2}(z)v_{i,\varepsilon}(z)\left(\displaystyle{\int_{\Omega}}\frac{u_{\varepsilon}^{2^{*}_\mu}(y)}{|z-y|^{\mu}}dy\right)dz}_{:=I}+\underbrace{\varepsilon\,\lambda_{i,\varepsilon}\displaystyle{\int_{\Omega}}G(x,z)v_{i,\varepsilon}(z)dz}_{:=II}\\
&+\underbrace{2^{*}_\mu\,\lambda_{i,\varepsilon}\displaystyle{\int_{\Omega}}G(x,z)u_{\varepsilon}^{2^{*}_\mu-1}(z)\left(\displaystyle{\int_{\Omega}}\frac{u_{\varepsilon}^{2^{*}_\mu-1}(y)v_{i,\varepsilon}(y)}{|z-y|^{\mu}}dy\right)dz}_{:=III}.
\end{align*}
   Recall that for any $x\in\Omega$, $|v_{i,\varepsilon}(x)|\leq 1$ and $u_{\varepsilon}(x)\leq CU_{x_{\varepsilon},\tau_{\varepsilon}}(x)$. Then, from Lemma \ref{lema 2.5}, we have
    \begin{align*}
        I\leq &~C\displaystyle{\int_{\Omega}}\frac{1}{|z-x|^{N-2}}U_{x_{\varepsilon},\tau_{\varepsilon}}^{2^{*}_\mu-2}(z)\left(\displaystyle{\int_{\R^{N}}}\frac{U_{x_{\varepsilon},\tau_{\varepsilon}}^{2^{*}_\mu}(y)}{|z-y|^{\mu}}dy\right)dz\\
\leq&~C\displaystyle{\int_{\Omega}}\frac{1}{|z-x|^{N-2}}U_{x_{\varepsilon},\tau_{\varepsilon}}^{2^{*}-2}(z)dz =C\displaystyle{\int_{\Omega}}\frac{1}{|z-x|^{N-2}}\frac{\tau_{\varepsilon}^{2}}{(1+\tau_{\varepsilon}^{2}|z-x_{\varepsilon}|^{2})^{2}}dz\\
\leq&~C \displaystyle{\int_{\R^N}}\frac{1}{|\tau_{\varepsilon}(x-x_{\varepsilon})-z|^{N-2}}\frac{1}{(1+|z|)^{4}}dz\leq \frac{C}{(1+\tau_{\varepsilon}|x-x_{\varepsilon}|)^{2}}.
    \end{align*}
Using similar arguments as \cite[Proposition 3.5]{CLP}, it is easy to check that $II=O(\varepsilon)$. Moreover, by Lemma \ref{lema 2.6} with $\mu\in(0,4)$ and HLS, we can calculate that
\begin{align*}
    III
    \leq & C\displaystyle\displaystyle{\int_{\Omega}}\frac{1}{|z-x|^{N-2}}U_{x_{\varepsilon},
    \tau_{\varepsilon}}^{2^{*}_\mu-1}(z)\left({\displaystyle\displaystyle{\int_{\Omega}}}
    \frac{U_{x_{\varepsilon},\tau_{\varepsilon}}^{2^{*}_\mu-1}(y)}{|z-y|^{\mu}}dy\right)dz\\
   \leq& ~C\!\left(\!{\displaystyle\displaystyle{\int_{\Omega}}}\frac{\tau_{\varepsilon}^{\frac{N(N+2-\mu)}{2N-\mu}}}{(1+\tau_{\varepsilon}|y-x_{\varepsilon}|)^{\frac{2N(N+2-\mu)}{2N-\mu}}}dy\!\right)^{\frac{2N-\mu}{2N}}\!\left(\displaystyle{\displaystyle{\int_{\Omega}}}\frac{1}{|z-x|^{\frac{2N(N-2)}{2N-\mu}}}\frac{\tau_{\varepsilon}^{\frac{N(N+2-\mu)}{2N-\mu}}}{(1+\tau_{\varepsilon}|z-x_{\varepsilon}|)^{\frac{2N(N+2-\mu)}{2N-\mu}}}dz\!\right)^{\frac{2N-\mu}{2N}}\\
 \leq&  ~C\left(\displaystyle{\displaystyle{\int_{\R^N}}}\frac{1}{|\tau_{\varepsilon}(x-x_{\varepsilon})-z)|^{\frac{2N(N-2)}{2N-\mu}}}\frac{1}{(1+|z|)^{\frac{2N(N+2-\mu)}{2N-\mu}}}dz\right)^{\frac{2N-\mu}{2N}}\leq  \frac{C}{(1+\tau_{\varepsilon}|x-x_{\varepsilon}|)^{\frac{2N-\mu}{2}}}.
\end{align*}
Hence, we derive that
\begin{align*}
    v_{i,\varepsilon}(x)=O\left(\frac{1}{(1+\tau_{\varepsilon}|x-x_{\varepsilon}|)^{2}}\right)+O(\varepsilon).
\end{align*}
Repeating the above process, we also have
\begin{align*}
    v_{i,\varepsilon}(x)=O\left(\frac{1}{(1+\tau_{\varepsilon}|x-x_{\varepsilon}|)^{4}}\right)+O(\varepsilon^{2}).
\end{align*}
Then, we can proceed the above arguments for finite number of times to obtain \eqref{eq 2.10}.
\end{proof}
\end{lemma}

\begin{Rem}\label{rmk2.1}
    From Lemma \ref{lema 2.7}, we deduce that for any $x\in \Omega$, $\|u_\varepsilon\|_{\infty}\left|v_{i,\varepsilon}(x)\right|\leq CU_{x_{\varepsilon},\tau_{\varepsilon}}(x)$, which implies that the rescaled eigenfuntion $\tilde{v}_{i,\varepsilon}$ satisfies $|\tilde{v}_{i,\varepsilon}(x)|\leq CU_{0,1}(x)$, $\forall x\in \Omega_{\varepsilon}$.
\end{Rem}

In the subsequent two Lemmas, we perform the asymptotic behavior of the eigenfuntions ${v}_{i,\varepsilon}$ to \eqref{eq 1.12} and the rescaled eigenfuntions $\tilde{v}_{i,\varepsilon}$ to \eqref{eq 1.13}.
\begin{lemma} \label{lema 2.8}
For any $i\in \mathbb{N}$, suppose that $\lambda_{i}=\underset{\varepsilon\rightarrow0}{\lim}\lambda_{i,\varepsilon}=1$. Then as $\varepsilon\rightarrow0$, we have
\begin{align}\label{eq 2.11}
\tilde{v}_{i,\varepsilon}(x)\longrightarrow\sum_{k=1}^{N}
\frac{a_{i,k}\,x_k}{\big(1+|x|^{2}\big)^{\frac{N}{2}}}+b_{i}\frac{1-|x|^{2}}{\big(1+|x|^{2}\big)^{\frac{N}{2}}},\quad\text{in $C^{1}_{loc}(\R^N)$,}
\end{align}
 where $(a_{i,1},\cdots,a_{i,N},b_{i})\neq(0,\cdots,0)\in \R^{N+1}$.
\begin{proof}
 Recalling that the rescaled eigenfuntion $\tilde{v}_{i,\varepsilon}$ satisfies \eqref{eq 1.13},
by the elliptic theory, we obtain that there exists $V_{i}$ such that, up to a subsequence, $\{\tilde{v}_{i,\varepsilon}\}$ converges to $V_i$ in $C^{1}_{loc}(\R^N)$. Moreover, from \eqref{eq 1.9}, we have $\tilde{u}_{\varepsilon}\rightarrow U_{0,1}$ in $C^{1}_{loc}(\R^N)$ as $\varepsilon\rightarrow 0$. Then, passing to the limit in \eqref{eq 1.13}, we obtain that $V_i$ satisfies \eqref{eq 1.14} with $\lambda_{i}=1$.
On the other hand, by the estimates $$\tilde{u}_{i,\varepsilon}(x)\leq C_{1}U_{0,1}(x),\quad \left|\tilde{v}_{i,\varepsilon}(x)\right|\leq C_{2}U_{0,1}(x),\quad \forall\,x\in \Omega_{\varepsilon}, $$
  a direct calculation yields that $\tilde{v}_{i,\varepsilon}$ is uniformly bounded in $\mathcal{D}^{1,2}(\R^N)$. Indeed,
\begin{align*}
\displaystyle{\int_{\Omega_{\varepsilon}}}|\tilde{v}_{i,\varepsilon}|^{2^{*}}dx\leq C\displaystyle{\int_{\R^N}}|U_{0,1}|^{2^{*}}dx<+\infty.
\end{align*}
Moreover, from \eqref{eq 1.13}, we have
\begin{equation*}\label{eq 2.12}
    \begin{split}
&\displaystyle{\int_{\Omega_{\varepsilon}}}|\nabla\tilde{v}_{i,\varepsilon}|^{2}dx\\
\leq&(2^{*}_\mu-1)\displaystyle{\int_{\Omega_{\varepsilon}}}\tilde{u}_{\varepsilon}^{2^{*}_\mu-2}\tilde{v}_{i,\varepsilon}^{2}
\left(\displaystyle{\int_{\Omega_{\varepsilon}}}\frac{\tilde{u}_{\varepsilon}^{2^{*}_\mu}(y)}{|x-y|^{\mu}}dy\right)dx+\frac{\varepsilon}{\tau_{\varepsilon}^{2}} \displaystyle{\int_{\Omega_{\varepsilon}}}\tilde{v}_{i,\varepsilon}^{2}dx\\
&+2^{*}_{\mu} \displaystyle{\int_{\Omega_{\varepsilon}}}\tilde{u}_{\varepsilon}^{2^{*}_\mu-1}|\tilde{v}_{i,\varepsilon}|
\left(\displaystyle{\int_{\Omega_{\varepsilon}}}\frac{\tilde{u}_{\varepsilon}^{2^{*}_\mu-1}(y)|\tilde{v}_{i,\varepsilon}(y)|}{|x-y|^{\mu}}dy\right)dx.
\end{split}\end{equation*}
Note that \begin{equation*}
    \begin{split}
      \frac{\varepsilon}{\tau_{\varepsilon}^{2}} \displaystyle{\int_{\Omega_{\varepsilon}}}\tilde{v}_{i,\varepsilon}^{2}dx\leq &\,\frac{\varepsilon}{\tau_{\varepsilon}^{2}} \left(\frac{1}{\Lambda_{1}\left(\Omega_{\varepsilon}\right)}\displaystyle{\int_{\Omega_{\varepsilon}}}|\nabla \tilde{v}_{i,\varepsilon}|^{2}dx\right)=\varepsilon\left(\frac{1}{\Lambda_{1}(\Omega)}\displaystyle{\int_{\Omega_{\varepsilon}}}|\nabla \tilde{v}_{i,\varepsilon}|^{2}dx\right),
    \end{split}
\end{equation*}
where $\Lambda_{1}(\cdot)$ means the first eigenvalue of $-\Delta$ under Dirichlet boundary condition and $\Lambda_{1}\left(\Omega_{\varepsilon}\right)=\|u_{\varepsilon}\|_{\infty}^{2-2^{*}}\Lambda_{1}\left(\Omega\right)$.
Then, we obtain that
\begin{equation}\label{2-13}
    \begin{split}
        &\big(1+o(1)\big)
\displaystyle{\int_{\Omega_{\varepsilon}}}|\nabla\tilde{v}_{i,\varepsilon}|^{2}dx\\
\leq &~ (2^{*}_\mu-1)\displaystyle{\int_{\Omega_{\varepsilon}}}\tilde{u}_{\varepsilon}^{2^{*}_\mu-2}\tilde{v}_{i,\varepsilon}^{2}\left(\displaystyle{\int_{\Omega_{\varepsilon}}}\frac{\tilde{u}_{\varepsilon}^{2^{*}_\mu}(y)}{|x-y|^{\mu}}dy\right)dx\\
&+2^{*}_{\mu} \displaystyle{\int_{\Omega_{\varepsilon}}}\tilde{u}_{\varepsilon}^{2^{*}_\mu-1}|\tilde{v}_{i,\varepsilon}|\left(\displaystyle{\int_{\Omega_{\varepsilon}}}\frac{\tilde{u}_{\varepsilon}^{2^{*}_\mu-1}(y)|\tilde{v}_{i,\varepsilon}(y)|}{|x-y|^{\mu}}dy\right)dx\\
\leq&~C\displaystyle{\int_{\R^N}}U_{0,1}^{2^{*}_\mu}(x)\left(\displaystyle{\int_{\R^N}}\frac{U_{0,1}^{2^{*}_\mu}(y)}{|x-y|^{\mu}}dy\right)dx=~ C\displaystyle{\int_{\R^N}}U_{0,1}^{2^{*}}(x)dx<+\infty.
    \end{split}
\end{equation}
Hence, the sequence $\tilde{v}_{i,\varepsilon}$ is uniformly bounded in $\mathcal{D}^{1,2}(\R^N)$ and the limit function  $V_{i}\in\mathcal{D}^{1,2}(\R^N)$ verifying \eqref{eq 1.14}. Under the assumption that $\lambda_{i}=\underset{\varepsilon\rightarrow0}{\lim}\lambda_{i,\varepsilon}=1$, from Lemma \ref{lma2.5}, we deduce that there exist some constants $a_{i,1},\cdots,a_{i,N}$ and $b_{i}$ such that
$$ V_i(x)=\sum_{k=1}^{N}\frac{a_{i,k}\,x_k}{\big(1+|x|^{2}\big)^{\frac{N}{2}}}+b_{i}\frac{1-|x|^{2}}{\big(1+|x|^{2}\big)^{\frac{N}{2}}}.$$
In addition, we can derive that the limit function $V_{i}\not\equiv 0$ by proving the vector $(a_{i,1},\cdots,a_{i,N},b_{i})$ is not equal to zero. Assume by contradiction that $(a_{i,1},\cdots,a_{i,N},b_{i})=(0,\cdots,0)$, then $\tilde{v}_{i,\varepsilon}$ would converge to zero. On the other hand, since $\|{v}_{i,\varepsilon}\|_{L^{\infty}(\Omega)}=\|\tilde{v}_{i,\varepsilon}\|_{L^{\infty}(\Omega_\varepsilon)}=1$, there would exist $x_{\varepsilon}^{i}\in \Omega$ such that $|{v}_{i,\varepsilon}(x_{\varepsilon}^{i})|=1$ and then $|\tilde{v}_{i,\varepsilon}(y_{\varepsilon}^{i})|=1$ with $y_{\varepsilon}^{i}=(x_{\varepsilon}^{i}-x_\varepsilon)\|u_\varepsilon\|_{\infty}^{\frac{2}{N-2}}$. Therefore, if $\tilde{v}_{i,\varepsilon}$ converges to zero in $C^{1}_{loc}(\R^N)$, the point $y_{\varepsilon}^{i}\rightarrow \infty$ as $\varepsilon\rightarrow 0$, which is a contradiction with the estimate $|\tilde{v}_{i,\varepsilon}|\leq CU_{0,1}$ given in Remark \ref{rmk2.1}. Thus, we complete the proof.
\end{proof}
\end{lemma}

\begin{lemma}\label{lema 2.9}
For any $i\in \mathbb{N}$, suppose that $\lambda_{i}=\underset{\varepsilon\rightarrow0}{\lim}\lambda_{i,\varepsilon}=1$ and $b_{i}\neq 0$ in \eqref{eq 2.11}. Then, we have
\begin{align}\label{eq 2.14}
\|u_\varepsilon\|_{\infty}^{2}v_{i,\varepsilon}(x)\longrightarrow -\frac{\sigma_{N}}{N}b_{i}G(x,x_0),\quad \text{in $C^{1}_{loc}(\Omega\backslash\{x_0\})$ as $\varepsilon\rightarrow0$}.
\end{align}
\begin{proof}
    Define $\phi_{i,\varepsilon}=\|u_\varepsilon\|_{\infty}^{2}v_{i,\varepsilon}$. A direct calculation yields that
    \begin{align*}
        -\Delta \phi_{i,\varepsilon}=f_{i,\varepsilon}(x),\quad x\in \Omega,
    \end{align*}
where
\begin{align*}
    f_{i,\varepsilon}(x):=&(2^{*}_\mu-1)\lambda_{i,\varepsilon}\|u_\varepsilon\|_{\infty}^{2}u^{2^{*}_\mu-2}_{\varepsilon}(x)v_{i,\varepsilon}(x)\left(\displaystyle{\int_{\Omega}}\frac{u^{2^{*}_\mu}_{\varepsilon}(y)}{|x-y|^{\mu}}dy\right)+\varepsilon\,\lambda_{i,\varepsilon}\|u_\varepsilon\|_{\infty}^{2} v_{i,\varepsilon}(x)\nonumber\\
&+2^{*}_\mu\lambda_{i,\varepsilon}\|u_\varepsilon\|_{\infty}^{2} u^{2^{*}_\mu-1}_{\varepsilon}(x)\left(\displaystyle{\int_{\Omega}}\frac{u^{2^{*}_\mu-1}_{\varepsilon}(y)v_{i,\varepsilon}(y)}{|x-y|^{\mu}}dy\right).
\end{align*}
Recall that
\begin{alignat*}{3}
&{u}_\varepsilon(x)\leq C U_{x_\varepsilon,\tau_{\varepsilon}}(x), &&\quad |{v}_{i,\varepsilon}(x)|\leq \frac{C }{\left(1+\tau_{\varepsilon}^{2}|x-x_{\varepsilon}|^{2}\right)^{\frac{N-2}{2}}},&&\quad \forall x\in\Omega, \\
      &\tilde{u}_\varepsilon(x)\leq C U_{0,1}(x), &&\quad|\tilde{v}_{i,\varepsilon}(x)|\leq C U_{0,1}(x),&&\quad \forall x\in\Omega_{\varepsilon}.
\end{alignat*}
Then a direct computation yields that
\begin{align*}
\|f_{i,\varepsilon}\|_{L^{1}(\Omega)}\leq
&(2^{*}_\mu-1)\displaystyle{\int_{\Omega_{\varepsilon}}}\tilde{u}_{\varepsilon}^{2^{*}_\mu-2}(x)|\tilde{v}_{i,\varepsilon}(x)|\left(\displaystyle{\int_{\Omega_{\varepsilon}}}\frac{\tilde{u}_{\varepsilon}^{2^{*}_\mu}(y)}{|x-y|^{\mu}}dy\right)dx+\frac{\varepsilon}{\tau_{\varepsilon}^{2}} \displaystyle{\int_{\Omega_{\varepsilon}}} |\tilde{v}_{i,\varepsilon}(x)|dx\\
&+2^{*}_\mu \displaystyle{\int_{\Omega_{\varepsilon}}}\tilde{u}_{\varepsilon}^{2^{*}_\mu-1}(x)\left(\displaystyle{\int_{\Omega_{\varepsilon}}}\frac{\tilde{u}_{\varepsilon}^{2^{*}_\mu-1}(y)|\tilde{v}_{i,\varepsilon}(y)|}{|x-y|^{\mu}}dy\right)dx\\
\leq&~C(2\cdot2^{*}_\mu-1)\displaystyle{\int_{\R^N}}U_{0,1}^{2^{*}_\mu-1}(x)\left(\displaystyle{\int_{\R^N}}\frac{U_{0,1}^{2^{*}_\mu}(y)}{|x-y|^{\mu}}dy\right)dx\\
&+\frac{C\,\varepsilon}{\tau_{\varepsilon}^{2}} \left(\displaystyle{\int_{ B_{1}(0)} }U_{0,1}(x)dx+\displaystyle{\int_{\Omega_{\varepsilon}\backslash B_{1}(0)}} U_{0,1}(x)dx\right)\\
\leq&~C\displaystyle{\int_{\R^N}}U_{0,1}^{2^{*}-1}(x)dx+C\varepsilon<+\infty.
\end{align*}
Hence, $f_{i,\varepsilon}(x)\in{L^{1}(\Omega)}$. Moreover, for any compact set $\omega$ in $\bar{\Omega}\backslash\{x_0\}$, from \eqref{eq 2.4}, we obtain that
\begin{align*}
    \|f_{i,\varepsilon}\|_{L^{\infty}(\omega)}\leq&C\|u_\varepsilon\|_{\infty}\left((2\cdot2^{*}_\mu-1)\,\underset{x\in\omega}{\sup}\,\left|U^{2^{*}_\mu-1}_{x_\varepsilon,\tau_\varepsilon}(x)\left(\displaystyle{\int_{\Omega}}\frac{U_{x_\varepsilon,\tau_\varepsilon}^{2^{*}_\mu}(y)}{|x-y|^{\mu}}dy\right)\right|+\varepsilon\,\underset{x\in\omega}{\sup}\,U_{x_\varepsilon,\tau_\varepsilon}(x)\right)\\
\leq&C\|u_\varepsilon\|_{\infty}\left((2\cdot2^{*}_\mu-1)\,\underset{x\in\omega}{\sup}\,U^{2^{*}-1}_{x_\varepsilon,\tau_\varepsilon}(x)+\varepsilon\,\underset{x\in\omega}{\sup}\,U_{x_\varepsilon,\tau_\varepsilon}(x)\right)\\
\leq&\frac{C}{\|u_\varepsilon\|_{\infty}^{2^{*}-2}}+\frac{C}{\|u_\varepsilon\|_{\infty}^{\frac{2(N-4)}{N-2}}}\rightarrow0,\quad\text{as $\varepsilon\rightarrow 0$.}
\end{align*}
Now, we claim that $f_{i,\varepsilon}(x)\rightarrow -\frac{\sigma_{N}}{N}b_{i}\delta_{x_{0}}(x)$ in the sense of distribution.
Indeed, for any $\psi\in C_{0}^{\infty}(\Omega)$, we denote its rescaled function as
$$\tilde{\psi}(x):=\psi(\tau_\varepsilon^{-1}x+x_\varepsilon),\quad \forall\,x\in\Omega_{\varepsilon}:=\left\{x\in\R^N:\tau_{\varepsilon}^{-1}x+x_\varepsilon\in\Omega\right\}.$$
Then, it holds that
\begin{align*}
    &\displaystyle{\int_{\Omega}} f_{i,\varepsilon}(x)\psi(x)dx\\
=&(2^{*}_\mu-1)\lambda_{i,\varepsilon}\left(\displaystyle{\int_{\Omega_{\varepsilon}}}\tilde{u}^{2^{*}_\mu-2}_{\varepsilon}(x)\tilde{v}_{i,\varepsilon}(x)\tilde{\psi}(x)\left(\displaystyle{\int_{\Omega_{\varepsilon}}}\frac{\tilde{u}^{2^{*}_\mu}_{\varepsilon}(y)}{|x-y|^{\mu}}dy\right)dx\right)+\frac{\varepsilon}{\tau^{2}_{\varepsilon}}\lambda_{i,\varepsilon}\displaystyle{\int_{\Omega_{\varepsilon}}}\tilde{v}_{i,\varepsilon}(x)\tilde{\psi}(x)dx\\
&+2^{*}_\mu\lambda_{i,\varepsilon}\left(\displaystyle{\int_{\Omega_{\varepsilon}}}\tilde{u}^{2^{*}_\mu-1}_{\varepsilon}(x)\tilde{\psi}(x)\left(\displaystyle{\int_{\Omega_{\varepsilon}}}\frac{\tilde{u}^{2^{*}_\mu-1}_{\varepsilon}(y)\tilde{v}_{i,\varepsilon}(y)}{|x-y|^{\mu}}dy\right)dx\right)\\
=&(2^{*}_\mu-1)b_{i}{\psi}(x_0)\left(\displaystyle{\int_{\R^N}}U_{0,1}^{2^{*}_\mu-2}(x)\,\frac{1-|x|^{2}}{\big(1+|x|^{2}\big)^{\frac{N}{2}}}\left(\displaystyle{\int_{\R^N}}\frac{U_{0,1}^{2^{*}_\mu}(y)}{|x-y|^{\mu}}dy\right)dx+o(1)\right)+O(\varepsilon)\\
&+2^{*}_\mu b_{i}\psi(x_0)\left(\displaystyle{\int_{\R^N}}U_{0,1}^{2^{*}_\mu-1}(x)\left(\displaystyle{\int_{\R^N}}\frac{U_{0,1}^{2^{*}_\mu-1}(y)}{|x-y|^{\mu}}\,\frac{1-|y|^{2}}{\big(1+|y|^{2}\big)^{\frac{N}{2}}}dy\right)dx+o(1)\right)\\
\longrightarrow&(2^{*}-1)b_{i}{\psi}(x_0)\displaystyle{\int_{\R^N}}\frac{1-|x|^{2}}{\big(1+|x|^{2}\big)^{\frac{N+4}{2}}}dx
=-\frac{\sigma_{N}}{N}b_{i}\psi(x_{0}),
\end{align*}
where we have used the results obtained in Lemma \ref{lema 2.8}, Lemma \ref{a-1} and the following identity

$$\displaystyle{\int_{\R^N}}U_{0,1}^{2^{*}-2}(x)\frac{x_k}{\big(1+|x|^{2}\big)^{\frac{N}{2}}}dx=0,\quad \text{for $k=1,\cdots,N$.}$$
 Then, utilizing Lemma \ref{lema 2.3} to the function $\phi_{i,\varepsilon}$,
we obtain the desired result.
\end{proof}
\end{lemma}
Now, we prove the following identity, which is useful in the proof of Lemma \ref{lema 2.11}.
\begin{lemma}\label{lema 2.10}
 For any $z\in\R^N$, define $\omega_{\varepsilon}(x)=(x-z)\cdot\nabla u_{\varepsilon}(x)+\frac{N-2}{2}u_\varepsilon(x)$. Then the following integral identity holds
  \begin{align}
&\displaystyle{\int_{\partial\Omega}}\frac{\partial u_{\varepsilon}}{\partial \nu}(x)\frac{\partial v_{i,\varepsilon}}{\partial \nu}(x)(x-z)\cdot \nu\,dS_{x}\nonumber\\
      =&\left(1-\lambda_{i,\varepsilon}\right)\Bigg\{\varepsilon\displaystyle{\int_{\Omega}}v_{i,\varepsilon}(x)\omega_{\varepsilon}(x)dx+(2^{*}_\mu-1)\displaystyle{\int_{\Omega}}u^{2^{*}_\mu-2}_{\varepsilon}(x)v_{i,\varepsilon}(x)\omega_{\varepsilon}(x)\left(\displaystyle{\int_{\Omega}}\frac{u^{2^{*}_\mu}_{\varepsilon}(y)}{|x-y|^{\mu}}dy\right)dx\nonumber\\
&\,\quad\,\quad\,\,\quad\,\quad+2^{*}_{\mu}\displaystyle{\int_{\Omega}}\displaystyle{\int_{\Omega}}\frac{u^{2^{*}_\mu-1}_{\varepsilon}(x)v_{i,\varepsilon}(x)u^{2^{*}_\mu-1}_{\varepsilon}(y)\omega_{\varepsilon}(y)}{|x-y|^{\mu}}dxdy\Bigg\}+2\varepsilon\displaystyle{\int_{\Omega}}u_{\varepsilon}(x)v_{i,\varepsilon}(x)dx.\label{eq 2.15}
  \end{align}

  \begin{proof}
     It follows form  Lemma \ref{lma A-3} that
      \begin{align}\label{eq 2.16}
          -\Delta \omega_{\varepsilon}(x)=&\varepsilon \omega_{\varepsilon}(x)+\left(2^{*}_\mu-1\right)u_\varepsilon^{2^{*}_\mu-2}(x)\omega_{\varepsilon}(x)\left(\displaystyle{\int_{\Omega}}\frac{u_\varepsilon^{2^{*}_\mu}(y)}{|x-y|^{\mu}}dy\right)\nonumber\\
          &+2^{*}_\mu u_\varepsilon^{2^{*}_\mu-1}(x)\left(\displaystyle{\int_{\Omega}}\frac{u_\varepsilon^{2^{*}_\mu-1}(y)\omega_{\varepsilon}(y)}{|x-y|^{\mu}}dy\right)+2\varepsilon u_\varepsilon(x),\quad x\in\Omega.
      \end{align}
 On the one hand, multiplying \eqref{eq 2.16} by $v_{i,\varepsilon}$ and integrating, we obtain that
      \begin{equation}\label{eq-2.17}
          \begin{split}
\displaystyle{\int_{\Omega}}\nabla \omega_\varepsilon(x) \nabla v_{i,\varepsilon}(x) dx=&\varepsilon  \displaystyle{\int_{\Omega}} \omega_\varepsilon(x) v_{i,\varepsilon}(x) dx+2\varepsilon \displaystyle{\int_{\Omega}} u_\varepsilon(x)v_{i,\varepsilon}(x) dx\\
          &+\left(2^{*}_\mu-1\right)\displaystyle{\int_{\Omega}}u_\varepsilon^{2^{*}_\mu-2}(x)\omega_{\varepsilon}(x)v_{i,\varepsilon}(x)\left(\displaystyle{\int_{\Omega}}\frac{u_\varepsilon^{2^{*}_\mu}(y)}{|x-y|^{\mu}}dy\right)dx\\
          &+2^{*}_\mu \displaystyle{\int_{\Omega}}u_\varepsilon^{2^{*}_\mu-1}(x)v_{i,\varepsilon}(x)\left(\displaystyle{\int_{\Omega}}\frac{u_\varepsilon^{2^{*}_\mu-1}(y)\omega_{\varepsilon}(y)}{|x-y|^{\mu}}dy\right)dx.
          \end{split}
      \end{equation}
      On the other hand, multiplying \eqref{eq 1.12} by $\omega_{\varepsilon}$, we have
    \begin{equation}\label{eq-2.18}
          \begin{split}
    &\displaystyle{\int_{\Omega}}\nabla \omega_\varepsilon(x) \nabla v_{i,\varepsilon}(x) dx\\=&\lambda_{i,\varepsilon}\Bigg\{\left(2^{*}_\mu-1\right)\displaystyle{\int_{\Omega}}u_\varepsilon^{2^{*}_\mu-2}(x)\omega_{\varepsilon}(x)v_{i,\varepsilon}(x)\left(\displaystyle{\int_{\Omega}}\frac{u_\varepsilon^{2^{*}_\mu}(y)}{|x-y|^{\mu}}dy\right)dx+\varepsilon  \displaystyle{\int_{\Omega}} \omega_\varepsilon(x) v_{i,\varepsilon}(x) dx\\
          &\quad\,\,\quad\,\,+2^{*}_\mu \displaystyle{\int_{\Omega}}u_\varepsilon^{2^{*}_\mu-1}(x)\omega_{\varepsilon}(x)\left(\displaystyle{\int_{\Omega}}\frac{u_\varepsilon^{2^{*}_\mu-1}(y)v_{i,\varepsilon}(y)}{|x-y|^{\mu}}dy\right)dx\Bigg\}\\
          &+\displaystyle{\int_{\partial\Omega}}\frac{\partial v_{i,\varepsilon}(x)}{\partial \nu}\frac{\partial u_{\varepsilon}(x)}{\partial \nu}(x-z)\cdot \nu\,dS_{x},
          \end{split}
      \end{equation}
where we have used the fact that $u_\varepsilon=0$ on $\partial\Omega$. Thus, due to the symmetry property of the following double integrals:
\begin{align*}
\displaystyle{\int_{\Omega}}u_\varepsilon^{2^{*}_\mu-1}(x)\omega_{\varepsilon}(x)\left(\displaystyle{\int_{\Omega}}\frac{u_\varepsilon^{2^{*}_\mu-1}(y)v_{i,\varepsilon}(y)}{|x-y|^{\mu}}dy\right)dx=\displaystyle{\int_{\Omega}}u_\varepsilon^{2^{*}_\mu-1}(x)v_{i,\varepsilon}(x)\left(\displaystyle{\int_{\Omega}}\frac{u_\varepsilon^{2^{*}_\mu-1}(y)\omega_{\varepsilon}(y)}{|x-y|^{\mu}}dy\right)dx,
\end{align*}
\eqref{eq 2.15} can be obtained from \eqref{eq-2.17} and \eqref{eq-2.18}.
\end{proof}
\end{lemma}
Combining with the above identity in Lemma \ref{lema 2.10}, we deduce the precise estimate of eigenvalues to \eqref{eq 1.12}, which plays a crucial role in analyzing the asymptotic behavior of eigenfunctions in Section 4 and Section 5.
\begin{lemma}\label{lema 2.11}
For any $i\in \mathbb{N}$, suppose that $\lambda_{i}=\underset{\varepsilon\rightarrow0}{\lim}\lambda_{i,\varepsilon}=1$ and $b_{i}\neq 0$ in \eqref{eq 2.11}. Then there exists $\kappa>0$ such that
    \begin{align}\label{eq 2.17}
\lambda_{i,\varepsilon}-1=&\frac{1}{\|u_\varepsilon\|_{\infty}^{2}}\big(\kappa+o(1)\big),
    \end{align}
where
\begin{align*}
    \kappa:=\frac{2(N-4)\sigma_{N}^{2}R(x_0)}{(N+2)N^{2}\left(\displaystyle{\int_{\R^N}}\frac{(1-|x|^{2})^{2}}{(1+x^{2})^{N+2}}dx\right)}>0.
\end{align*}
\end{lemma}
\begin{proof}
   Taking $z=x_\varepsilon$ in \eqref{eq 2.15}, we find that
   \begin{align*}
       \text{LHS of \eqref{eq 2.15}}=&\displaystyle{\int_{\partial\Omega}}\frac{\partial u_{\varepsilon}}{\partial \nu}\,\frac{\partial v_{i,\varepsilon}}{\partial \nu}\,(x-x_\varepsilon)\cdot \nu \,dS_{x}\\
    =&\frac{1}{\|u_\varepsilon\|_{\infty}^{3}}\,\displaystyle{\int_{\partial\Omega}}\,\frac{\partial }{\partial \nu}\Big(\|u_\varepsilon\|_{\infty}u_{\varepsilon}(x)\Big)\,\frac{\partial }{\partial \nu}\Big(\|u_\varepsilon\|_{\infty}^{2}v_{i,\varepsilon}(x)\Big)\, (x-x_\varepsilon)\cdot \nu \,dS_{x}\\
    =&\frac{1}{\|u_\varepsilon\|_{\infty}^{3}}\left(-\,\frac{b_{i}\sigma_{N}^{2}(N-2)}{N^{2}}R(x_0)+o(1)\right).
   \end{align*}
Now we write
   \begin{align*}
       \text{RHS of \eqref{eq 2.15}}=
       &\underbrace{\left(1-\lambda_{i,\varepsilon}\right)\left(2^{*}_\mu-1\right)
       \int_{\Omega}u_\varepsilon^{2^{*}_\mu-2}(x)\omega_{\varepsilon}(x)v_{i,\varepsilon}(x)
       \left(\int_{\Omega}\frac{u_\varepsilon^{2^{*}_\mu}(y)}{|x-y|^{\mu}}dy\right)dx}_{:=I_{1,\varepsilon}}\\
       &+\underbrace{\left(1-\lambda_{i,\varepsilon}\right)2^{*}_\mu\int_{\Omega}
       \int_{\Omega}\frac{u_\varepsilon^{2^{*}_\mu-1}(x)v_{i,\varepsilon}(x)u_\varepsilon^{2^{*}_\mu-1}(y)\omega_{\varepsilon}(y)}{|x-y|^{\mu}}dxdy}_{:=I_{2,\varepsilon}}\\
       &+\underbrace{\left(1-\lambda_{i,\varepsilon}\right)\varepsilon  \int_{\Omega} \omega_\varepsilon(x) v_{i,\varepsilon}(x)dx}_{:=I_{3,\varepsilon}}+\underbrace{2\varepsilon \int_{\Omega} u_\varepsilon(x)v_{i,\varepsilon}(x) dx}_{:=I_{4,\varepsilon}}.
\end{align*}
Then there holds
\begin{align*}
    I_{1,\varepsilon}=&\left(1-\lambda_{i,\varepsilon}\right)\left(2^{*}_\mu-1\right)\frac{1}{\|u_\varepsilon\|_{\infty}}
    \int_{\Omega_\varepsilon}\tilde{u}_\varepsilon^{2^{*}_\mu-2}(x)\tilde{\omega}_{\varepsilon}(x)\tilde{v}_{i,\varepsilon}(x)\left(\int_{\Omega_{\varepsilon}}\frac{\tilde{u}_\varepsilon^{2^{*}_\mu}(y)}{|x-y|^{\mu}}dy\right)dx\\
     =&\frac{(N-2)\left(2^{*}_\mu-1\right)\left(1-\lambda_{i,\varepsilon}\right)\,b_{i}}{2\,\|u_\varepsilon\|_{\infty}}
     \left(\displaystyle{\int_{\R^N}}U_{0,1}^{2^{*}_\mu-2}(x)\,\frac{(1-|x|^{2})^{2}}{(1+x^{2})^{N}}
     \left(\displaystyle{\int_{\R^N}}\frac{U_{0,1}^{2^{*}_\mu}(y)}{|x-y|^{\mu}}dy\right)dx+o(1)\right)\\
    =&\frac{(N+2-\mu)\left(1-\lambda_{i,\varepsilon}\right)\,b_{i}}{2\,\|u_\varepsilon\|_{\infty}}
    \left(\displaystyle{\int_{\R^N}}U_{0,1}^{2^{*}-2}(x)\,\frac{(1-|x|^{2})^{2}}{(1+x^{2})^{N}}\,dx+o(1)\right),
\end{align*}
where $\tilde{\omega}_{\varepsilon}(x)=\tau_{\varepsilon}^{-\frac{N-2}{2}}{\omega}_{\varepsilon}\left(\tau_{\varepsilon}^{-1}x+x_{\varepsilon}\right)$ is denoted as the rescaled function of $\omega_{\varepsilon}$. Moreover,
\begin{align*}
    I_{2,\varepsilon}=&\left(1-\lambda_{i,\varepsilon}\right)\frac{2^{*}_\mu}{\|u_\varepsilon\|_{\infty}}\int_{\Omega_\varepsilon}\int_{\Omega_\varepsilon}\frac{\tilde{u}_\varepsilon^{2^{*}_\mu-1}(x)\tilde{v}_{i,\varepsilon}(x)\tilde{u}_\varepsilon^{2^{*}_\mu-1}(y)\tilde{\omega}_{\varepsilon}(y)}{|x-y|^{\mu}}dxdy\\
          =&\frac{2^{*}_\mu(N-2)\left(1-\lambda_{i,\varepsilon}\right)\,b_{i}}{2\,\|u_\varepsilon\|_{\infty}}\left(\displaystyle{\int_{\R^N}}U_{0,1}^{2^{*}_\mu-1}(x)\,\frac{1-|x|^{2}}{(1+x^{2})^{\frac{N}{2}}}\left(\displaystyle{\int_{\R^N}}\frac{U_{0,1}^{2^{*}_\mu-1}(y)\,\frac{1-|y|^{2}}{(1+y^{2})^{\frac{N}{2}}}}{|x-y|^{\mu}}dy\right)dx+o(1)\right)\\
\underset{(*)}{=}&\frac{\mu\left(1-\lambda_{i,\varepsilon}\right)\,b_{i}}{2\,\|u_\varepsilon\|_{\infty}}\left(\displaystyle{\int_{\R^N}}U_{0,1}^{2^{*}-2}(x)\,\frac{(1-|x|^{2})^{2}}{(1+x^{2})^{N}}\,dx+o(1)\right),
\end{align*}
where $(*)$ holds from \eqref{a-2}. Similarly, we find
\begin{align*}
I_{3,\varepsilon}
        =& \frac{\varepsilon(N-2)\left(1-\lambda_{i,\varepsilon}\right)\,b_{i}}{2\,\|u_\varepsilon\|_{\infty}^{\frac{N+2}{N-2}}} \left(\displaystyle{\int_{\R^N}}\frac{(1-|x|^{2})^{2}}{(1+x^{2})^{N}}\,dx+o(1)\right),
\end{align*}
and
\begin{align*}
    I_{4,\varepsilon}
    =\frac{2\varepsilon\, b_{i}}{\|u_\varepsilon\|_{\infty}^{\frac{N+2}{N-2}}}\left(\displaystyle{\int_{\R^N}}U_{0,1}(x)\,\frac{1-|x|^{2}}{(1+x^{2})^{\frac{N}{2}}}\,dx+o(1)\right).
\end{align*}
Thus, if $b_{i}\neq 0$, combining the above estimates of both sides in \eqref{eq 2.15} with \eqref{eq 2.4}, we derive that
\begin{equation}\label{2-20}
    \begin{split}
     &\frac{1}{\|u_\varepsilon\|_{\infty}^{3}}\left(-\,\frac{\sigma_{N}^{2}(N-2)}{N^{2}}R(x_0)-\frac{\sigma_{N}(N-2)R(x_0)}{a_{N}N^{2}}\displaystyle{\int_{\R^N}}U_{0,1}(x)\,\frac{1-|x|^{2}}{(1+x^{2})^{\frac{N}{2}}}dx+o(1)\right)\\
    =&\left(1-\lambda_{i,\varepsilon}\right)\Bigg(\frac{N+2}{2\,\|u_\varepsilon\|_{\infty}}\displaystyle{\int_{\R^N}}U_{0,1}^{2^{*}-2}(x)\,\frac{(1-|x|^{2})^{2}}{(1+x^{2})^{N}}dx+O\left(\frac{1}{\|u_\varepsilon\|^{3}_{\infty}}\right)\Bigg).
    \end{split}
\end{equation}
Hence, a direct calculation yields the desired result \eqref{eq 2.17}. Indeed, we define
\begin{alignat*}{2}
    &A:=\frac{N+2}{2}\displaystyle{\int_{\R^N}}U_{0,1}^{2^{*}-2}(x)\,\frac{(1-|x|^{2})^{2}}{(1+x^{2})^{N}}dx,\quad &&B:=-\frac{\sigma_{N}^{2}(N-2)R(x_0)}{N^{2}},\\
    &C:=\int_{\R^N}U_{0,1}(x)\,\frac{1-|x|^{2}}{(1+x^{2})^{\frac{N}{2}}}dx=-\sigma_{N}\frac{\Gamma\left(\frac{N}{2}\right)\Gamma\left(\frac{N}{2}-2\right)}{\Gamma\left(N-1\right)},\quad &&D:=\frac{\sigma_{N}(N-2)R(x_0)}{a_{N}N^{2}}.
\end{alignat*}
Then, from \eqref{2-20}, we derive that
\begin{align*}
\lambda_{i,\varepsilon}-1=\frac{1}{\|u_\varepsilon\|_{\infty}^{2}}\left(\frac{DC-B}{A}+o(1)\right),
\end{align*}
where
\begin{align*}
   DC-B:=&\left(\frac{\sigma_{N}}{N} \right)^{2}(N-2)R(x_0)\left(1-\frac{2\,\Gamma(N-2)}{\Gamma(N-1)}\right)\\
   =&\left(\frac{\sigma_{N}}{N} \right)^{2}(N-2)\,R(x_0)\left(1-\frac{2}{N-2}\right)=\left(\frac{\sigma_{N}}{N} \right)^{2}(N-4)\,R(x_0)>0,\quad\text{if $N>4$.}
\end{align*}
Hence, a simple computation yields that
\begin{align*}
    \kappa:=\frac{DC-B}{A}=\frac{2(N-4)\sigma_{N}^{2}R(x_0)}{(N+2)N^{2}\left(\displaystyle{\int_{\R^N}}\frac{(1-|x|^{2})^{2}}{(1+x^{2})^{N+2}}dx\right)}>0.
\end{align*}
\end{proof}

\noindent
\section{\label{Proof of Thm1.1}Estimates for the first eigenpair $\big(\lambda_{1,\varepsilon},v_{1,\varepsilon}\big)$}
\setcounter{equation}{0}
\vskip 0.2cm
In this section, we prove Theorem 1.1 by estimating the first eigenpair $\big(\lambda_{1,\varepsilon},v_{1,\varepsilon}\big)$.
\begin{proof}[\textbf{Proof of Theorem \ref{th1.1}}]
By the variational characterization of $\lambda_{1,\varepsilon}$, we have
    \begin{align*}
\lambda_{1,\varepsilon}
=\!\underset{v\in H_{0}^{1}(\Omega)\backslash\{0\}}{\inf}\frac{\displaystyle{\int_{\Omega}}|\nabla v|^{2}dx}{\displaystyle{\int_{\Omega}}\varepsilon v^{2}\!+\!(2^{*}_\mu-1)u_{\varepsilon}^{2^{*}_\mu-2}v^{2}\!\left(\!\displaystyle{\int_{\Omega}}\frac{u^{2^{*}_\mu}_{\varepsilon}(y)}{|x-y|^{\mu}}dy\!\right)\!+2^{*}_{\mu} u_{\varepsilon}^{2^{*}_\mu-1}v\left(\!\displaystyle{\int_{\Omega}}\frac{u^{2^{*}_\mu-1}_{\varepsilon}(y)v(y)}{|x-y|^{\mu}}dy\!\right)\!dx}.
        \end{align*}
Setting $v=u_{\varepsilon}$, a direct calculation yields that
\begin{align*}
    \lambda_{1,\varepsilon}\leq& \frac{\displaystyle{\int_{\Omega}}|\nabla u_{\varepsilon}|^{2}dx}{\varepsilon\displaystyle{\int_{\Omega}} u_{\varepsilon}^{2}dx+(2\cdot2^{*}_\mu-1)\displaystyle{\int_{\Omega}}\displaystyle{\int_{\Omega}}\frac{u_{\varepsilon}^{2^{*}_\mu}(x)u_{\varepsilon}^{2^{*}_\mu}(y)}{|x-y|^{\mu}}dxdy}\\
    =&\frac{\varepsilon \displaystyle{\int_{\Omega}} u_{\varepsilon}^{2}dx+\displaystyle{\int_{\Omega}}\displaystyle{\int_{\Omega}}\frac{u_{\varepsilon}^{2^{*}_\mu}(x)u_{\varepsilon}^{2^{*}_\mu}(y)}{|x-y|^{\mu}}dxdy}{\varepsilon\displaystyle{\int_{\Omega}} u_{\varepsilon}^{2}dx+(2\cdot2^{*}_\mu-1)\displaystyle{\int_{\Omega}}\displaystyle{\int_{\Omega}}\frac{u_{\varepsilon}^{2^{*}_\mu}(x)u_{\varepsilon}^{2^{*}_\mu}(y)}{|x-y|^{\mu}}dxdy}\\
    =&\frac{\varepsilon \|u_{\varepsilon}\|_{\infty}^{2-2^{*}}\displaystyle{\int_{\Omega_{\varepsilon}}} \tilde{u}_{\varepsilon}^{2}dx+\displaystyle{\int_{\Omega_{\varepsilon}}}\displaystyle{\int_{\Omega_{\varepsilon}}}\frac{\tilde{u}_{\varepsilon}^{2^{*}_\mu}(x)\tilde{u}_{\varepsilon}^{2^{*}_\mu}(y)}{|x-y|^{\mu}}dxdy}{\varepsilon \|u_{\varepsilon}\|_{\infty}^{2-2^{*}}\displaystyle{\int_{\Omega_{\varepsilon}}} \tilde{u}_{\varepsilon}^{2}dx+(2\cdot2^{*}_\mu-1)\displaystyle{\int_{\Omega_{\varepsilon}}}\displaystyle{\int_{\Omega_{\varepsilon}}}\frac{\tilde{u}_{\varepsilon}^{2^{*}_\mu}(x)\tilde{u}_{\varepsilon}^{2^{*}_\mu}(y)}{|x-y|^{\mu}}dxdy}\\
    =&\frac{\displaystyle{\int_{\R^{N}}}\displaystyle{\int_{\R^{N}}}\frac{{U}_{0,1}^{2^{*}_\mu}(x){U}_{0,1}^{2^{*}_\mu}(y)}{|x-y|^{\mu}}dxdy+o(1)}{(2\cdot2^{*}_\mu-1)\displaystyle{\int_{\R^{N}}}\displaystyle{\int_{\R^{N}}}\frac{{U}_{0,1}^{2^{*}_\mu}(x){U}_{0,1}^{2^{*}_\mu}(y)}{|x-y|^{\mu}}dxdy+o(1)}
\end{align*}
as $\varepsilon\rightarrow0$, which implies that
$$\underset{\varepsilon\rightarrow 0}{\limsup}\,{\lambda_{1,\varepsilon}}\leq \frac{1}{2\cdot2^{*}_\mu-1}.$$
Hence, by choosing a subsequence, we may assume that $\lambda_{1,\varepsilon}\rightarrow \lambda_{1}\in\left[\,0,(2\cdot2^{*}_\mu-1)^{-1}\right]$. On the other hand, since $\tilde{v}_{1,\varepsilon}$ satisfies \eqref{eq 1.13}, then using similar arguments as in the proof of \eqref{2-13}, we know that $\tilde{v}_{1,\varepsilon}$ is bounded in $D^{1,2}(\R^{N})$. Up to a subsequence, there exists $0\not\equiv V_1\in\mathcal{D}^{1,2}(\R^{N})$ such that $\tilde{v}_{1,\varepsilon} \rightharpoonup  V_{1}$ in $\mathcal{D}^{1,2}(\R^N)$  and $\tilde{v}_{1,\varepsilon}\rightarrow V_1$ in $C^{1}_{loc}(\R^N)$, where $V_1$ satisfies
\begin{equation*}
 -\Delta V_1=\lambda_{1}\left\{(2^{*}_\mu-1) {U}_{0,1}^{2^{*}_\mu-2}V_1\left(\displaystyle{\int_{\R^{N}}}\frac{{U}_{0,1}^{2^{*}_\mu}(y)}{|x-y|^{\mu}}dy\right)+2^{*}_{\mu} {U}_{0,1}^{2^{*}_\mu-1}\left(\displaystyle{\int_{\R^{N}}}\frac{{U}_{0,1}^{2^{*}_\mu-1}(y)V_{1}(y)}{|x-y|^{\mu}}dy\right)\right\},\quad \text{in $\R^N$}.
\end{equation*}
Since there exists no eigenvalue $\lambda$ less than $1/(2\cdot2^{*}_\mu-1)$ from Lemma
\ref{lma2.5}, we obtain that
$$\lambda_{1}=\frac{1}{2\cdot2^{*}_\mu-1}\quad\text{and}\quad V_{1}={U}_{0,1},$$ which implies that \eqref{eq 1.16} and \eqref{eq 1.17} in Theorem \ref{th1.1} hold.
 Moreover, we derive that $\lambda_{1,\varepsilon}$ is simple. In fact, suppose on the contrary, there exist at least two eigenfunctions $v^{(1)}_{1,\varepsilon}$ and $v^{(2)}_{1,\varepsilon}$,  corresponding to $\lambda_{1,\varepsilon}$, which are orthogonal in the sense of \eqref{eqq1-11}. Since the rescaled function $\tilde{v}^{(1)}_{1,\varepsilon}$ and $\tilde{v}^{(2)}_{1,\varepsilon}$ converge to ${U}_{0,1}$ respectively  for $\varepsilon>0$ small, we derive that
 \begin{align*}
     0=&(2^{*}_\mu-1)\displaystyle{\int_{\Omega}}u^{2^{*}_\mu-2}_{\varepsilon}(x)\,v^{(1)}_{1,\varepsilon}(x)\,v^{(2)}_{1,\varepsilon}(x)\left(\displaystyle{\int_{\Omega}}\frac{u^{2^{*}_\mu}_{\varepsilon}(y)}{|x-y|^{\mu}}dy\right)dx+\varepsilon \displaystyle{\int_{\Omega}}v^{(1)}_{1,\varepsilon}(x)\,v^{(2)}_{1,\varepsilon}(x)\,dx\\
&+2^{*}_\mu\displaystyle{\int_{\Omega}} u^{2^{*}_\mu-1}_{\varepsilon}(x)\,v^{(2)}_{1,\varepsilon}(x)\,\left(\displaystyle{\int_{\Omega}}\frac{u^{2^{*}_\mu-1}_{\varepsilon}(y)v^{(1)}_{1,\varepsilon}(y)}{|x-y|^{\mu}}dy\right)dx\\
=&(2^{*}_\mu-1)\displaystyle{\int_{\Omega_{\varepsilon}}}\tilde{u}^{2^{*}_\mu-2}_{\varepsilon}(x)\,\tilde{v}^{(1)}_{1,\varepsilon}(x)\,\tilde{v}^{(2)}_{1,\varepsilon}(x)\,\left(\displaystyle{\int_{\Omega_{\varepsilon}}}\frac{\tilde{u}^{2^{*}_\mu}_{\varepsilon}(y)}{|x-y|^{\mu}}dy\right)dx+ \frac{\varepsilon}{\|u_\varepsilon\|_{\infty}^{2^{*}-2}}\displaystyle{\int_{\Omega_{\varepsilon}}}\tilde{v}^{(1)}_{1,\varepsilon}(x)\,\tilde{v}^{(2)}_{1,\varepsilon}(x)\,dx\\
&+2^{*}_\mu\displaystyle{\int_{\Omega_{\varepsilon}}} \tilde{u}^{2^{*}_\mu-1}_{\varepsilon}(x)\,\tilde{v}^{(2)}_{1,\varepsilon}(x)\,\left(\displaystyle{\int_{\Omega_{\varepsilon}}}\frac{\tilde{u}^{2^{*}_\mu-1}_{\varepsilon}(y)\tilde{v}_{1,\varepsilon}(y)}{|x-y|^{\mu}}dy\right)dx\\
\longrightarrow&(2^{*}_\mu-1)\displaystyle{\int_{\R^N}}{U}_{0,1}^{2^{*}_\mu}(x)\left(\displaystyle{\int_{\R^N}}\frac{{U}_{0,1}^{2^{*}_\mu}(y)}{|x-y|^{\mu}}dy\right)dx+2^{*}_\mu\displaystyle{\int_{\R^N}}{U}_{0,1}^{2^{*}_\mu}(x)\left(\displaystyle{\int_{\R^N}}\frac{{U}_{0,1}^{2^{*}_\mu}(y)}{|x-y|^{\mu}}dy\right)dx\\
    =&(2\cdot2^{*}_\mu-1)\displaystyle{\int_{\R^N}}\displaystyle{\int_{\R^N}}\frac{{U}_{0,1}^{2^{*}_\mu}(x){U}_{0,1}^{2^{*}_\mu}(y)}{|x-y|^{\mu}}\,dxdy=(2\cdot2^{*}_\mu-1)\displaystyle{\int_{\R^N}}{U}_{0,1}^{2^{*}}(x)\,dx,
 \end{align*}
which is a contradiction. Thus, $\lambda_{1,\varepsilon}$ is simple. Finally, we prove that \eqref{eq 1.18} holds. A simple computation yields that the function $\phi_{1,\varepsilon}:=\|u_\varepsilon\|_{\infty}^{2}v_{1,\varepsilon}$ satisfies
\begin{equation}\label{eq 3.1}
    \begin{cases}
        -\Delta \phi_{1,\varepsilon}=f_{1,\varepsilon}(x),\quad &x\in\Omega,\\
     \phi_{1,\varepsilon}=0,\quad &\text{on $\partial\Omega$},
     \end{cases}
 \end{equation}
 where
 \begin{align*}
f_{1,\varepsilon}(x):=&(2^{*}_\mu-1)\,\lambda_{1,\varepsilon}\|u_\varepsilon\|_{\infty}^{2}u^{2^{*}_\mu-2}_{\varepsilon}(x)v_{1,\varepsilon}(x)\left(\displaystyle{\int_{\Omega}}\frac{u^{2^{*}_\mu}_{\varepsilon}(y)}{|x-y|^{\mu}}dy\right)+\varepsilon\,\lambda_{1,\varepsilon}\|u_\varepsilon\|_{\infty}^{2} v_{1,\varepsilon}(x)\nonumber\\
&+2^{*}_\mu\,\lambda_{1,\varepsilon}\|u_\varepsilon\|_{\infty}^{2} u^{2^{*}_\mu-1}_{\varepsilon}(x)\left(\displaystyle{\int_{\Omega}}\frac{u^{2^{*}_\mu-1}_{\varepsilon}(y)v_{1,\varepsilon}(y)}{|x-y|^{\mu}}dy\right).
 \end{align*}
 Since $\lambda_{1,\varepsilon}\rightarrow 1/(2\cdot2^{*}_\mu-1)$ and $\tilde{v}_{1,\varepsilon}\rightarrow U_{0,1}$ in $C^{1}_{loc}(\R^N)$ as $\varepsilon\rightarrow0$, the same arguments as in the proof of \eqref{eq 2.3} implies that $f_{1,\varepsilon}\in L^{1}(\Omega)$ and $f_{1,\varepsilon}(x)\rightarrow0$ for any $x\in\bar{\Omega}\backslash\{x_0\}$ as $\varepsilon\rightarrow0$. Moreover, $$\displaystyle{\int_{\Omega}}f_{1,\varepsilon}(x)\psi(x)dx\rightarrow\frac{\sigma_{N}}{N}\psi(x_{0}),\quad\forall\,\psi(x)\in C_{0}^{\infty}(\Omega).$$
Thus, we obtain that $f_{1,\varepsilon}\rightarrow\frac{\sigma_{N}}{N}\delta_{x_0}$ in the sense of distributions. By Lemma \ref{lema 2.3} and standard elliptic estimates, we derive \eqref{eq 1.18}. This completes the proof of Theorem \ref{th1.1}.
 \end{proof}

\noindent
\section{\label{Proof of Thm1.2}Estimates for the eigenpairs $\left\{\big(\lambda_{i,\varepsilon},v_{i,\varepsilon}\big)\right\}_{i=2}^{N+1}$}
\setcounter{equation}{0}
\vskip 0.2cm
 Let $u_\varepsilon$ be a solution of \eqref{eq 1.1} satisfying \eqref{eq 1.2}, from \eqref{eq 1.8}, we know that, for $\varepsilon>0$ small, there exist $x_\varepsilon\in \Omega$ and $\tau_\varepsilon\in\R^{+}$ such that $\|u_\varepsilon\|_{L^{\infty}(\Omega)}=u_\varepsilon(x_\varepsilon)=\tau_{\varepsilon}^{\frac{N-2}{2}}\rightarrow+\infty$ and $x_\varepsilon\rightarrow x_0\in \Omega$, where $x_0$ is the blow-up point. Noting that the point $x_0$ is an interior point of $\Omega$, we may assume that there exists $\rho>0$ such that $B(x_\varepsilon,2\rho)\subseteq \Omega$ for any $\varepsilon>0$ sufficiently small. Define
\begin{align}\label{eq 4.1}
    \phi(x)=\bar{\phi}(x-x_\varepsilon),
\end{align}
where $\bar{\phi}\in C_0^{\infty}\left(B(0,2\rho)\right)$ satisfying $\bar{\phi}\equiv 1$ in $B(0,\rho)$, $0\leq\bar{\phi}\leq1$ in $B(0,2\rho)$. Denote
\begin{align}
    \psi_{i,\varepsilon}(x)&=\phi(x)\,\frac{\partial u_{\varepsilon}(x)}{\partial x_i},\quad i=1,\cdots,N,\label{eq 4.2}\\
    \psi_{N+1,\varepsilon}(x)&=\phi(x)\left((x-x_{\varepsilon})\cdot\nabla u_{\varepsilon}(x)+\frac{N-2}{2}u_{\varepsilon}(x)\right).\label{eq 4.3}
\end{align}

Using similar arguments as Lemma 2.7 in \cite{Ta3}, we obtain the following result.
\begin{lemma}\label{lema 4.1}
For $\varepsilon>0$ sufficiently small, the functions
$$u_\varepsilon,\, \psi_{1,\varepsilon}, \cdots, \,\psi_{N+1,\varepsilon}$$
are linearly independent in $H^{1}_{0}(\Omega)$.
\end{lemma}
\begin{proof}
Assume by contradiction that there exist $\alpha_{0,\varepsilon},\,\alpha_{1,\varepsilon},\cdots,\,\alpha_{N+1,\varepsilon}$ such that $\sum_{i=0}^{N+1}\alpha_{i,\varepsilon}^{2}\neq 0$ and
\begin{align}\label{eq 4.4}
\alpha_{0,\varepsilon}\,u_\varepsilon+\sum_{i=1}^{N+1}\alpha_{i,\varepsilon}\,\psi_{i,\varepsilon}\equiv 0,\quad \forall\, x\in\Omega.
\end{align}
WLOG, we can assume that
\begin{align}\label{eq 4.5}
\sum_{i=0}^{N+1}\alpha_{i,\varepsilon}^{2}\equiv 1,\quad \text{for any $\varepsilon>0$ small.}
\end{align}

 Firstly, we claim that $\alpha_{0,\varepsilon}=0$. If not, from \eqref{eq 4.4}, we derive that
\begin{align}\label{eq 4.6}
u_\varepsilon(x)=\sum_{i=1}^{N+1}\beta_{i,\varepsilon}\psi_{i,\varepsilon}(x),\quad \beta_{i,\varepsilon}:=-\frac{\alpha_{i,\varepsilon}}{\alpha_{0,\varepsilon}}.
\end{align}
Inserting $x=x_{\varepsilon}$ in \eqref{eq 4.6}, since $\phi(x_\varepsilon)=1$ and $\nabla u_\varepsilon(x_\varepsilon)=0$, we have
\begin{align*}
    u_{\varepsilon}(x_\varepsilon)=\frac{(N-2)\beta_{N+1,\varepsilon}}{2}\, u_{\varepsilon}(x_\varepsilon),
\end{align*}
which implies that
\begin{align}\label{eq 4.7}
\beta_{N+1,\varepsilon}=\frac{2}{N-2}>0,\quad\text{if $\alpha_{0,\varepsilon}\neq0$.}
\end{align}
Moreover, note that $\phi(x)\equiv1,\,\forall \,x\in B_{\rho}(x_\varepsilon)$. Then from \eqref{eqA-4} and \eqref{eqA-5}, we obtain that for any $x\in B_{\rho}(x_\varepsilon)$,
\begin{equation}\label{4-8}
\begin{split}
-\Delta\left(\sum_{i=1}^{N+1}\beta_{i,\varepsilon}\psi_{i,\varepsilon}(x)\right)=&\varepsilon\left(\sum_{i=1}^{N+1}\beta_{i,\varepsilon}\psi_{i,\varepsilon}(x)\right)+2\varepsilon \beta_{N+1,\varepsilon}u_\varepsilon(x)\\&+(2^{*}_\mu-1)u_\varepsilon^{2^{*}_\mu-2}(x)\left(\sum_{i=1}^{N+1}\beta_{i,\varepsilon}\psi_{i,\varepsilon}(x)\right)\left(\displaystyle{\int_{\Omega}}\frac{u_\varepsilon^{2^{*}_\mu}(y)}{|x-y|^{\mu}}dy\right)\\
 &+2^{*}_\mu u_\varepsilon^{2^{*}_\mu-1}(x)\left(\displaystyle{\int_{\Omega}}\frac{u_\varepsilon^{2^{*}_\mu-1}(y)\left(\sum_{i=1}^{N+1}\beta_{i,\varepsilon}\psi_{i,\varepsilon}(y)\right)}{|x-y|^{\mu}}dy\right).
\end{split}
    \end{equation}
On the other hand, since $u_\varepsilon$ is the solution to \eqref{eq 1.1}, combining with \eqref{4-8}, we derive that for any $x\in B_{\rho}(x_\varepsilon)$ and $\alpha_{0,\varepsilon}\neq0$, the following equality holds
\begin{align*}
 2(1-2^{*}_\mu)\,u_\varepsilon^{2^{*}_\mu-1}(x)\left(\displaystyle{\int_{\Omega}}\frac{u_\varepsilon^{2^{*}_\mu}(y)}{|x-y|^{\mu}}dy\right)=2\varepsilon \beta_{N+1,\varepsilon}\,u_\varepsilon(x),
\end{align*}
which implies that
     \begin{align*}
\beta_{N+1,\varepsilon}=\frac{(1-2^{*}_\mu)\,u_\varepsilon^{2^{*}_\mu-2}(x)\left(\displaystyle{\int_{\Omega}}\frac{u_\varepsilon^{2^{*}_\mu}(y)}{|x-y|^{\mu}}dy\right)}{\varepsilon },\quad\forall\,x\in B_{\rho}(x_\varepsilon).
     \end{align*}
Recalling \eqref{eq 4.7}, we get a contradiction since $(1-2^{*}_\mu)<0$ when $N\geq6$ and $\mu\in(0,4)$. Hence, we conclude that $\alpha_{0,\varepsilon}=0$.

Next, we prove that $\alpha_{N+1,\varepsilon}=0$. Indeed, from the first step, we know that \eqref{eq 4.4} reduces into
\begin{align*}
\sum_{i=1}^{N}\alpha_{i,\varepsilon}\psi_{i,\varepsilon}+\alpha_{N+1,\varepsilon}\psi_{N+1,\varepsilon}\equiv 0,\quad \forall\, x\in\Omega.
\end{align*}
Then, setting $x=x_\varepsilon$, since $\nabla u_\varepsilon(x_\varepsilon)=0$ and $\phi(x_\varepsilon)=1$, we have
$$\frac{(N-2)\alpha_{N+1,\varepsilon}}{2}\,u_\varepsilon(x_\varepsilon)=0,$$
which implies that $\alpha_{N+1,\varepsilon}=0$. Now, we derive that for any $x\in\Omega$, $\sum_{i=1}^{N}\alpha_{i,\varepsilon}\psi_{i,\varepsilon}(x) \equiv 0$. By scaling, we have
$$\sum_{i=1}^{N}\alpha_{i,\varepsilon}\,{\phi}_{\varepsilon}(x)\,\frac{\partial\tilde{u}_{\varepsilon}(x)}{\partial x_{i}} \equiv 0,\quad \forall\,x\in\Omega_{\varepsilon},$$
where $\phi_{\varepsilon}(x):=\phi\left(\tau_{\varepsilon}^{-1}{x}+x_\varepsilon\right)$. From \eqref{eq 1.9}, we know that the rescaled function $\tilde{u}_{\varepsilon}$ converges to the function ${U}_{0,1}$ in $C^{1}_{loc}(\R^N)$, which implies that
\begin{align*}
\sum_{i=1}^{N}\alpha_{i}\frac{\partial{U}_{0,1}(x)}{\partial x_i}\equiv 0,\quad \forall\,x\in \R^N,
\end{align*}
where $\alpha_{i}=\underset{\varepsilon\rightarrow0}{\lim}\alpha_{i,\varepsilon}$. Since the functions $\frac{\partial{U}_{0,1}}{\partial x_i}$ are linearly independent, we obtain that  $\alpha_i=0$ for $i=1,\cdots,\,N$, which is a contradiction with \eqref{eq 4.5}. Thus, we complete the proof.
\end{proof}
\vskip 0.2cm
In the following, we prove Theorem \ref{th1.2}.
\smallskip
\begin{lemma}\label{lema 4.2}
    For $i=2,\cdots,\,N+1$, we have
    \begin{equation}\label{eq 4.13}
\lambda_{i,\varepsilon}\leq1+O\left(\frac{1}{\|u_{\varepsilon}\|_{\infty}^{2^{*}}}\right),
    \end{equation}
and $\underset{\varepsilon\rightarrow 0}{\lim}\,\lambda_{i,\varepsilon}=1$.
\end{lemma}
\begin{proof}
    By the variational characterization, $\lambda_{i,\varepsilon}$ can be expressed as
   \begin{equation*}
\lambda_{i,\varepsilon}\!=\!\mathop{\inf}_{W\subseteq H_{0}^{1}(\Omega)\atop dimW=i}\!{\underset{v\in W}{\max}}\frac{\displaystyle{\int_{\Omega}}|\nabla v|^{2}dx}{\displaystyle{\int_{\Omega}}\varepsilon v^{2}\!+\!(2^{*}_\mu-1)u_{\varepsilon}^{2^{*}_\mu-2}v^{2}\!\left(\!\displaystyle{\int_{\Omega}}\frac{u^{2^{*}_\mu}_{\varepsilon}(y)}{|x-y|^{\mu}}dy\!\right)\!+2^{*}_{\mu} u_{\varepsilon}^{2^{*}_\mu-1}v\left(\!\displaystyle{\int_{\Omega}}\frac{u^{2^{*}_\mu-1}_{\varepsilon}(y)v(y)}{|x-y|^{\mu}}dy\!\right)\!dx}.
\end{equation*}
Take
$$W_i=span\big\{u_\varepsilon,\,\psi_{1,\varepsilon},\cdots,\,\psi_{i-1,\varepsilon}\big\},$$
where $\psi_{j,\varepsilon}$ are defined in \eqref{eq 4.2}. By Lemma \ref{lema 4.1}, we obtain that $dim W_i=i$. Hence, we have
\begin{align*}
\lambda_{i,\varepsilon}\!\leq{\underset{v\in W_{i}}{\max}}\frac{\displaystyle{\int_{\Omega}}|\nabla v|^{2}dx}{\displaystyle{\int_{\Omega}}\varepsilon v^{2}\!dx+\!\displaystyle{\int_{\Omega}}\displaystyle{\int_{\Omega}}\!\left(\frac{(2^{*}_\mu-1)u_{\varepsilon}^{2^{*}_\mu-2}(x)v^{2}(x)u_{\varepsilon}^{2^{*}_\mu}(y)}{|x-y|^{\mu}}+\!\frac{2^{*}_\mu u_{\varepsilon}^{2^{*}_\mu-1}(x)v(x)u_{\varepsilon}^{2^{*}_\mu-1}(y)v(y)}{|x-y|^{\mu}}\!\right)\!dxdy}.
\end{align*}
Let us evaluate separately the numerator and the denominator. For any function $v\in W_i$, we assume that there exist $a_0,\,a_1,\cdots,\,a_{i-1}\in \R$ such that
\begin{align}\label{eq 4.16}
    v=a_0 u_\varepsilon+\sum_{j=1}^{i-1}a_j\psi_{j,\varepsilon}=a_0 u_\varepsilon+\phi z_\varepsilon\in W_i,
\end{align}
where $z_\varepsilon(x)=\sum_{j=1}^{i-1}a_j\,\frac{\partial u_\varepsilon(x)}{\partial x_j}$. Thus, we have
\begin{equation}\label{eq4.15}
    \begin{split}
\displaystyle{\int_{\Omega}}|\nabla v|^{2}dx=&\displaystyle{\int_{\Omega}}|\nabla (a_0 u_\varepsilon+\phi z_\varepsilon)|^{2}\,dx\\
=&a_0^{2}\displaystyle{\int_{\Omega}}|\nabla u_\varepsilon|^{2}\,dx+2a_0\displaystyle{\int_{\Omega}}\nabla u_\varepsilon\nabla(\phi z_\varepsilon)\,dx+\displaystyle{\int_{\Omega}}|\nabla (\phi z_\varepsilon)|^{2}\,dx,
    \end{split}
\end{equation}
where a simple computation yields that
    \begin{equation}\label{eq4.17}
        \begin{split}
\displaystyle{\int_{\Omega}}\nabla u_{\varepsilon}\nabla(\phi z_\varepsilon)\,dx=& \displaystyle{\int_{\Omega}}(-\Delta u_\varepsilon)\,\phi\,z_\varepsilon\,dx\\
    =& \varepsilon \displaystyle{\int_{\Omega}}u_\varepsilon\,\phi\, z_\varepsilon\, dx+\displaystyle{\int_{\Omega}}u_{\varepsilon}^{2^{*}_\mu-1}(x)\phi(x) z_\varepsilon(x)\left(\displaystyle{\int_{\Omega}}\frac{u_{\varepsilon}^{2^{*}_\mu}(y)}{|x-y|^{\mu}}dy\right)dx.
        \end{split}
    \end{equation}
Moreover, from \eqref{eqA-4} in Lemma \ref{lma A-3}, we derive that $z_\varepsilon$ satisfies the equation
\begin{equation}\label{eq 4.18}
    \begin{split}
     -\Delta z_\varepsilon(x)=&\varepsilon z_\varepsilon(x)+(2^{*}_\mu-1)u_{\varepsilon}^{2^{*}_\mu-2}(x)z_\varepsilon(x)\left(\displaystyle{\int_{\Omega}}\frac{u_{\varepsilon}^{2^{*}_\mu}(y)}{|x-y|^{\mu}}dy\right)\\
    &+2^{*}_\mu u_{\varepsilon}^{2^{*}_\mu-1}(x)\left(\displaystyle{\int_{\Omega}}\frac{u_{\varepsilon}^{2^{*}_\mu-1}(y)z_\varepsilon(y)}{|x-y|^{\mu}}dy\right),\quad \text{in }\,\Omega.
    \end{split}
\end{equation}
Multiplying \eqref{eq 4.18} by $\phi^{2}z_\varepsilon$ and integrating, we have
\begin{equation}\label{eq 4.19}
\begin{split}
&\displaystyle{\int_{\Omega}} \phi^{2}|\nabla z_\varepsilon|^{2}dx+2\displaystyle{\int_{\Omega}} \phi z_\varepsilon {\nabla \phi} \nabla z_\varepsilon\, dx\\
=&\,\varepsilon\displaystyle{\int_{\Omega}} \phi^{2}(x)z_\varepsilon^{2}(x)\,dx+(2^{*}_\mu-1)\displaystyle{\int_{\Omega}}\displaystyle{\int_{\Omega}}\frac{u_{\varepsilon}^{2^{*}_\mu-2}(x)\phi^{2}(x)z_\varepsilon^{2}(x) u_{\varepsilon}^{2^{*}_\mu}(y)}{|x-y|^{\mu}}\,dxdy\\
&+2^{*}_\mu \displaystyle{\int_{\Omega}}\displaystyle{\int_{\Omega}}\frac{u_{\varepsilon}^{2^{*}_\mu-1}(x)\phi^{2}(x)z_\varepsilon (x)u_{\varepsilon}^{2^{*}_\mu-1}(y)z_\varepsilon(y)}{|x-y|^{\mu}}\,dxdy,
    \end{split}
\end{equation}
and then from \eqref{eq 4.19}, we have
\begin{equation}\label{eq4.16}
    \begin{split}  \displaystyle{\int_{\Omega}}|\nabla (\phi z_\varepsilon)|^{2}\,dx=&\displaystyle{\int_{\Omega}}|\nabla\phi|^{2}|z_\varepsilon|^{2}\,dx+\displaystyle{\int_{\Omega}}|\phi|^{2}|\nabla z_\varepsilon|^{2}\,dx+2\displaystyle{\int_{\Omega}}\phi z_\varepsilon\nabla\phi\nabla z_\varepsilon\,dx\\
=&\displaystyle{\int_{\Omega}}|\nabla\phi|^{2}|z_\varepsilon|^{2}\,dx+\varepsilon\displaystyle{\int_{\Omega}} \phi^{2}(x)z_\varepsilon^{2}(x)\,dx\\
&+(2^{*}_\mu-1)\displaystyle{\int_{\Omega}} \displaystyle{\int_{\Omega}}\frac{u_{\varepsilon}^{2^{*}_\mu-2}(x)\phi^{2}(x)z_\varepsilon^{2}(x) u_{\varepsilon}^{2^{*}_\mu}(y)}{|x-y|^{\mu}}\,dxdy\\
    &+2^{*}_\mu \displaystyle{\int_{\Omega}}\displaystyle{\int_{\Omega}}\frac{u_{\varepsilon}^{2^{*}_\mu-1}(x)\phi^{2}(x)z_\varepsilon (x)u_{\varepsilon}^{2^{*}_\mu-1}(y)z_\varepsilon(y)}{|x-y|^{\mu}}\,dxdy.
    \end{split}
\end{equation}

Therefore,  it holds
\begin{equation}\label{eq4.18}
  \lambda_{i,\varepsilon}
\leq\underset{\left(a_0,\,a_1,\cdots,\,a_{i-1}\right)\,\in\R^{i}}{\max}{\left\{1+\frac{N_\varepsilon}{D_\varepsilon}\right\}},
\end{equation}
and we postpone this proof to Appendix B.
Here
\begin{equation}\label{N-ep}
N_\varepsilon:=N_{1,\varepsilon}+N_{2,\varepsilon}+N_{3,\varepsilon},
\end{equation}
with
\begin{equation}\label{N-ep,1}
N_{1,\varepsilon}:=a_{0}^{2}\,\left(2-2\cdot2^{*}_\mu\right) \displaystyle{\int_{\Omega}}\displaystyle{\int_{\Omega}}\frac{u_{\varepsilon}^{2^{*}_\mu}(x)u_{\varepsilon}^{2^{*}_\mu}(y)}{|x-y|^{\mu}}\,dxdy,
\end{equation}
\smallskip
\begin{equation}\label{N-ep,2}
   N_{2,\varepsilon}:=2a_{0}\,\left(2-2\cdot2^{*}_\mu\right)\displaystyle{\int_{\Omega}}\displaystyle{\int_{\Omega}}\frac{u_{\varepsilon}^{2^{*}_\mu-1}(x)\phi(x)\left(\sum_{j=1}^{i-1}a_j\frac{\partial u_\varepsilon}{\partial x_j}(x)\right)u_{\varepsilon}^{2^{*}_\mu}(y)}{|x-y|^{\mu}}\,dxdy,
\end{equation}
\smallskip
\begin{equation}\label{N-ep,3}
\begin{split}
N_{3,\varepsilon}:&=\underbrace{\displaystyle{\int_{\Omega}}|\nabla\phi|^{2}\left(\sum_{j=1}^{i-1}a_j\frac{\partial u_\varepsilon}{\partial x_j}(x)\right)\left(\sum_{l=1}^{i-1}a_l\frac{\partial u_\varepsilon}{\partial x_l}(x)\right)dx}_{:=N_{3,\varepsilon}^{(1)}}\\
&+\!\underbrace{2^{*}_\mu\!\displaystyle{\int_{\Omega}}u_{\varepsilon}^{2^{*}_\mu-1}\phi\!\left(\sum_{j=1}^{i-1}a_j\frac{\partial u_\varepsilon(x)}{\partial x_j}\right)\!\left(\displaystyle{\int_{\Omega}}\frac{u_{\varepsilon}^{2^{*}_\mu-1}(y)\!\big(\phi(x)-\phi(y)\big)\left(\sum_{l=1}^{i-1}a_l\frac{\partial u_\varepsilon}{\partial y_l}(y)\right)}{|x-y|^{\mu}}dy\!\right)\!dx}_{:=N_{3,\varepsilon}^{(2)}},
   \end{split}
   \end{equation}
and
\begin{equation}\label{D-ep}
D_\varepsilon:=D_{1,\varepsilon}+D_{2,\varepsilon}+D_{3,\varepsilon},
\end{equation}
with
\begin{equation}\label{D-ep,1}
D_{1,\varepsilon}:=a_{0}^{2} \left(\varepsilon \displaystyle{\int_{\Omega}}u_{\varepsilon}^{2}(x)dx+(2\cdot2^{*}_\mu-1)\displaystyle{\int_{\Omega}}\displaystyle{\int_{\Omega}}\frac{u_{\varepsilon}^{2^{*}_\mu}(x)u_{\varepsilon}^{2^{*}_\mu}(y)}{|x-y|^{\mu}}\,dxdy\right),
\end{equation}
\smallskip
 \begin{equation}\label{D-ep,2}
     \begin{split}
D_{2,\varepsilon}:=&2a_0\,\underbrace{\varepsilon \displaystyle{\int_{\Omega}}
 u_{\varepsilon}(x)\phi(x)\left(\sum_{j=1}^{i-1}a_j\frac{\partial u_\varepsilon}{\partial x_j}(x)\right)dx}_{:=D_{2,\varepsilon}^{(1)}}\\
 &+2a_0(2\cdot2^{*}_\mu-1)\underbrace{\displaystyle{\int_{\Omega}}\displaystyle{\int_{\Omega}}\frac{u_{\varepsilon}^{2^{*}_\mu-1}(x)\phi(x)\left(\sum_{j=1}^{i-1}a_j\frac{\partial u_\varepsilon}{\partial x_j}(x)\right)u_{\varepsilon}^{2^{*}_\mu}(y)}{|x-y|^{\mu}}\,dxdy}_{:=D_{2,\varepsilon}^{(2)}},
     \end{split}
 \end{equation}
 \smallskip
\begin{equation}\label{D-ep,3}
    \begin{split}
     D_{3,\varepsilon}:&=\varepsilon \displaystyle{\int_{\Omega}}\phi^{2}(x)\left(\sum_{j=1}^{i-1}a_j\frac{\partial u_\varepsilon}{\partial x_j}(x)\right)\left(\sum_{l=1}^{i-1}a_l\frac{\partial u_\varepsilon}{\partial x_l}(x)\right)dx\\
+&(2^{*}_\mu-1)\displaystyle{\int_{\Omega}}\displaystyle{\int_{\Omega}}\frac{u_{\varepsilon}^{2^{*}_\mu-2}(x)u_{\varepsilon}^{2^{*}_\mu}(y)\phi^{2}(x)\left(\sum_{j=1}^{i-1}a_j\frac{\partial u_\varepsilon}{\partial x_j}(x)\right)\left(\sum_{l=1}^{i-1}a_l\frac{\partial u_\varepsilon}{\partial x_l}(x)\right)}{|x-y|^{\mu}}dxdy\\
+&2^{*}_\mu\displaystyle{\int_{\Omega}}u_{\varepsilon}^{2^{*}_\mu-1}(x)\phi(x)\left(\sum_{j=1}^{i-1}a_j\frac{\partial u_\varepsilon}{\partial x_j}(x)\right)\left(\displaystyle{\int_{\Omega}}\frac{u_{\varepsilon}^{2^{*}_\mu-1}(y)\phi(y)\left(\sum_{l=1}^{i-1}a_l\frac{\partial u_\varepsilon}{\partial y_l}(y)\right)}{|x-y|^{\mu}}dy\!\right)dx.
    \end{split}
\end{equation}
Note that the symmetry property of the double integrals plays a crucial role in the computation of the above terms. The rest of the proof is divided into three steps.

\textbf{Step 1.} We prove some estimates which are needed to compute the quotient in \eqref{eq4.18}.

As for $D_{2,\varepsilon}^{(1)}$ given in \eqref{D-ep,2}, 
there holds  for $j=1,\cdots,i-1$,
\begin{equation}\label{4-ep}
    \begin{split}
       & \varepsilon \displaystyle{\int_{\Omega}}
 u_{\varepsilon}(x)\,\phi(x)\left(\frac{\partial u_\varepsilon(x)}{\partial x_j}\right)dx=\varepsilon\displaystyle{\int_{\Omega}} \,\frac{\partial}{\partial x_{j}}\left(\frac{u_{\varepsilon}^{2}(x)}{2}\right)\phi(x)\,dx=-\frac{\varepsilon}{2}\displaystyle{\int_{\Omega}} \,\left(\frac{\partial \phi(x)}{\partial x_{j}}\right)\,u_{\varepsilon}^{2}(x)dx\\=&-\frac{\varepsilon}{2\|u_\varepsilon\|_{\infty}^{2}}\displaystyle{\int_{\Omega\cap\{|x-x_{\varepsilon}|\geq\rho\}}} \left(\frac{\partial \phi(x)}{\partial x_{j}}\right)\Big(\|u_\varepsilon\|_{\infty}u_{\varepsilon}(x)\Big)^{2}dx
 =O\left(\frac{\varepsilon}{\|u_\varepsilon\|_{\infty}^{2}}\right).
      \end{split}
\end{equation}
For $D_{2,\varepsilon}^{(2)}$, using integration by parts, we have
\begin{equation}\label{hatH2}
    \begin{split}
&\displaystyle{\int_{\Omega}}\displaystyle{\int_{\Omega}}\frac{u_{\varepsilon}^{2^{*}_\mu-1}(x)\phi(x)\left(\sum_{j=1}^{i-1}a_j\frac{\partial u_\varepsilon}{\partial x_j}(x)\right)u_{\varepsilon}^{2^{*}_\mu}(y)}{|x-y|^{\mu}}dxdy\\
          =&\frac{1}{2^{*}_\mu}\displaystyle{\int_{\Omega}}\displaystyle{\int_{\Omega}}\frac{\phi(x)\left(\sum_{j=1}^{i-1}a_j\frac{\partial }{\partial x_j}u_{\varepsilon}^{2^{*}_\mu}(x)\right)u_{\varepsilon}^{2^{*}_\mu}(y)}{|x-y|^{\mu}}dxdy\\
          =&\underbrace{\frac{1}{2^{*}_\mu}\displaystyle{\int_{\Omega}}\displaystyle{\int_{\Omega}}\frac{\left(-\sum_{j=1}^{i-1}a_j\frac{\partial }{\partial x_j}\phi(x)\right)u_{\varepsilon}^{2^{*}_\mu}(x)u_{\varepsilon}^{2^{*}_\mu}(y)}{|x-y|^{\mu}}dxdy}_{:=\widehat{H_{1}}}\\
          &+\underbrace{\frac{1}{2^{*}_\mu}\displaystyle{\int_{\Omega}}\displaystyle{\int_{\Omega}}\left(-\sum_{j=1}^{i-1}a_j\frac{\partial }{\partial x_j}\left(\frac{1}{|x-y|^{\mu}}\right)\right)\phi(x)u_{\varepsilon}^{2^{*}_\mu}(x)u_{\varepsilon}^{2^{*}_\mu}(y)dxdy}_{:=\widehat{H_{2}}}.
    \end{split}
\end{equation}
By HLS, \eqref{eq 2.2} and \eqref{eq 2.3}, we get that
     \begin{equation}\label{4-19}
     \begin{split}
        \widehat{H_{1}}
\leq&~C\left(\displaystyle{\int_{\Omega}}|u_{\varepsilon}(y)|^{2^{*}}dy\right)^{\frac{2N-\mu}{2N}}\left(\displaystyle{\int_{\Omega}}\left|\left(-\sum_{j=1}^{i-1}a_j\frac{\partial }{\partial x_j}\phi(x)\right)u_{\varepsilon}^{2^{*}_\mu}(x)\right|^{\frac{2N}{2N-\mu}}dx\right)^{\frac{2N-\mu}{2N}}\\
        \leq&~\frac{C}{\|u_\varepsilon\|_{\infty}^{2^{*}_\mu}}\left(\displaystyle{\int_{\Omega\cap\{|x-x_\varepsilon|\geq \rho\}}}\left|\left(-\sum_{j=1}^{i-1}a_j\frac{\partial }{\partial x_j}\phi(x)\right)\Big(\|u_\varepsilon\|_{\infty}u_{\varepsilon}\Big)^{2^{*}_\mu}\right|^{\frac{2N}{2N-\mu}}dx\right)^{\frac{2N-\mu}{2N}}\\
        \leq&~\frac{C}{\|u_\varepsilon\|_{\infty}^{2^{*}_\mu}}.
         \end{split}
     \end{equation}
Moreover, a direct calculation yields that
     \begin{align*}
\widehat{H_{2}}=&~\frac{\mu\sum_{j=1}^{i-1}a_j}{2^{*}_\mu}\displaystyle{\int_{\Omega}}\displaystyle{\int_{\Omega}}\frac{\left(x_j-y_j\right)\phi(x)u_{\varepsilon}^{2^{*}_\mu}(x)u_{\varepsilon}^{2^{*}_\mu}(y)}{|x-y|^{\mu+2}}dxdy\\
=&~\frac{\mu\sum_{j=1}^{i-1}a_j}{2^{*}_\mu}\underbrace{\displaystyle{\int_{B_{\rho}(x_\varepsilon)}}\displaystyle{\int_{B_{\rho}(x_\varepsilon)}}\frac{\left(x_j-y_j\right)\phi(x)u_{\varepsilon}^{2^{*}_\mu}(x)u_{\varepsilon}^{2^{*}_\mu}(y)}{|x-y|^{\mu+2}}dxdy}_{:=\widehat{H_{2,1}}}\\
&+\frac{\mu\sum_{j=1}^{i-1}a_j}{2^{*}_\mu}\underbrace{\displaystyle{\int_{B_{\rho}(x_\varepsilon)}}\displaystyle{\int_{\Omega\backslash B_{\rho}(x_\varepsilon)}}\frac{\left(x_j-y_j\right)\phi(x)u_{\varepsilon}^{2^{*}_\mu}(x)u_{\varepsilon}^{2^{*}_\mu}(y)}{|x-y|^{\mu+2}}dxdy}_{:=\widehat{H_{2,2}}}\\
&+\frac{\mu\sum_{j=1}^{i-1}a_j}{2^{*}_\mu}\underbrace{\displaystyle{\int_{\Omega\backslash B_{\rho}(x_\varepsilon)}}\displaystyle{\int_{B_{\rho}(x_\varepsilon)}}\frac{\left(x_j-y_j\right)\phi(x)u_{\varepsilon}^{2^{*}_\mu}(x)u_{\varepsilon}^{2^{*}_\mu}(y)}{|x-y|^{\mu+2}}dxdy}_{:=\widehat{H_{2,3}}}\\
&+\frac{\mu\sum_{j=1}^{i-1}a_j}{2^{*}_\mu}\underbrace{\displaystyle{\int_{\Omega\backslash B_{\rho}(x_\varepsilon)}}\displaystyle{\int_{\Omega\backslash B_{\rho}(x_\varepsilon)}}\frac{\left(x_j-y_j\right)\phi(x)u_{\varepsilon}^{2^{*}_\mu}(x)u_{\varepsilon}^{2^{*}_\mu}(y)}{|x-y|^{\mu+2}}dxdy}_{:=\widehat{H_{2,4}}}.
     \end{align*}
Note that for any $x\in B_{\rho}(x_\varepsilon)$, $\phi(x)\equiv1$. Then due to the symmetry, we have
    \begin{align*}
\widehat{H_{2,1}}=&\!\displaystyle{\int_{B_{\rho}(x_\varepsilon)}}\displaystyle{\int_{B_{\rho}(x_\varepsilon)}}\frac{\left(x_j-y_j\right)u_{\varepsilon}^{2^{*}_\mu}(x)u_{\varepsilon}^{2^{*}_\mu}(y)}{|x-y|^{\mu+2}}dxdy=\!\displaystyle{\int_{B_{\rho}(x_\varepsilon)}}\displaystyle{\int_{B_{\rho}(x_\varepsilon)}}\!\frac{\left(y_j-x_j\right)u_{\varepsilon}^{2^{*}_\mu}(x)u_{\varepsilon}^{2^{*}_\mu}(y)}{|x-y|^{\mu+2}}dxdy\\
=&\!-\!\displaystyle{\int_{B_{\rho}(x_\varepsilon)}}\displaystyle{\int_{B_{\rho}(x_\varepsilon)}}\!\frac{\left(x_j-y_j\right)u_{\varepsilon}^{2^{*}_\mu}(x)u_{\varepsilon}^{2^{*}_\mu}(y)}{|x-y|^{\mu+2}}dxdy,
    \end{align*}
which implies that $\widehat{H_{2,1}}=0$. Moreover, using HLS with $\frac{2N-\mu}{2N}+\frac{2N-\mu-2}{2N}+\frac{\mu+1}{N}=2$, we derive that
    \begin{align*}
\widehat{H_{2,2}}\leq&~C\displaystyle{\int_{B_{\rho}(x_\varepsilon)}}\displaystyle{\int_{\Omega\backslash B_{\rho}(x_\varepsilon)}}\frac{|u_{\varepsilon}(x)|^{2^{*}_\mu}|u_{\varepsilon}(y)|^{2^{*}_\mu}}{|x-y|^{\mu+1}}dxdy\leq C\displaystyle{\int_{B_{\rho}(x_\varepsilon)}}\displaystyle{\int_{\Omega\backslash B_{\rho}(x_\varepsilon)}}\frac{U_{x_\varepsilon,\tau_{\varepsilon}}^{2^{*}_\mu}(x)U_{x_\varepsilon,\tau_{\varepsilon}}^{2^{*}_\mu}(y)}{|x-y|^{\mu+1}}dxdy\\
=&~C\tau_{\varepsilon}^{2N-\mu}\displaystyle{\int_{B_{\rho}(x_\varepsilon)}}\displaystyle{\int_{\Omega\backslash B_{\rho}(x_\varepsilon)}}\frac{1}{\left(1+\tau_{\varepsilon}^{2}|x-x_{\varepsilon}|^{2}\right)^{\frac{2N-\mu}{2}}|x-y|^{\mu+1}\left(1+\tau_{\varepsilon}^{2}|y-x_{\varepsilon}|^{2}\right)^{\frac{2N-\mu}{2}}}dxdy\\
\leq&~C\tau_{\varepsilon}^{2N-\mu}\left(\displaystyle{\int_{B_{\rho}(x_\varepsilon)}}\frac{1}{\left(1+\tau_{\varepsilon}^{2}|y-x_{\varepsilon}|^{2}\right)^{\frac{2N-\mu}{2}\frac{2N}{2N-\mu}}}\,dy\right)^{\frac{2N-\mu}{2N}}\\
&\quad\,\quad\,\quad\cdot\left(\displaystyle{\int_{\Omega\backslash B_{\rho}(x_\varepsilon)}}\frac{1}{\left(1+\tau_{\varepsilon}^{2}|x-x_{\varepsilon}|^{2}\right)^{\frac{2N-\mu}{2}\frac{2N}{2N-\mu-2}}}\,dx\right)^{\frac{2N-\mu-2}{2N}}\\
\leq&~C\left(\displaystyle{\int_{B_{\rho}(x_\varepsilon)}}\frac{1}{\left(1+\tau_{\varepsilon}^{2}|y-x_{\varepsilon}|^{2}\right)^{N}}\,dy\right)^{\frac{2N-\mu}{2N}}=C\left(\displaystyle{\int_{B_{\tau_{\varepsilon}\rho}(0)}}\frac{1}{\left(1+|z|^{2}\right)^{N}}\frac{1}{\tau_{\varepsilon}^{N}}\,dz\right)^{\frac{2N-\mu}{2N}}\\
\leq&~\frac{C}{\tau_{\varepsilon}^{\frac{2N-\mu}{2}}}=\frac{C}{\|u_\varepsilon\|_{\infty}^{2^{*}_\mu}}.
    \end{align*}
Similarly, we obtain that $$\widehat{H_{2,3}}=O\left(\frac{1}{\|u_\varepsilon\|_{\infty}^{2^{*}_\mu}}\right),\quad \widehat{H_{2,4}}=O\left(\frac{1}{\|u_\varepsilon\|_{\infty}^{2\cdot2^{*}_\mu}}\right).$$
Thus, from the estimates of $\widehat{H_{2,1}}-\widehat{H_{2,4}}$, we have $\widehat{H_{2}}=O\left({1}/{\|u_{\varepsilon}\|_{\infty}^{2^{*}_\mu}}\right)$. Combining  \eqref{hatH2} and \eqref{4-19}, we get that
\begin{equation}\label{4-20}
D_{2,\varepsilon}^{(2)}:=\widehat{H_{1}}+\widehat{H_{2}}=O\left(\frac{1}{\|u_\varepsilon\|_{\infty}^{2^{*}_\mu}}\right).
\end{equation}
Therefore, from \eqref{D-ep,2}, \eqref{4-ep} and \eqref{4-20}, we conclude that
\begin{equation}\label{4-21}
D_{2,\varepsilon}:=D_{2,\varepsilon}^{(1)}+D_{2,\varepsilon}^{(2)}=O\left(\frac{\varepsilon}{\|u_\varepsilon\|_{\infty}^{2}}\right)+O\left(\frac{1}{\|u_{\varepsilon}\|_{\infty}^{2^{*}_\mu}}\right).
\end{equation}
As for $D_{3,\varepsilon}$ given in \eqref{D-ep,3}, by a change of variables and the utilization of \eqref{eq 1.9} and \eqref{a-1}, we derive that
\begin{equation*}
\begin{split}
&D_{3,\varepsilon}\\=&\varepsilon \displaystyle{\int_{\Omega_{\varepsilon}}}{\phi_{\varepsilon}^{2}}(x)\left(\sum_{j=1}^{i-1}a_j\frac{\partial \tilde{u}_\varepsilon}{\partial x_j}(x)\right)\left(\sum_{l=1}^{i-1}a_l\frac{\partial \tilde{u}_\varepsilon}{\partial x_l}(x)\right)dx\\
&+(2^{*}_\mu-1)\tau_\varepsilon^{2}\!\displaystyle{\int_{\Omega_{\varepsilon}}}\displaystyle{\int_{\Omega_{\varepsilon}}}\frac{\tilde{u}_{\varepsilon}^{2^{*}_\mu-2}(x){\phi_{\varepsilon}^{2}}(x)\left(\sum_{j=1}^{i-1}a_j\frac{\partial \tilde{u}_\varepsilon}{\partial x_j}(x)\right)\left(\sum_{l=1}^{i-1}a_l\frac{\partial \tilde{u}_\varepsilon}{\partial x_l}(x)\right)\tilde{u}_{\varepsilon}^{2^{*}_\mu}(y)}{|x-y|^{\mu}}\,dxdy\\
&+2^{*}_\mu\tau_\varepsilon^{2}\displaystyle{\int_{\Omega_{\varepsilon}}}\tilde{u}_{\varepsilon}^{2^{*}_\mu-1}(x){\phi_{\varepsilon}}(x)\left(\sum_{j=1}^{i-1}a_j\frac{\partial \tilde{u}_\varepsilon}{\partial x_j}(x)\right)\left(\displaystyle{\int_{\Omega_{\varepsilon}}}\frac{\tilde{u}_{\varepsilon}^{2^{*}_\mu-1}(y)\phi_{\varepsilon}(y)\left(\sum_{l=1}^{i-1}a_l\frac{\partial \tilde{u}_\varepsilon}{\partial y_l}(y)\right)}{|x-y|^{\mu}}dy\right)dx\\
=&\varepsilon\left(\sum_{j=1}^{i-1}a_j^{2}\right)\left(\frac{1}{N}\displaystyle{\int_{\R^N}}\left|\nabla U_{0,1}\right|^{2}dx+o(1)\right)\\
&+\!(2^{*}_\mu-1)\tau_{\varepsilon}^{2}\!\left(\displaystyle{\int_{\R^N}}U_{0,1}^{2^{*}_\mu-2}\!\left(\sum_{j=1}^{i-1}a_j\frac{\partial {U_{0,1}(x)}}{\partial x_j}\right)\!\left(\sum_{l=1}^{i-1}a_l\frac{\partial U_{0,1}(x)}{\partial x_l}\right)\!\left(\displaystyle{\int_{\R^N}}\frac{U_{0,1}^{2^{*}_\mu}(y)}{|x-y|^{\mu}}dy\right)dx+o(1)\!\right)\\
&+2^{*}_\mu\tau_\varepsilon^{2}\!\left(\displaystyle{\int_{\R^N}}U_{0,1}^{2^{*}_\mu-1}\!\left(\sum_{j=1}^{i-1}a_j\frac{\partial U_{0,1}(x)}{\partial x_j}\right)\!\left(\displaystyle{\int_{\R^N}}\frac{U^{2^{*}_\mu-1}_{0,1}(y)\!\left(\sum_{l=1}^{i-1}a_l\frac{\partial U_{0,1}(y)}{\partial y_l}\right)}{|x-y|^{\mu}}dy\right)dx+o(1)\!\right)\\
=&\varepsilon\left(\sum_{j=1}^{i-1}a_j^{2}\right)\!\left(\frac{1}{N}\displaystyle{\int_{\R^N}}\left|\nabla U_{0,1}\right|^{2}dx+o(1)\right)\\
&+\left(2^{*}-1\right)\tau_\varepsilon^{2}\left(\displaystyle{\int_{\R^N}}U_{0,1}^{2^{*}-2}(x)\left(\sum_{j=1}^{i-1}a_j\frac{\partial {U_{0,1}(x)}}{\partial x_j}\right)\left(\sum_{l=1}^{i-1}a_l\frac{\partial U_{0,1}(x)}{\partial x_l}\right)dx+o(1)\right)\\
=&\varepsilon\left(\sum_{j=1}^{i-1}a_j^{2}\right)\left(\frac{1}{N}\displaystyle{\int_{\R^N}}\left|\nabla U_{0,1}\right|^{2}dx+o(1)\right)\\
&+\left(2^{*}-1\right)\tau_\varepsilon^{2}\left(\sum_{j=1}^{i-1}a_j^{2}\right)\left(\frac{1}{N}\displaystyle{\int_{\R^N}}U_{0,1}^{2^{*}-2}(x)\left|\nabla U_{0,1}\right|^{2}dx+o(1)\right),
\end{split}
\end{equation*}
where $\phi_{\varepsilon}(x)=\phi\left(\tau_{\varepsilon}^{-1}x+x_{\varepsilon}\right)$, $\forall\,x\in\Omega_{\varepsilon}$. Hence, we get that
  \begin{equation}\label{4-23}
  \begin{split}
D_{3,\varepsilon}=&\varepsilon\left(\sum_{j=1}^{i-1}a_j^{2}\right)\left(\frac{1}{N}\displaystyle{\int_{\R^N}}\left|\nabla U_{0,1}\right|^{2}dx+o(1)\right)\\
 &+\left(2^{*}-1\right)\|u_{\varepsilon}\|_{\infty}^{2^{*}-2}\left(\sum_{j=1}^{i-1}a_j^{2}\right)\left(\frac{1}{N}\displaystyle{\int_{\R^N}}U_{0,1}^{2^{*}-2}(x)\left|\nabla U_{0,1}\right|^{2}dx+o(1)\right)
  \end{split}
  \end{equation}
Therefore, combining  \eqref{D-ep}, \eqref{D-ep,1}, \eqref{4-21} and \eqref{4-23}, we conclude that
\begin{equation}\label{eq4.19}
   \begin{split}
D_{\varepsilon} =&D_{1,\varepsilon}+D_{2,\varepsilon}+D_{3,\varepsilon}\\
=&~ a_{0}^{2}\left(\varepsilon \displaystyle{\int_{\Omega}}u_{\varepsilon}^{2}(x)dx+(2\cdot2^{*}_\mu-1)\displaystyle{\int_{\Omega}}\displaystyle{\int_{\Omega}}\frac{u_{\varepsilon}^{2^{*}_\mu}(x)u_{\varepsilon}^{2^{*}_\mu}(y)}{|x-y|^{\mu}}dxdy\right)\\
&+O\left(\frac{\varepsilon}{\|u_\varepsilon\|_{\infty}^{2}}\right)+O\left(\frac{1}{\|u_{\varepsilon}\|_{\infty}^{2^{*}_\mu}}\right)+\varepsilon\left(\sum_{j=1}^{i-1}a_j^{2}\right)\left(\frac{1}{N}\displaystyle{\int_{\R^N}}\left|\nabla U_{0,1}\right|^{2}dx+o(1)\right)\\
&+\left(2^{*}-1\right)\|u_{\varepsilon}\|_{\infty}^{2^{*}-2}\left(\sum_{j=1}^{i-1}a_j^{2}\right)\left(\frac{1}{N}\displaystyle{\int_{\R^N}}U_{0,1}^{2^{*}-2}(x)\left|\nabla U_{0,1}\right|^{2}dx+o(1)\right).
\end{split}
\end{equation}
As for the numerator $N_{\varepsilon}$, the term $N_{2,\varepsilon}$ defined by \eqref{N-ep,2} can be estimated by the similar way in the computation of $D_{2,\varepsilon}^{(2)}$ in \eqref{D-ep,2} and then it holds
\begin{equation}\label{4-25}
N_{2,\varepsilon}=O\left(\frac{1}{\|u_{\varepsilon}\|_{\infty}^{2^{*}_\mu}}\right).
\end{equation}
As for $N_{3,\varepsilon}$, the term $N_{3,\varepsilon}^{(1)}$ can be esimated as follows
\begin{equation}\label{4-34}
\begin{split}
&\int_{\Omega}|\nabla\phi|^{2}\left(\sum_{j=1}^{i-1}a_j\frac{\partial u_\epsilon}{\partial x_j}(x)\right)\left(\sum_{l=1}^{i-1}a_l\frac{\partial u_\epsilon}{\partial x_l}(x)\right)dx\\
=&~\frac{1}{\|u_{\epsilon}\|_{\infty}^{2}}\int_{\Omega\cap\{|x-x_{\epsilon}|\geq\rho\}}|\nabla\phi|^{2}\left(\sum_{j=1}^{i-1}a_j\frac{\partial \left(\|u_{\epsilon}\|_{\infty}u_\epsilon\right)}{\partial x_j}(x)\right)\left(\sum_{l=1}^{i-1}a_l\frac{\partial \left(\|u_{\epsilon}\|_{\infty}u_\epsilon\right)}{\partial x_l}(x)\right)dx\\
=&~O\left(\frac{1}{\|u_{\epsilon}\|_{\infty}^{2}}\right).
\end{split}
\end{equation}
For $N_{3,\varepsilon}^{(2)}$, due to the symmetry of the double integrals, we have
\begin{align*}
&2^{*}_\mu\displaystyle{\int_{\Omega}}u_{\varepsilon}^{2^{*}_\mu-1}(x)\phi(x)\left(\sum_{j=1}^{i-1}a_j\frac{\partial u_\varepsilon}{\partial x_j}(x)\right)\left(\displaystyle{\int_{\Omega}}\frac{u_{\varepsilon}^{2^{*}_\mu-1}(y)\left(\phi(x)-\phi(y)\right)\left(\sum_{l=1}^{i-1}a_l\frac{\partial u_\varepsilon}{\partial y_l}(y)\right)}{|x-y|^{\mu}}dy\right)dx\\
=&~2^{*}_\mu\displaystyle{\int_{\Omega}}u_{\varepsilon}^{2^{*}_\mu-1}(y)\phi(y)\left(\sum_{j=1}^{i-1}a_j\frac{\partial u_\varepsilon}{\partial y_j}(y)\right)\left(\displaystyle{\int_{\Omega}}\frac{u_{\varepsilon}^{2^{*}_\mu-1}(x)\left(\phi(y)-\phi(x)\right)\left(\sum_{l=1}^{i-1}a_l\frac{\partial u_\varepsilon}{\partial x_l}(x)\right)}{|x-y|^{\mu}}dx\right)dy\\
=&-2^{*}_\mu\displaystyle{\int_{\Omega}}u_{\varepsilon}^{2^{*}_\mu-1}(y)\phi(y)\left(\sum_{j=1}^{i-1}a_j\frac{\partial u_\varepsilon}{\partial y_j}(y)\right)\left(\displaystyle{\int_{\Omega}}\frac{u_{\varepsilon}^{2^{*}_\mu-1}(x)\left(\phi(x)-\phi(y)\right)\left(\sum_{l=1}^{i-1}a_l\frac{\partial u_\varepsilon}{\partial x_l}(x)\right)}{|x-y|^{\mu}}dx\right)dy,
\end{align*}
then we obtain that
\begin{align}
&N_{3,\varepsilon}^{(2)}\notag\\
=&\frac{2^{*}_\mu}{2}\displaystyle{\int_{\Omega}}\displaystyle{\int_{\Omega}}\frac{\left(\phi(x)-\phi(y)\right)^{2}u_{\varepsilon}^{2^{*}_\mu-1}(x)\left(\sum_{j=1}^{i-1}a_j\frac{\partial u_\varepsilon}{\partial x_j}(x)\right)u_{\varepsilon}^{2^{*}_\mu-1}(y)\left(\sum_{l=1}^{i-1}a_l\frac{\partial u_\varepsilon}{\partial y_l}(y)\right)}{|x-y|^{\mu}}dxdy\notag\\
=&\frac{1}{2\cdot2^{*}_\mu}\displaystyle{\int_{\Omega}}\displaystyle{\int_{\Omega}}\frac{\left(\phi(x)-\phi(y)\right)^{2}\left(\sum_{j=1}^{i-1}a_j\frac{\partial u_\varepsilon^{2^{*}_\mu}}{\partial x_j}(x)\right)\left(\sum_{l=1}^{i-1}a_l\frac{\partial u_\varepsilon^{2^{*}_\mu}}{\partial y_l}(y)\right)}{|x-y|^{\mu}}dxdy\notag\\
=&\frac{1}{2\cdot2^{*}_\mu}\displaystyle{\int_{\Omega}}\displaystyle{\int_{\Omega}}\left(\sum_{j,l=1}^{i-1}a_{j}a_{l}\frac{\partial^{2}}{\partial x_{j}\partial y_{l}}\left(\frac{\left(\phi(x)-\phi(y)\right)^{2}}{|x-y|^{\mu}}\right)\right)u_{\varepsilon}^{2^{*}_\mu}(x)u_{\varepsilon}^{2^{*}_\mu}(y)dxdy\notag\\
=&\underbrace{\frac{1}{2\cdot2^{*}_\mu}\displaystyle{\int_{\Omega}}\displaystyle{\int_{\Omega}}\left(\sum_{j,l=1}^{i-1}a_{j}a_{l}\frac{\partial^{2}}{\partial x_{j}\partial y_{l}}\Big(\phi(x)-\phi(y)\Big)^{2}\right)\frac{u_{\varepsilon}^{2^{*}_\mu}(x)u_{\varepsilon}^{2^{*}_\mu}(y)}{|x-y|^{\mu}}dxdy}_{:=\widehat{K_{1}}}\notag\\
&+\underbrace{\frac{1}{2\cdot2^{*}_\mu}\displaystyle{\int_{\Omega}}\displaystyle{\int_{\Omega}}\left(\sum_{j,l=1}^{i-1}a_{j}a_{l}\frac{\partial^{2}}{\partial x_{j}\partial y_{l}}\left(\frac{1}{|x-y|^{\mu}}\right)\right)\left(\phi(x)-\phi(y)\right)^{2}u_{\varepsilon}^{2^{*}_\mu}(x)u_{\varepsilon}^{2^{*}_\mu}(y)dxdy}_{:=\widehat{K_{2}}}\notag\\
&+\underbrace{\frac{1}{2\cdot2^{*}_\mu}\displaystyle{\int_{\Omega}}\displaystyle{\int_{\Omega}}\left(\sum_{j,l=1}^{i-1}a_{j}a_{l}\frac{\partial}{\partial x_{j}}\left(\frac{1}{|x-y|^{\mu}}\right)\frac{\partial}{\partial y_{l}}\left(\phi(x)-\phi(y)\right)^{2}\right)u_{\varepsilon}^{2^{*}_\mu}(x)u_{\varepsilon}^{2^{*}_\mu}(y)dxdy}_{:=\widehat{K_{3}}}\notag\\
&+\underbrace{\frac{1}{2\cdot2^{*}_\mu}\displaystyle{\int_{\Omega}}\displaystyle{\int_{\Omega}}\left(\sum_{j,l=1}^{i-1}a_{j}a_{l}\frac{\partial}{\partial y_{l}}\left(\frac{1}{|x-y|^{\mu}}\right)\frac{\partial}{\partial x_{j}}\left(\phi(x)-\phi(y)\right)^{2}\right)u_{\varepsilon}^{2^{*}_\mu}(x)u_{\varepsilon}^{2^{*}_\mu}(y)dxdy}_{:=\widehat{K_{4}}}.\label{eqq4.35}
\end{align}
From HLS, we get that
\begin{align*}
    \widehat{K_{1}}=&\frac{1}{2\cdot2^{*}_\mu}\displaystyle{\int_{\Omega}}\displaystyle{\int_{\Omega}}\left(-2\sum_{j,l=1}^{i-1}a_{j}a_{l}\frac{\partial \phi(x)}{\partial x_{j}}\frac{\partial \phi(y)}{\partial y_{l}}\right)\frac{u_{\varepsilon}^{2^{*}_\mu}(x)u_{\varepsilon}^{2^{*}_\mu}(y)}{|x-y|^{\mu}}dxdy\\
\leq&C\left(\displaystyle{\int_{\Omega}}\left|\sum_{j=1}^{i-1}a_{j}\frac{\partial \phi(x)}{\partial x_{j}}u_{\varepsilon}^{2^{*}_\mu}(x)\right|^{\frac{2N}{2N-\mu}}dx\right)^{\frac{2N-\mu}{2N}}\left(\displaystyle{\int_{\Omega}}\left|\sum_{l=1}^{i-1}a_{l}\frac{\partial \phi(y)}{\partial y_{l}}u_{\varepsilon}^{2^{*}_\mu}(y)\right|^{\frac{2N}{2N-\mu}}dy\right)^{\frac{2N-\mu}{2N}}\\
    =&O\left(\frac{1}{\|u_\varepsilon\|_{\infty}^{2\cdot2^{*}_\mu}}\right).
\end{align*}
Moreover, a simple computation yields that
\begin{align*}
    \widehat{K_{2}}
\leq&~C\displaystyle{\int_{\Omega}}\displaystyle{\int_{\Omega}}\frac{\left(\phi(x)-\phi(y)\right)^{2}u_{\varepsilon}^{2^{*}_\mu}(x)u_{\varepsilon}^{2^{*}_\mu}(y)}{|x-y|^{\mu+2}}dxdy\\
=&~\underbrace{C\displaystyle{\int_{B_{\rho(x_\varepsilon)}}}\displaystyle{\int_{B_{\rho(x_\varepsilon)}}}\frac{\left(\phi(x)-\phi(y)\right)^{2}u_{\varepsilon}^{2^{*}_\mu}(x)u_{\varepsilon}^{2^{*}_\mu}(y)}{|x-y|^{\mu+2}}dxdy}_{:=\widehat{K_{2,1}}}\\
&+\underbrace{2C\displaystyle{\int_{B_{\rho(x_\varepsilon)}}}\displaystyle{\int_{\Omega\backslash B_{\rho(x_\varepsilon)}}}\frac{\left(\phi(x)-\phi(y)\right)^{2}u_{\varepsilon}^{2^{*}_\mu}(x)u_{\varepsilon}^{2^{*}_\mu}(y)}{|x-y|^{\mu+2}}dxdy}_{:=\widehat{K_{2,2}}}\\
&+\underbrace{C\displaystyle{\int_{\Omega\backslash B_{\rho(x_\varepsilon)}}}\displaystyle{\int_{\Omega\backslash B_{\rho(x_\varepsilon)}}}\frac{\left(\phi(x)-\phi(y)\right)^{2}u_{\varepsilon}^{2^{*}_\mu}(x)u_{\varepsilon}^{2^{*}_\mu}(y)}{|x-y|^{\mu+2}}dxdy}_{:=\widehat{K_{2,3}}}.
\end{align*}
Note that if $x,y\in B_{\rho}(x_\varepsilon)$, $\phi(x)=\phi(y)=1$, then $\widehat{K_{2,1}}=0$. Moreover, 
by HLS with $\frac{2N-\mu}{2N}+\frac{2N-\mu-4}{2N}+\frac{\mu+2}{N}=2$, we have
\begin{align*}
\widehat{K_{2,2}}\leq&~ C\displaystyle{\int_{B_{\rho(x_\varepsilon)}}}\displaystyle{\int_{\Omega\backslash B_{\rho(x_\varepsilon)}}}\frac{U_{x_\varepsilon,\tau_\varepsilon}^{2^{*}_\mu}(x)U_{x_\varepsilon,\tau_\varepsilon}^{2^{*}_\mu}(y)}{|x-y|^{\mu+2}}dxdy\\
\leq~&C\tau_\varepsilon^{2N-\mu}\displaystyle{\int_{B_{\rho(x_\varepsilon)}}}\displaystyle{\int_{\Omega\backslash B_{\rho(x_\varepsilon)}}}\frac{1}{\left(1+\tau_\varepsilon^{2}|x-x_\varepsilon|^{2}\right)^{\frac{2N-\mu}{2}}|x-y|^{\mu+2}\left(1+\tau_\varepsilon^{2}|y-x_\varepsilon|^{2}\right)^{\frac{2N-\mu}{2}}}dxdy\\
\leq~&C\left(\displaystyle{\int_{B_{\rho(x_\varepsilon)}}}\frac{1}{(1+\tau_\varepsilon^{2}|y-x_\varepsilon|^{2})^{N}}dy\right)^{\frac{2N-\mu}{2N}}=C\left(\displaystyle{\int_{B_{\tau_\varepsilon\rho(0)}}}\frac{1}{(1+|z|^{2})^{N}}\frac{1}{\tau_\varepsilon^{N}}dz\right)^{\frac{2N-\mu}{2N}}\\
=&~O\left(\frac{1}{\tau_\varepsilon^{\frac{2N-\mu}{2}}}\right)=O\left(\frac{1}{\|u_\varepsilon\|_{\infty}^{2^{*}_\mu}}\right).
\end{align*}
Similarly, we get that
$\widehat{K_{2,3}}=O\left({1}/{\|u_\varepsilon\|_{\infty}^{2\cdot2^{*}_\mu}}\right)$. Thus, from the estimates of $\widehat{K_{2,1}}-\widehat{K_{2,3}}$, we obtain that
\begin{equation*}
\widehat{K_{2}}=O\left(\frac{1}{\|u_\varepsilon\|_{\infty}^{2^{*}_\mu}}\right).
\end{equation*}
As for $\widehat{K_{3}}$, a simple computation yields that
\begin{align*}
    &\widehat{K_{3}}\\
=&\frac{\mu}{2^{*}_\mu}\displaystyle{\int_{B_{\rho}(x_\varepsilon)}}\displaystyle{\int_{B_{\rho}(x_\varepsilon)}}\left(\sum_{j,l=1}^{i-1}a_{j}a_{l}\,\frac{(x_j-y_j)}{|x-y|^{\mu+2}}\frac{\partial \phi(y)}{\partial y_{l}}\right)\left(\phi(x)-\phi(y)\right)u_{\varepsilon}^{2^{*}_\mu}(x)u_{\varepsilon}^{2^{*}_\mu}(y)dxdy\\
&\!+\!\frac{\sum_{j,l=1}^{i-1}a_{j}a_{l}}{2\cdot2^{*}_\mu}\!\displaystyle{\int}\!\displaystyle{\int_{\left(\Omega\times\Omega\right)\backslash\left(B_{\rho}(x_\varepsilon)\times B_{\rho}(x_\varepsilon)\right)}}\!\left(\frac{\partial}{\partial x_{j}}\left(\!\frac{1}{|x-y|^{\mu}}\!\right)\right)\!\left(\!\frac{\partial}{\partial y_{l}}\Big(\!\phi(x)-\!\phi(y)\!\Big)^{2}\right)\!u_{\varepsilon}^{2^{*}_\mu}(x)u_{\varepsilon}^{2^{*}_\mu}(y)dxdy\\
    =&0+O\left(\frac{1}{\|u_{\varepsilon}\|_{\infty}^{2^{*}_\mu}}\right)=O\left(\frac{1}{\|u_{\varepsilon}\|_{\infty}^{2^{*}_\mu}}\right),
\end{align*}
where we have used the same arguments as the estimates of $\widehat{H_{2,2}}-\widehat{H_{2,4}}$ in $D_{2,\varepsilon}^{(2)}$ and the fact that $\phi(x)=\phi(y)=1$ for $x,y\in B_{\rho}(x_\varepsilon)$. Similarly, we derive that $\widehat{K_{4}}=O\left({1}/{\|u_{\varepsilon}\|_{\infty}^{2^{*}_\mu}}\right)$.
Then, from \eqref{eqq4.35} and the estimates of $\widehat{K_{1}}-\widehat{K_{4}}$, we have
\begin{equation}\label{N_{3,2}}
N_{3,\varepsilon}^{(2)}=O\left(\frac{1}{\|u_\varepsilon\|_{\infty}^{2^{*}_\mu}}\right).
\end{equation}
Hence, from \eqref{N-ep,3}, \eqref{4-34} and \eqref{N_{3,2}}, we conclude that for any $\mu\in(0,4)$,
\begin{equation}\label{4-36}
N_{3,\varepsilon}:=N_{3,\varepsilon}^{(1)}+N_{3,\varepsilon}^{(2)}=O\left(\frac{1}{\|u_{\varepsilon}\|_{\infty}^{2}}\right)+O\left(\frac{1}{\|u_\varepsilon\|_{\infty}^{2^{*}_\mu}}\right)=O\left(\frac{1}{\|u_{\varepsilon}\|_{\infty}^{2}}\right).
\end{equation}
Therefore, combining  \eqref{N-ep}, \eqref{N-ep,1}, \eqref{4-25} and \eqref{4-36}, we obtain that
\begin{equation}\label{eq4.21}
    \begin{split}
N_{\varepsilon}:=&N_{1,\varepsilon}+N_{2,\varepsilon}+N_{3,\varepsilon}\\
=&a_{0}^{2}\,\left(2-2\cdot2^{*}_\mu\right) \displaystyle{\int_{\Omega}}\displaystyle{\int_{\Omega}}\frac{u_{\varepsilon}^{2^{*}_\mu}(x)u_{\varepsilon}^{2^{*}_\mu}(y)}{|x-y|^{\mu}}\,dxdy+O\left(\frac{1}{\|u_{\varepsilon}\|_{\infty}^{2}}\right).
    \end{split}
\end{equation}

\textbf{Step 2.} We claim that there exists $C_{0}>0$ such that
\begin{equation}\label{eq 4.25}
  \underset{(a_0,a_1,\cdots,a_{i-1})\in \R^{i}}{\max}\left(1+\frac{N_{\varepsilon}}{D_{\varepsilon}}\right)\geq 1+\frac{C_{0}}{\|u_{\varepsilon}\|_{\infty}^{2^{*}}}.
\end{equation}
Indeed, by testing $(a_0,a_1,\cdots,a_{i-1})=(0,1,\cdots,1)\in \R^{i}$, we obtain that
\begin{align*}
\underset{(a_0,a_1,\cdots,a_{i-1})\in \R^{i}}{\max}\left(1+\frac{N_{\varepsilon}}{D_{\varepsilon}}\right)\geq1+\frac{N_{3,\varepsilon}}{D_{3,\varepsilon}}=1+\frac{O\left(\frac{1}{\|u_{\varepsilon}\|_{\infty}^{2}}\right)}{\|u_{\varepsilon}\|_{\infty}^{2^{*}-2}\left(\frac{2^{*}-1}{N}\displaystyle{\int_{\R^N}}U_{0,1}^{2^{*}-2}(x)\left|\nabla U_{0,1}\right|^{2}dx+o(1)\right)}.
\end{align*}
Thus, we get \eqref{eq 4.25}.

\textbf{Step 3.} Let $(a_{0,\varepsilon},a_{1,\varepsilon},\cdots,a_{i-1,\varepsilon})\in\R^{i}$ be a maximizer of
\begin{equation*}
  \underset{(a_0,a_1,\cdots,a_{i-1})\in \R^{i}}{\max}\left(1+\frac{N_{\varepsilon}}{D_{\varepsilon}}\right).
\end{equation*}
We claim that $\sum_{j=1}^{i-1}a_{j,\varepsilon}^{2}\neq0$ for any $\varepsilon>0$ small. Suppose on the contrary that $\sum_{j=0}^{i-1}a_{j,\varepsilon}^{2}=1$ with $a_{0,\varepsilon}=1$. Then from \eqref{eq 4.16} and \eqref{eq 4.25}, we have $v=v_{1,\varepsilon}$ and the corresponding eigenvalue $\lambda_{1,\epsilon}>1$ in this case, which is a contradiction with \eqref{eq 1.16} since $1/\left(2\cdot2^{*}_\mu-1\right)<1$. Thus, we can assume that $\sum_{j=1}^{i-1}a_{j,\varepsilon}^{2}=1$. Moreover, 
we can deduce that
$a_{0,\varepsilon}^{2}\|u_{\varepsilon}\|_{\infty}^{2}$ is uniformly bounded in $\varepsilon$. In fact, if not, this gives a contradiction with \eqref{eq 4.25}. Then, from the estimates \eqref{eq4.18}, \eqref{eq4.19} and \eqref{eq4.21}, we have
\begin{align*}
\lambda_{i,\varepsilon}\leq&~\left(1+\frac{N_{\varepsilon}}{D_{\varepsilon}}\right)\Bigg|_{(a_{0},a_{1},\cdots,a_{i-1}):=(a_{0,\varepsilon},a_{1,\varepsilon},\cdots,a_{i-1,\varepsilon})}\\
=&1+\frac{a_{0,\varepsilon}^{2}\|u_{\varepsilon}\|_{\infty}^{2}\,\left(2-2\cdot2^{*}_\mu\right) \displaystyle{\int_{\Omega}}\displaystyle{\int_{\Omega}}\frac{u_{\varepsilon}^{2^{*}_\mu}(x)u_{\varepsilon}^{2^{*}_\mu}(y)}{|x-y|^{\mu}}\,dxdy+O(1)}{\|u_{\varepsilon}\|_{\infty}^{2^{*}}\left(\frac{\left(2^{*}-1\right)\left(\sum_{j=1}^{i-1}a_{j,\varepsilon}^{2}\right)}{N}\displaystyle{\int_{\R^N}}U_{0,1}^{2^{*}-2}(x)\left|\nabla U_{0,1}\right|^{2}dx+o(1)\right)}\leq~1+\frac{C}{\|u_\varepsilon\|_{\infty}^{2^{*}}}.
\end{align*}
Thus, we obtain \eqref{eq 4.13}. Now, we prove that $\lambda_{i}=\underset{\varepsilon\rightarrow 0}{\limsup}\,\lambda_{i,\varepsilon}=1$. Obviously, from \eqref{eq 4.13}, we have $\lambda_{i,\varepsilon}\rightarrow\lambda_{i}\in[0,1]$ as $\varepsilon\rightarrow0$, and hence we only need to exclude the possibility of $\lambda_{i}<1$. Indeed, 
from Lemma \ref{lema 2.8}, we know that for $\varepsilon>0$ small, the sequence $\tilde{v}_{i,\varepsilon}$ satisfying \eqref{eq 1.13} is bounded in $D^{1,2}(\R^{N})$. Then up to a subsequence, there exists $0\not\equiv V_i \in\mathcal{D}^{1,2}(\R^{N})$ such that
 $\tilde{v}_{i,\varepsilon} \rightharpoonup V_i$ in $\mathcal{D}^{1,2}(\R^{N})$ and $\tilde{v}_{i,\varepsilon}\rightarrow V_i$ in $C^{1}_{loc}(\R^N)$ with $V_i$ satisfies \eqref{eq 1.14}.
On the other hand, from Lemma \ref{lma2.5}, we see that if $\lambda_{i}<1$, then $\lambda_{i}={\left(2\cdot2^{*}_\mu-1\right)}^{-1}$ and $V_{i}(x)=U_{0,1}(x)$. However, since $v_{i,\varepsilon}$ is orthogonal to $v_{1,\varepsilon}$ in the sense of \eqref{eqq1-11} for $i\geq 2$,  from \eqref{eq 1.17} and $\tilde{v}_{i,\varepsilon}\rightarrow V_{i}=U_{0,1}$ in this case, we derive that
\begin{align*}
    0=\left(2\cdot2^{*}_\mu-1\right)\displaystyle{\int_{\R^N}}U_{0,1}^{2^{*}_\mu}(x)\left(\displaystyle{\int_{\R^N}}\frac{U_{0,1}^{2^{*}_\mu}(y)}{|x-y|^{\mu}}dy\right)dx=\left(2\cdot2^{*}_\mu-1\right)\displaystyle{\int_{\R^N}}U_{0,1}^{2^{*}}(x)dx,
\end{align*}
which is a contradiction. Hence, we conclude that $\lambda_i=1$ and the proof of Lemma \ref{lema 4.2} is finished.
\end{proof}

Now, we prove the asymptotic behavior of  the eigenfunctions $v_{i,\varepsilon}$ for $i=2,\cdots,N+1$. We start by proving \eqref{eq 1.21} in Theorem 1.2.
\begin{proof} [\textbf{Proof of \eqref{eq 1.21}}]  From Lemma \ref{lema 2.8}, we only need to prove that the number $b_{i}$ in \eqref{eq 2.11} is zero. Assuming by contradiction that for $i=2,\cdots,N+1$ the number $b_{i}\neq 0$, from Lemma \ref{lema 2.11} and \eqref{eq 4.13}, we derive that
\begin{equation*}
   \frac{\kappa+o(1)}{\|u_\varepsilon\|_{\infty}^{2}}=\lambda_{i,\varepsilon}-1\,\leq O\left(\frac{1}
{\|u_\varepsilon\|_{\infty}^{2^{*}}}\right),
\end{equation*}
where $\kappa>0$. Then we derive that
\begin{equation*}
   0<\kappa+o(1)\leq O\left(\frac{1}
{\|u_\varepsilon\|_{\infty}^{2^{*}-2}}\right)\longrightarrow 0, \quad\text{as $\varepsilon\rightarrow0$,}
\end{equation*}
which is a contradiction. Thus, we get that $b_{i}=0$ in \eqref{eq 2.11} and \eqref{eq 1.21} in Theorem 1.2 is proved, for some $\vec{a}_{i}\neq \vec{0}$.
\end{proof}
\indent Next, we utilize \eqref{eq 1.21} to prove \eqref{eq 1.22} in Theorem 1.2.
\begin{proof}[\textbf{Proof of \eqref{eq 1.22}}] For any $x\in\bar{\Omega}\backslash \{x_{0}\}$, using the Green representation formula, we derive that
    \begin{equation}\label{4-0}
        \begin{split}
          v_{i,\varepsilon}(x)=&\underbrace{\lambda_{i,\varepsilon}\,\varepsilon\displaystyle{\int_{\Omega}}G(x,z)v_{i,\varepsilon}(z)dz}_{:=I_{1}(\varepsilon)}\\
&+\underbrace{\lambda_{i,\varepsilon}(2^{*}_\mu-1)\displaystyle{\int_{\Omega}}G(x,z)u_\varepsilon^{2^{*}_\mu-2}(z)v_{i,\varepsilon}(z)\left(\displaystyle{\int_{\Omega}}\frac{u_\varepsilon^{2^{*}_\mu}(y)}{|z-y|^{\mu}}dy\right)dz}_{:=I_{2}(\varepsilon)}\\
&+\underbrace{\lambda_{i,\varepsilon}\,2^{*}_\mu \displaystyle{\int_{\Omega}}G(x,z)u_\varepsilon^{2^{*}_\mu-1}(z)\left(\displaystyle{\int_{\Omega}}\frac{u_\varepsilon^{2^{*}_\mu-1}(y)v_{i,\varepsilon}(y)}{|z-y|^{\mu}}dy\right)dz}_{:=I_{3}(\varepsilon)}.
        \end{split}
    \end{equation}
As for $I_{1}(\varepsilon)$, we can derive that if $N\geq 6$,
\begin{align}\label{4-1}
    \|u_{\varepsilon}\|_{\infty}^{2+\frac{2}{N-2}}I_{1}(\varepsilon)=O\left(\frac{1}{ \|u_{\varepsilon}\|_{\infty}^{\frac{2N-10}{N-2}}}\right).
\end{align}
Indeed, by a change of variables, we have
\begin{align}\label{eqA-6}
  I_{1}(\varepsilon)=  \lambda_{i,\varepsilon}\,\varepsilon\int_{\Omega}G(x,z)v_{i,\varepsilon}(z)dz=\frac{\lambda_{i,\varepsilon}\,\varepsilon}{\tau_{\varepsilon}^{N}}\int_{\Omega_\varepsilon}G_{\varepsilon}(x,z')\tilde{v}_{i,\varepsilon}(z')\,dz',
\end{align}
where
$$G_{\varepsilon}(x,z')=G\left(x,\frac{z'}{\tau_{\varepsilon}}+x_{\varepsilon}\right),\quad \text{for $z'\in \Omega_{\varepsilon}:=\left\{z'\in\R^N, z:={\tau_{\varepsilon}^{-1}}{z'}+x_{\varepsilon}\in \Omega\right\}$.}$$
From \eqref{eq 1.21}, we know that
$$\tilde{v}_{i,\varepsilon}(z')\longrightarrow\sum_{k=1}^{N}a_{i,k}\,\frac{\partial}{\partial z^{'}_k}\left(\frac{-1}{N-2}U_{0,1}(z')\right),\quad\text{uniformly on any compact subsets of $\R^N$.}$$
Then, we can deduce that there exists a unique solution $\widehat{w}_{1,\varepsilon}(z')$ satisfying the following initial value problem
\begin{equation*}
    \begin{cases}
\,\sum_{j=1}^{N}a_{i,j}\frac{\partial w(z')}{\partial {z'_{j}}}= \tilde{v}_{i,\varepsilon}(z'),\quad z'\in \R^N,\\
\,w(z')\big|_{\Gamma_{a}}=\frac{-1}{N-2}U_{0,1}(z'),
    \end{cases}
\end{equation*}
where $\Gamma_{a}=\{z'\in\R^N:z'\cdot \Vec{a}=0 \}$. Moreover, we derive that
\begin{equation}\label{W_1}
\widehat{w}_{1,\varepsilon}\rightarrow \frac{-1}{N-2}U_{0,1},\quad\text{uniformly on any compact subsets of $\R^N$,}
 \end{equation}
  and
  \begin{equation}\label{W_2}
  \widehat{w}_{1,\varepsilon}(z')=O\left({|z'|^{3-N}}\right),\quad\text{as $|z'|\rightarrow+\infty$.}
  \end{equation}
 Then, using integration by parts, we obtain that for any $x\in\bar{\Omega}\backslash\{x_0\}$,
\begin{align*}
   I_{1}(\varepsilon)=& \frac{\lambda_{i,\varepsilon}\,\varepsilon}{\tau_{\varepsilon}^{N}}\int_{\Omega_\varepsilon}G_{\varepsilon}(x,z')\!\left(\sum_{j=1}^{N}a_{i,j}\frac{\partial\,\widehat{w}_{1,\varepsilon}(z')}{\partial {z'_{j}}}\!\right)\!dz'=\frac{\lambda_{i,\varepsilon}\,\varepsilon}{\tau_{\varepsilon}^{N}}\!\int_{\Omega_\varepsilon}\widehat{w}_{1,\varepsilon}(z')\!\left(-\sum_{j=1}^{N}a_{i,j}\frac{\partial G_{\varepsilon}(x,z')}{\partial {z'_{j}}}\!\right)dz'\\
=&\frac{\lambda_{i,\varepsilon}\,\varepsilon}{\tau_{\varepsilon}^{N}}\int_{\Omega_\varepsilon}\widehat{w}_{1,\varepsilon}(z')\left(-\sum_{j=1}^{N}a_{i,j}\left(\frac{\partial}{\partial {z'_{j}}}G\left(x,\frac{z'}{\tau_{\varepsilon}}+x_{\varepsilon}\right)\right)\right)\,dz'\\
=&\frac{\lambda_{i,\varepsilon}\,\varepsilon}{\tau_{\varepsilon}^{N+1}}\int_{\Omega_\varepsilon}\widehat{w}_{1,\varepsilon}(z')\left(-\sum_{j=1}^{N}a_{i,j}\left(\frac{\partial G\left(x,z\right)}{\partial {z_{j}}}\bigg|_{z=\frac{z'}{\tau_{\varepsilon}}+x_\varepsilon}\right)\right)\,dz'\,\underset{(*)}{\leq}\frac{C\lambda_{i,\varepsilon}\,\varepsilon}{ \|u_{\varepsilon}\|_{\infty}^{2}},
\end{align*}
where $(*)$ holds since
\begin{equation}\label{4-27}
\int_{\Omega_\varepsilon}\widehat{w}_{1,\varepsilon}(z')\left(-\sum_{j=1}^{N}a_{i,j}\left(\frac{\partial G\left(x,z\right)}{\partial {z_{j}}}\bigg|_{z=\frac{z'}{\tau_{\varepsilon}}+x_\varepsilon}\right)\right)\,dz'=O\left(\|u_\varepsilon\|_{\infty}^{\frac{6}{N-2}}\right).
\end{equation}
In fact, since $\left|\nabla G(x,z)\right|\leq \frac{C}{|x-z|^{N-1}}$,  for some $R,\,\tilde{R}>0$ with $B_{R}(0)\subseteq \Omega_{\varepsilon}\subseteq B_{\tilde{R}}(0)$, we have
\begin{equation*}
\begin{split}
    \text{LHS of \eqref{4-27}}
    \leq&\underbrace{C\displaystyle{\int_{B_{R}(0)}}|\widehat{w}_{1,\varepsilon}(z')|\,\frac{1}{\left|x-(\tau^{-1}_{\varepsilon}z'+x_\varepsilon)\right|^{N-1}}\,dz'}_{:=I_{1}^{(1)}(\varepsilon)}\\
&+\underbrace{C\displaystyle{\int_{B_{\tilde{R}}(0)\backslash B_{R}(0)}}|\widehat{w}_{1,\varepsilon}(z')|\,\frac{1}{\left|x-(\tau^{-1}_{\varepsilon}z'+x_\varepsilon)\right|^{N-1}}\,dz'}_{:=I_{1}^{(2)}(\varepsilon)}.
\end{split}
\end{equation*}
As for $I_{1}^{(1)}(\varepsilon)$, noting that for any $x\in\bar{\Omega}\backslash\{x_0\}$, $\left|x-(\tau^{-1}_{\varepsilon}z'+x_\varepsilon)\right|\geq |x-x_\varepsilon|-R\tau_{\varepsilon}^{-1}>0$,   from \eqref{W_1}, a direct computation yields that
\begin{align*}
  I_{1}^{(1)}(\varepsilon) \leq C\displaystyle{\int_{B_{R}(0)}}U_{0,1}(z')\,dz'\leq C\displaystyle{\int_{B_{R}(0)}}\,\frac{1}{|z'|^{N-2}}\,dz'<+\infty.
\end{align*}
As for $I_{1}^{(2)}(\varepsilon)$, using similar arguments as  the proof of Lemma \ref{lema 2.5} in \cite{WY}, we derive that for any $\sigma<N-2$,
$$\displaystyle{\int_{\R^N}}\frac{1}{|y-z|^{N-1}}\,\frac{1}{\left(1+|z|\right)^{2+\sigma}}\,dz\leq \frac{C}{\left(1+|y|\right)^{1+\sigma}}, $$
then from \eqref{W_2}, we have
\begin{align*}
    I_{1}^{(1)}(\varepsilon)\leq C\displaystyle{\int_{B_{\tilde{R}}(0)\backslash B_{R}(0)}}\frac{1}{|z^{'}|^{N-3}}\,\frac{\tau_{\varepsilon}^{N-1}}{\left|\tau_{\varepsilon}(x-x_\varepsilon)-z'\right|^{N-1}}\,dz'\leq \,C\frac{\tau_{\varepsilon}^{N-1}}{\left(\tau_{\varepsilon}|x-x_\varepsilon|\right)^{N-4}}\leq C\|u_\varepsilon\|_{\infty}^{\frac{6}{N-2}}.
\end{align*}
From above estimates of $I_{1}^{(1)}(\varepsilon)$ and $I_{1}^{(2)}(\varepsilon)$, we get \eqref{4-27}. Moreover,
from \eqref{eq 2.4}, we derive \eqref{4-1}, which implies that $\|u_{\varepsilon}\|_{\infty}^{2+\frac{2}{N-2}}\,I_{1}(\varepsilon)=o(1)$ provided $N\geq6$.

\vskip 0.1cm

For $I_{2}(\varepsilon)$, by a change of variables, we obtain that
\begin{align}\label{eq4.26}
    I_{2}(\varepsilon)=\frac{\lambda_{i,\varepsilon}(2^{*}_\mu-1)}{\|u_{\varepsilon}\|_{\infty}^{2}}\displaystyle{\int_{\Omega_{\varepsilon}}}G_{\varepsilon}(x,z')\,\tilde{u}_\varepsilon^{2^{*}_\mu-2}(z')\,\tilde{v}_{i,\varepsilon}(z')\left(\displaystyle{\int_{\Omega_{\varepsilon}}}\frac{\tilde{u}_\varepsilon^{2^{*}_\mu}(y')}{|z{'}-y{'}|^{\mu}}dy{'}\right)dz{'}.
\end{align}
From \eqref{eq 1.7}, \eqref{eq 1.9} and \eqref{eq 1.21}, we derive that
\begin{align*}
\tilde{v}_{i,\varepsilon}(z')&\longrightarrow\sum_{j=1}^{N}a_{i,j}\frac{\partial}{\partial {z'_{j}}}\left(\frac{-1}{N-2}U_{0,1}(z')\right),\\
 \tilde{u}_\varepsilon^{2^{*}_\mu-2}(z')\tilde{v}_{i,\varepsilon}(z')\left(\displaystyle{\int_{\Omega_{\varepsilon}}}\frac{\tilde{u}_\varepsilon^{2^{*}_\mu}(y')}{|z{'}-y{'}|^{\mu}}dy{'}\right)&\longrightarrow\sum_{j=1}^{N}a_{i,j}\frac{\partial}{\partial {z'_{j}}}\left(\frac{-1}{N+2}U_{0,1}^{2^{*}-1}(z')\right),
\end{align*}
uniformly on any compact subsets of $\R^N$. Now, we consider the following linear first order PDE:
\begin{equation}\label{eq4.27}
    \begin{cases}
\,\displaystyle\sum_{j=1}^{N}a_{i,j}\frac{\partial w(z')}{\partial {z'_{j}}}= \tilde{u}_\varepsilon^{2^{*}_\mu-2}(z')\tilde{v}_{i,\varepsilon}(z')\left(\displaystyle{\int_{\Omega_{\varepsilon}}}\frac{\tilde{u}_\varepsilon^{2^{*}_\mu}(y')}{|z{'}-y{'}|^{\mu}}dy{'}\right),\quad z'\in \R^N,\\[6mm]
\,w(z')\Big|_{\Gamma_{a}}=\frac{-1}{N+2}U_{0,1}^{2^{*}-1}(z'),
    \end{cases}
\end{equation}
where $\Gamma_{a}=\{z'\in\R^N:z'\cdot \Vec{a}=0 \}$. Utilizing Lemma 2.4 in \cite{Ta1}, we derive that there exists a unique solution $\widehat{w}_{2,\varepsilon}$  of \eqref{eq4.27} with $\widehat{w}_{2,\varepsilon}(z')=O\left({|z'|^{-(N+1)}}\right)$ as $|z'|\rightarrow+\infty$. Moreover, since
\begin{align*}
\widehat{w}_{2,\varepsilon}(z')&\longrightarrow \frac{-1}{N+2}U_{0,1}^{2^{*}-1}(z'),
    \end{align*}
uniformly on any compact subsets of $\R^N$,
we have
    \begin{align*}
\displaystyle{\int_{\Omega_\varepsilon}}\widehat{w}_{2,\varepsilon}(z')dz'&\longrightarrow \frac{-1}{N+2} \displaystyle{\int_{\R^N}}U_{0,1}^{2^{*}-1}(z')dz'=\frac{-\sigma_{N}}{N(N+2)}.
    \end{align*}
Thus, using integration by parts, we derive from \eqref{eq4.26} that
\begin{align*}
I_{2}(\varepsilon)=&\frac{\lambda_{i,\varepsilon}(2^{*}_\mu-1)}{\|u_{\varepsilon}\|_{\infty}^{2}}\int_{\Omega_{\varepsilon}}G_{\varepsilon}(x,z')\left(\sum_{j=1}^{N}a_{i,j}\frac{\partial}{\partial {z'_{j}}}\widehat{w}_{2,\varepsilon}(z')\right)dz{'}\\=&\frac{\lambda_{i,\varepsilon}(2^{*}_\mu-1)}{\|u_{\varepsilon}\|_{\infty}^{2}}\displaystyle{\int_{\Omega_{\varepsilon}}}\widehat{w}_{2,\varepsilon}(z')\left(-\sum_{j=1}^{N}a_{i,j}\frac{\partial}{\partial {z'_{j}}}G_{\varepsilon}(x,z')\right)dz{'}\\
    =&\frac{\lambda_{i,\varepsilon}(2^{*}_\mu-1)}{\|u_{\varepsilon}\|_{\infty}^{2+\frac{2}{N-2}}}\displaystyle{\int_{\Omega_{\varepsilon}}}\widehat{w}_{2,\varepsilon}(z')\left(-\sum_{j=1}^{N}a_{i,j}\frac{\partial}{\partial {z_{j}}}G(x,z)\bigg|_{z=\frac{z'}{\tau_{\varepsilon}}+x_{\varepsilon}}\right)dz{'},
\end{align*}
which implies that
\begin{align}\label{4-2}
    \|u_{\varepsilon}\|_{\infty}^{2+\frac{2}{N-2}} I_{2}(\varepsilon) \longrightarrow\frac{(2^{*}_\mu-1)\sigma_{N}}{N(N+2)}\left(\sum_{j=1}^{N}a_{i,j}\frac{\partial}{\partial {z_{j}}}G(x,z)\bigg|_{z=x_{0}}\right).
\end{align}
Similarly, for $I_{3}(\varepsilon)$, we derive that
\begin{align}\label{I-3}
\|u_{\varepsilon}\|_{\infty}^{2+\frac{2}{N-2}} I_{3}(\varepsilon) \longrightarrow\frac{\mu\sigma_{N}}{N(N-2)(N+2)}\left(\sum_{j=1}^{N}a_{i,j}\frac{\partial}{\partial {z_{j}}}G(x,z)\bigg|_{z=x_{0}}\right).
\end{align}
Hence, from \eqref{4-0}, \eqref{4-1}, \eqref{4-2} and \eqref{I-3}, we conclude that if $N\geq6$,
\begin{align*}
     \|u_{\varepsilon}\|_{\infty}^{2+\frac{2}{N-2}} v_{i,\varepsilon}(x)\longrightarrow \frac{\sigma_{N}}{N(N-2)}\left(\sum_{j=1}^{N}a_{i,j}\frac{\partial}{\partial {z_{j}}}G(x,z)\bigg|_{z=x_{0}}\right),
\end{align*}
for any $x\in\bar{\Omega}\backslash\{x_{0}\}$. Moreover, by standard elliptic estimates, we obtain the desired result.
\end{proof}

In the following, we prove \eqref{eq 1.20} in Theorem 1.2.
\begin{proof}[\textbf{Proof of \eqref{eq 1.20}.}]
    Firstly, multiplying \eqref{eqA-4} by $v_{i,\varepsilon}$ and \eqref{eq 1.12} by $\frac{\partial u_{\varepsilon}}{\partial x_j}$, we derive
        \begin{equation}\label{eq 5.3}
            \begin{split}
        &\displaystyle{\int_{\partial\Omega}}\frac{\partial v_{i,\varepsilon}(x)}{\partial\nu_{x}}\frac{\partial u_{\varepsilon}(x)}{\partial x_{j}}dS_{x}\\
        =&(1-\lambda_{i,\varepsilon})\bigg\{\varepsilon\displaystyle{\int_{\Omega}}v_{i,\varepsilon}(x)\frac{\partial u_{\varepsilon}(x)}{\partial x_{j}}dx\\
        &\quad\,\quad\,\quad\,\quad+(2^{*}_\mu-1)\displaystyle{\int_{\Omega}}v_{i,\varepsilon}(x)u_{\varepsilon}^{2^{*}_\mu-2}(x)\frac{\partial u_{\varepsilon}(x)}{\partial x_{j}}\left(\displaystyle{\int_{\Omega}}\frac{u_{\varepsilon}^{2^{*}_\mu}(y)}{|x-y|^{\mu}}dy\right)dx\\
&\quad\,\quad\,\quad\,\quad+2^{*}_\mu\displaystyle{\int_{\Omega}}u_{\varepsilon}^{2^{*}_\mu-1}(x)\frac{\partial u_{\varepsilon}(x)}{\partial x_{j}}\left(\displaystyle{\int_{\Omega}}\frac{u_{\varepsilon}^{2^{*}_\mu-1}(y)v_{i,\varepsilon}(y)}{|x-y|^{\mu}}dy\right)dx\bigg\},\quad\text{for $j=1,\cdots,N$.}
            \end{split}
        \end{equation}
Then, from \eqref{eq 2.3} and \eqref{eq 1.22}, there holds
   \begin{equation}\label{LHS}
       \begin{split}
         \text{LHS of \eqref{eq 5.3}}=&\frac{1}{\|u_\varepsilon\|_{\infty}^{3+\frac{2}{N-2}}}\displaystyle{\int_{\partial\Omega}}\left(\frac{\partial }{\partial\nu_{x}}\left(\|u_\varepsilon\|_{\infty}^{2+\frac{2}{N-2}}v_{i,\varepsilon}(x)\right)\right)\left(\frac{\partial }{\partial x_{j}}\left(\|u_\varepsilon\|_{\infty}u_{\varepsilon}(x)\right)\right)dS_{x}\\
=&\frac{\sigma_{N}^{2}\sum_{k=1}^{N}a_{i,k}}{(N-2)N^{2}\|u_\varepsilon\|_{\infty}^{3+\frac{2}{N-2}}}\left(\displaystyle{\int_{\partial\Omega}}\,\frac{\partial}{\partial\nu_{x}}\left(\frac{\partial G(x,x_0)}{\partial y_{k}}\right)\frac{\partial G(x,x_0)}{\partial x_j}dS_{x}+o(1)\right)\\  =&\frac{\sigma_{N}^{2}\left(\sum_{k=1}^{N}a_{i,k}\right)}{(N-2)N^{2}\|u_\varepsilon\|_{\infty}^{3+\frac{2}{N-2}}}\left(\frac{1}{2}\frac{\partial^{2}R(y)}{\partial y_{k}\partial y_{j}}\bigg|_{y=x_0}+o(1)\right).
       \end{split}
   \end{equation}
On the other hand, by scaling, we have
\begin{equation*}
    \begin{split}
     \text{RHS of \eqref{eq 5.3}}=&\underbrace{\frac{\varepsilon(1-\lambda_{i,\varepsilon})}{\|u_\varepsilon\|_{\infty}^{\frac{N}{N-2}}}\displaystyle{\int_{\Omega_{\varepsilon}}}\tilde{v}_{i,\varepsilon}(x)\frac{\partial \tilde{u}_{\varepsilon}(x)}{\partial x_{j}}dx}_{:=E_{1}}\\
    &+\underbrace{\frac{(1-\lambda_{i,\varepsilon})(2^{*}_\mu-1)}{\|u_\varepsilon\|_{\infty}^{\frac{N-4}{N-2}}}\displaystyle{\int_{\Omega_{\varepsilon}}}\tilde{v}_{i,\varepsilon}\tilde{u}_\varepsilon^{2^{*}_\mu-2}\left(\frac{\partial \tilde{u}_{\varepsilon}}{\partial x_i}\right)\left(\displaystyle{\int_{\Omega_{\varepsilon}}}\frac{\tilde{u}_\varepsilon^{2^{*}_{\mu}}(y)}{|x-y|^{\mu}}dy\right)dx}_{:=E_{2}}\\
    &+\underbrace{\frac{(1-\lambda_{i,\varepsilon})2^{*}_\mu}{\|u_\varepsilon\|_{\infty}^{\frac{N-4}{N-2}}}\displaystyle{\int_{\Omega_{\varepsilon}}}\tilde{u}_\varepsilon^{2^{*}_\mu-1}(x)\frac{\partial \tilde{u}_\varepsilon(x) }{\partial x_{j}}\left(\displaystyle{\int_{\Omega_{\varepsilon}}}\frac{\tilde{u}_\varepsilon^{2^{*}_\mu-1}(y)\tilde{v}_{i,\varepsilon}(y)}{|x-y|^{\mu}}dy\right)dx}_{:=E_{3}}.
    \end{split}
\end{equation*}
Note that
$$\frac{\partial U_{0,1}(x)}{\partial x_{k}}=-\frac{\partial U_{\xi,1}(x)}{\partial \xi_{k}}\Big|_{\xi=0}=\frac{(2-N)\,x_k}{\left(1+|x|^{2}\right)^{\frac{N}{2}}},$$
then from \eqref{eq 1.7}, \eqref{eq 1.9} and \eqref{eq 1.21}, we derive that
\begin{align*}
   E_{1} =&\frac{\varepsilon(1-\lambda_{i,\varepsilon})}{\|u_\varepsilon\|_{\infty}^{\frac{N}{N-2}}}\left(\displaystyle{\int_{\R^N}}\left(\sum_{k=1}^{N}a_{i,k}\frac{x_k}{\left(1+|x|^{2}\right)^{\frac{N}{2}}}\right)\frac{\partial U_{0,1}(x)}{\partial x_{j}}dx+o(1)\right)\\
=&\frac{\varepsilon(\lambda_{i,\varepsilon}-1)}{(N-2)\|u_\varepsilon\|_{\infty}^{\frac{N}{N-2}}}\left(\displaystyle{\int_{\R^N}}\left(\sum_{k=1}^{N}a_{i,k}\frac{\partial U_{0,1}(x)}{\partial x_{k}}\right)\frac{\partial U_{0,1}(x)}{\partial x_{j}}dx+o(1)\right)\\
=&\frac{\varepsilon  (\lambda_{i,\varepsilon}-1)}{(N-2)\|u_\varepsilon\|_{\infty}^{\frac{N}{N-2}}}\left(\frac{a_{i,j}}{N}\displaystyle{\int_{\R^N}}|\nabla U_{0,1}|^{2}dx+o(1)\right).
\end{align*}
\begin{align*}
    E_{2}
    =&\frac{(2^{*}_\mu-1)(\lambda_{i,\varepsilon}-1)}{(N-2)\|u_\varepsilon\|_{\infty}^{\frac{N-4}{N-2}}}\left(\frac{a_{i,j}}{N}\displaystyle{\int_{\R^N}}\left|\nabla U_{0,1}\right|^{2}U_{0,1}^{2^{*}-2}(x)dx+o(1)\right),
\end{align*}
and
\begin{align*}
    E_{3}=\frac{(2^{*}-2^{*}_\mu)(\lambda_{i,\varepsilon}-1)}{(N-2)\|u_\varepsilon\|_{\infty}^{\frac{N-4}{N-2}}}\left(\frac{a_{i,j}}{N}\displaystyle{\int_{\R^N}}\left|\nabla U_{0,1}\right|^{2}U_{0,1}^{2^{*}-2}(x)dx+o(1)\right).
\end{align*}
Here we want to mention that we utilize the similar arguments as the proof of \eqref{a-1} to estimate $E_{3}$. From above estimates of $E_{1}-E_{3}$, we find that
\begin{equation}\label{RHS}
\begin{split}
 \text{RHS of \eqref{eq 5.3}}=&\frac{\varepsilon  (\lambda_{i,\varepsilon}-1)}{(N-2)\|u_\varepsilon\|_{\infty}^{\frac{N}{N-2}}}\left(\frac{a_{i,j}}{N}\displaystyle{\int_{\R^N}}|\nabla U_{0,1}|^{2}dx+o(1)\right)\\
 &+\frac{(2^{*}-1)(\lambda_{i,\varepsilon}-1)}{(N-2)\|u_\varepsilon\|_{\infty}^{\frac{N-4}{N-2}}}\left(\frac{a_{i,j}}{N}\displaystyle{\int_{\R^N}}\left|\nabla U_{0,1}\right|^{2}U_{0,1}^{2^{*}-2}(x)dx+o(1)\right)
\end{split}
\end{equation}
Therefore, from \eqref{LHS}, \eqref{RHS} and \eqref{eq 2.4}, it holds
\begin{equation}\label{4-33}
    \begin{split}
       &\frac{\sigma_{N}^{2}}{2N^{2}(N-2)}\left({\sum_{k=1}^{N}a_{i,k}}\frac{\partial^{2}R(x_0)}{\partial y_{k}\partial y_{j}}+o(1)\right)\\
  =& (\lambda_{i,\varepsilon}-1)\|u_\varepsilon\|_{\infty}^{\frac{2N}{N-2}}\left(\frac{a_{i,j}(N+2)}{N(N-2)^{2}}\displaystyle{\int_{\R^N}}\left|\nabla U_{0,1}\right|^{2}U_{0,1}^{2^{*}-2}(x)dx+o(1)\right).
    \end{split}
\end{equation}
Moreover, since $a_{i,j}\neq 0$ for some $j$, passing to the limit as $\varepsilon\rightarrow0$, we deduce that
\begin{align}\label{eq 5.5}
(\lambda_{i,\varepsilon}-1)\|u_\varepsilon\|_{\infty}^{\frac{2N}{N-2}}\longrightarrow M\eta_{i},\quad i=2,\cdots,N+1,
\end{align}
where
$$M:=\frac{(N-2)\sigma_{N}^{2}}{2N(N+2)\left(\displaystyle{\int_{\R^N}}\left|\nabla U_{0,1}\right|^{2}U_{0,1}^{2^{*}-2}(x)dx\right)}>0,$$
and
$$\eta_{i}:=\frac{\sum_{k=1}^{N}a_{i,k}\frac{\partial^{2}R(x_0)}{\partial y_{k}\partial y_{j}}}{a_{i,j}}\quad \text{for any $j$ such that $a_{i,j}\neq0$}.$$
By the definition of $\eta_i$, if $a_{i,j}\neq0$ for some $j=1,\cdots,N$, we have
\begin{equation*}
\sum_{k=1}^{N}a_{i,k}\frac{\partial^{2}R(x_0)}{\partial y_{k}\partial y_{j}}=\eta_{i}a_{i,j},
\end{equation*}
which still holds provided $a_{i,j}=0$ in virtue of \eqref{4-33}. Hence, $D^{2}R(x_0)\cdot\vec{a}_{i}=\eta_{i}\vec{a}_{i}$, that is, $\eta_i$ is an eigenvalue of the Hessian matrix $D^{2}R(x_0)$ with $\vec{a}_{i}=\left(a_{i,1},\cdots,a_{i,N}\right)$ as the corresponding eigenvector. Furthermore, for $i\neq j$, since the eigenfunctions $v_{i,\varepsilon}$ and $v_{j,\varepsilon}$ are orthogonal in the sense of \eqref{eqq1-11}, utilizing similar arguments as in \cite{Ta2}, we derive that $\vec{a}_{i}$ and $\vec{a}_{j}$ are perpendicular to each other in $\R^N$. Denote  $\gamma_{1}\leq\gamma_{2}\cdots\leq\gamma_{N}$ as the eigenvalues of the matrix $D^{2}R(x_0)$. Thus, we have $\eta_{i}=\gamma_{l}$, for some $l\in[\,1,N\,]$. Recalling that $\lambda_{2,\varepsilon}\leq \cdots\leq\lambda_{N+1,\varepsilon}$,  from \eqref{eq 5.5}, we deduce that $\eta_{i}=\gamma_{i-1}$, for $i=2,3,\cdots,N+1$. This ends the proof of Theorem 1.2.
\end{proof}
\noindent
\section{\label{Esti}Estimates of the $(N+2)$-th eigenpaires $\big(\lambda_{N+2,\varepsilon},v_{N+2,\varepsilon}\big)$}
\setcounter{equation}{0}
\vskip 0.2cm
In this section, we prove Theorem 1.3. Firstly, we clarify the following Lemma.
\smallskip
\begin{lemma} \label{lema 5.1} As $\varepsilon\rightarrow0$, there holds
\begin{align}\label{eq 6.1}
    \lambda_{N+2,\varepsilon}\longrightarrow 1.
\end{align}
\begin{proof}
    From Lemma \ref{lema 4.2} and the fact that $\lambda_{2,\varepsilon}\leq \cdots\leq\lambda_{N+1,\varepsilon}\leq\lambda_{N+2,\varepsilon}$, we deduce that
  \begin{align*}
      \underset{\varepsilon\rightarrow0}{\liminf}\,\lambda_{N+2,\varepsilon}\geq \underset{\varepsilon\rightarrow0}{\limsup}\,\lambda_{N+1,\varepsilon}=1.
\end{align*}
To obtain \eqref{eq 6.1}, it is enough to prove that
$$\underset{\varepsilon\rightarrow0}{\limsup}\,\lambda_{N+2,\varepsilon}\leq1.$$
By the variational characterization of $\lambda_{N+2,\varepsilon}$, we have
$$\lambda_{N+2,\varepsilon}\leq \underset{w\in W}{\max}\frac{\displaystyle{\int_{\Omega}}|\nabla w|^{2}dx}{\displaystyle{\int_{\Omega}}\varepsilon w^{2}dx+\displaystyle{\int_{\Omega}}\displaystyle{\int_{\Omega}}\frac{(2^{*}_\mu-1)u_{\varepsilon}^{2^{*}_\mu-2}(x)w^{2}(x)u_{\varepsilon}^{2^{*}_\mu}(y)+2^{*}_\mu u_{\varepsilon}^{2^{*}_\mu-1}(x)w(x)u_{\varepsilon}^{2^{*}_\mu-1}(y)w(y)}{|x-y|^{\mu}}dxdy},$$
where $W$ is the subspace generated by the functions $u_{\varepsilon},\,\psi_{1,\varepsilon},\cdots,\,\psi_{N+1,\varepsilon}$ defined in \eqref{eq 4.2} and \eqref{eq 4.3}. From Lemma \ref{lema 4.1}, we know that $dimW=N+2$. Then $\lambda_{N+2,\varepsilon}$ can be estimated via the similar way in  Lemma \ref{lema 4.2}. Indeed, for any function $w\in W$ which is of the form
$$w(x)=a_{0}u_{\varepsilon}(x)+\sum_{j=1}^{N}a_{j}\psi_{j,\varepsilon}(x)+d\psi_{N+1,\varepsilon}(x):=a_{0}u_{\varepsilon}(x)+\phi(x)\widehat{z}_{\varepsilon}(x),$$
for $a_0,a_{1},\cdots,a_{N},d\in \R$, where
$$\widehat{z}_{\varepsilon}(x)={z}_{\varepsilon}(x)+d\omega_{\varepsilon}(x):=\sum_{j=1}^{N}a_{j}\frac{\partial u_{\varepsilon}(x)}{\partial x_{j}}+d\left((x-x_{\varepsilon})\cdot\nabla u_{\varepsilon}(x)+\frac{N-2}{2}u_{\varepsilon}(x)\right),$$
we will evaluate separately the numerator and the denominator for the function $w\in W$. Obviously, from \eqref{eqA-4} and \eqref{eqA-5}, we have
\begin{align*}
    -\Delta \widehat{z}_{\varepsilon}(x)=&\varepsilon \widehat{z}_{\varepsilon}(x)+(2^{*}_\mu-1)u_\varepsilon^{2^{*}_\mu-2}(x)\widehat{z}_{\varepsilon}(x)\left(\displaystyle{\int_{\Omega}}\frac{u_\varepsilon^{2^{*}_\mu}(y)}{|x-y|^{\mu}}dy\right)\\
 &+2^{*}_\mu u_\varepsilon^{2^{*}_\mu-1}(x)\left(\displaystyle{\int_{\Omega}}\frac{u_\varepsilon^{2^{*}_\mu-1}(y)\widehat{z}_{\varepsilon}(y)}{|x-y|^{\mu}}dy\right)+2d\varepsilon u_\varepsilon(x).
\end{align*}
Then, similar to \eqref{eq4.18}, a direct calculation yields that
\begin{align}\label{eq5-2}
    \lambda_{N+2,\varepsilon}\leq \underset{\left(a_{0},\,a_{1},\cdots,a_{N},\,d\right)\,\in\R^{N+2}}{\max}\left\{1+\frac{\widehat{N}_{\varepsilon}}{\widehat{D}_{\varepsilon}}\right\}.
\end{align}
Here, we write
\begin{equation}\label{hatN-ep}
\widehat{N}_{\varepsilon}:=\widehat{N}_{1,\varepsilon}+\widehat{N}_{2,\varepsilon}+\widehat{N}_{3,\varepsilon}+\widehat{N}_{4,\varepsilon},
\end{equation}
with
\begin{equation}\label{hatN-ep,1}
\widehat{N}_{1,\varepsilon}:=a_{0}^{2}(2-2\cdot2^{*}_\mu)\displaystyle{\int_{\Omega}}\displaystyle{\int_{\Omega}}\frac{u_{\varepsilon}^{2^{*}_\mu}(x)u_{\varepsilon}^{2^{*}_\mu}(y)}{|x-y|^{\mu}}dxdy,
\end{equation}
\smallskip
\begin{equation}\label{hatN-ep,2}
    \begin{split}
    \widehat{N}_{2,\varepsilon}:=&2a_{0}(2-2\cdot2^{*}_\mu)\displaystyle{\int_{\Omega}}\displaystyle{\int_{\Omega}}\frac{u_{\varepsilon}^{2^{*}_\mu-1}(x)\phi(x)\left(\sum_{j=1}^{N}a_j\frac{\partial u_\varepsilon}{\partial x_j}(x)\right)u_{\varepsilon}^{2^{*}_\mu}(y)}{|x-y|^{\mu}}dxdy\\
    &+2a_{0}d(2-2\cdot2^{*}_\mu)\underbrace{\displaystyle{\int_{\Omega}}u_{\varepsilon}^{2^{*}_\mu-1}(x)\phi(x)\omega_{\varepsilon}(x)\left(\displaystyle{\int_{\Omega}}\frac{u_\varepsilon^{2^{*}_\mu}(y)}{|x-y|^{\mu}}dy\right)dx}_{:=\widehat{N}_{2,\varepsilon}^{(1)}},
    \end{split}
\end{equation}
\begin{equation}\label{hatN-ep,3}
    \begin{split}
\widehat{N}_{3,\varepsilon}:=&\displaystyle{\int_{\Omega}}|\nabla\phi|^{2}\left(\sum_{j=1}^{N}a_j\frac{\partial u_\varepsilon}{\partial x_j}(x)\right)\left(\sum_{l=1}^{N}a_l\frac{\partial u_\varepsilon}{\partial x_l}(x)\right)dx\\
&+2^{*}_\mu\displaystyle{\int_{\Omega}}\displaystyle{\int_{\Omega}}\frac{\phi(x)\left(\phi(x)-\phi(y)\right)u_{\varepsilon}^{2^{*}_\mu-1}(x)\left(\sum_{j=1}^{N}a_j\frac{\partial u_\varepsilon}{\partial x_j}(x)\right)u_{\varepsilon}^{2^{*}_\mu-1}(y)\left(\sum_{l=1}^{N}a_l\frac{\partial u_\varepsilon}{\partial y_l}(y)\right)}{|x-y|^{\mu}}dxdy\\
&+\underbrace{d^{2}\displaystyle{\int_{\Omega}}|\nabla\phi|^{2}\omega_{\varepsilon}^{2}(x)dx}_{:=\widehat{N}_{3,\varepsilon}^{(1)}}+\underbrace{2d\displaystyle{\int_{\Omega}}|\nabla\phi|^{2}\left(\sum_{j=1}^{N}a_j\frac{\partial u_\varepsilon}{\partial x_j}(x)\right)\omega_{\varepsilon}(x)dx}_{:={\widehat{N}_{3,\varepsilon}^{(2)}}}\\
&+\underbrace{2^{*}_{\mu}d^{2}\displaystyle{\int_{\Omega}}\displaystyle{\int_{\Omega}}\frac{\phi(x)\left(\phi(x)-\phi(y)\right)u_{\varepsilon}^{2^{*}_\mu-1}(x)\omega_{\varepsilon}(x)u_{\varepsilon}^{2^{*}_\mu-1}(y)\omega_{\varepsilon}(y)}{|x-y|^{\mu}}dxdy}_{{:=\widehat{N}_{3,\varepsilon}^{(3)}}}\\
&+\underbrace{2^{*}_{\mu}d\displaystyle{\int_{\Omega}}\displaystyle{\int_{\Omega}}\frac{\phi(x)\left(\phi(x)-\phi(y)\right)u_{\varepsilon}^{2^{*}_\mu-1}(x)u_{\varepsilon}^{2^{*}_\mu-1}(y)\left(\sum_{j=1}^{N}a_j\frac{\partial u_\varepsilon}{\partial x_j}(x)\right)\omega_{\varepsilon}(y)}{|x-y|^{\mu}}dxdy}_{:=\widehat{N}_{3,\varepsilon}^{(4)}}\\
&+\underbrace{2^{*}_{\mu}d\displaystyle{\int_{\Omega}}\displaystyle{\int_{\Omega}}\frac{\phi(x)\left(\phi(x)-\phi(y)\right)u_{\varepsilon}^{2^{*}_\mu-1}(x)u_{\varepsilon}^{2^{*}_\mu-1}(y)\left(\sum_{l=1}^{N}a_l\frac{\partial u_\varepsilon}{\partial y_l}(y)\right)\omega_{\varepsilon}(x)}{|x-y|^{\mu}}dxdy}_{{:=\widehat{N}_{3,\varepsilon}^{(5)}}},
    \end{split}
\end{equation}
\begin{equation}\label{hatN-ep,4}
\widehat{N}_{4,\varepsilon}:=2\varepsilon d\displaystyle{\int_{\Omega}}\phi^{2}(x) u_{\varepsilon}(x)\left(\sum_{j=1}^{N}a_j\frac{\partial u_\varepsilon}{\partial x_j}(x)\right)dx+2\varepsilon d^{2}\displaystyle{\int_{\Omega}}\phi^{2}(x) u_{\varepsilon}(x)\omega_{\varepsilon}(x)dx.
\end{equation}
Moreover, we denote
\begin{equation}\label{hatD-ep}
\widehat{D}_{\varepsilon}:=\widehat{D}_{1,\varepsilon}+\widehat{D}_{2,\varepsilon}+\widehat{D}_{3,\varepsilon},
\end{equation}
where
\begin{equation}\label{hatD-ep,1}
\widehat{D}_{1,\varepsilon}:=a_{0}^{2} \left(\varepsilon \displaystyle{\int_{\Omega}}u_{\varepsilon}^{2}(x)dx+(2\cdot2^{*}_\mu-1)\displaystyle{\int_{\Omega}}\displaystyle{\int_{\Omega}}\frac{u_{\varepsilon}^{2^{*}_\mu}(x)u_{\varepsilon}^{2^{*}_\mu}(y)}{|x-y|^{\mu}}dxdy\right),
\end{equation}
\smallskip
\begin{equation}\label{hatD-ep,2}
    \begin{split}
\widehat{D}_{2,\varepsilon}=&2a_0\varepsilon \displaystyle{\int_{\Omega}}
 u_{\varepsilon}(x)\phi(x)\left(\sum_{j=1}^{N}a_j\frac{\partial u_\varepsilon}{\partial x_j}(x)\right)dx+\underbrace{2a_0\varepsilon d\displaystyle{\int_{\Omega}}\phi(x) u_{\varepsilon}(x)\omega_{\varepsilon}(x)dx}_{:=\widehat{D}_{2,\varepsilon}^{(1)}}\\
&+2a_0(2\cdot2^{*}_\mu-1)\displaystyle{\int_{\Omega}}\displaystyle{\int_{\Omega}}\frac{u_{\varepsilon}^{2^{*}_\mu-1}(x)\phi(x)\left(\sum_{j=1}^{N}a_j\frac{\partial u_\varepsilon}{\partial x_j}(x)\right)u_{\varepsilon}^{2^{*}_\mu}(y)}{|x-y|^{\mu}}dxdy\\
&+2a_{0}d(2\cdot2^{*}_\mu-1)\underbrace{\displaystyle{\int_{\Omega}}\displaystyle{\int_{\Omega}}\frac{u_{\varepsilon}^{2^{*}_\mu-1}(x)\phi(x)\omega_{\varepsilon}(x)u_{\varepsilon}^{2^{*}_\mu}(y)}{|x-y|^{\mu}}dxdy}_{:=\widehat{D}_{2,\varepsilon}^{(2)}},
    \end{split}
\end{equation}
\begin{align}
\widehat{D}_{3,\varepsilon}=&\varepsilon \displaystyle{\int_{\Omega}}\phi^{2}(x)\left(\sum_{j=1}^{N}a_j\frac{\partial u_\varepsilon}{\partial x_j}(x)\right)\left(\sum_{l=1}^{N}a_l\frac{\partial u_\varepsilon}{\partial x_l}(x)\right)dx\notag\\
&+(2^{*}_\mu-1)\displaystyle{\int_{\Omega}}\displaystyle{\int_{\Omega}}\frac{u_{\varepsilon}^{2^{*}_\mu-2}(x)\phi^{2}(x)\left(\sum_{j=1}^{N}a_j\frac{\partial u_\varepsilon}{\partial x_j}(x)\right)\left(\sum_{l=1}^{N}a_l\frac{\partial u_\varepsilon}{\partial x_l}(x)\right)u_{\varepsilon}^{2^{*}_\mu}(y)}{|x-y|^{\mu}}dxdy\notag\\
&+2^{*}_\mu\displaystyle{\int_{\Omega}}u_{\varepsilon}^{2^{*}_\mu-1}(x)\phi(x)\left(\sum_{j=1}^{N}a_j\frac{\partial u_\varepsilon}{\partial x_j}(x)\right)\left(\displaystyle{\int_{\Omega}}\frac{u_{\varepsilon}^{2^{*}_\mu-1}(y)\phi(y)\left(\sum_{l=1}^{N}a_l\frac{\partial u_\varepsilon}{\partial y_l}(y)\right)}{|x-y|^{\mu}}dy\right)dx\notag\\
&+\underbrace{\varepsilon d^{2} \displaystyle{\int_{\Omega}}\phi^{2}(x)\omega_{\varepsilon}^{2}(x)dx}_{:=\widehat{D}_{3,\varepsilon}^{(1)}}+\underbrace{2\varepsilon d\displaystyle{\int_{\Omega}}\phi^{2}(x)\omega_{\varepsilon}(x)\left(\sum_{j=1}^{N}a_j\frac{\partial u_\varepsilon}{\partial x_j}(x)\right)dx}_{:=\widehat{D}_{3,\varepsilon}^{(2)}}\notag\\
&+\underbrace{(2^{*}_\mu-1)d^{2}\displaystyle{\int_{\Omega}}\displaystyle{\int_{\Omega}}\frac{u_{\varepsilon}^{2^{*}_\mu-2}(x)\phi^{2}(x)\omega_{\varepsilon}^{2}(x)u_{\varepsilon}^{2^{*}_\mu}(y)}{|x-y|^{\mu}}dxdy}_{:=\widehat{D}_{3,\varepsilon}^{(3)}}\notag\\
&+\underbrace{2^{*}_{\mu}d^{2}\displaystyle{\int_{\Omega}}u_{\varepsilon}^{2^{*}_\mu-1}(x)\phi(x)\omega_{\varepsilon}(x)\left(\displaystyle{\int_{\Omega}}\frac{u_{\varepsilon}^{2^{*}_\mu-1}(y)\phi(y)\omega_{\varepsilon}(y)}{|x-y|^{\mu}}dy\right)dx}_{:=\widehat{D}_{3,\varepsilon}^{(4)}}\notag\\
&+2\cdot(2^{*}_\mu-1)d\underbrace{\displaystyle{\int_{\Omega}}\displaystyle{\int_{\Omega}}\frac{u_{\varepsilon}^{2^{*}_\mu}(y)u_{\varepsilon}^{2^{*}_\mu-2}(x)\phi^{2}(x)\left(\sum_{j=1}^{N}a_j\frac{\partial u_\varepsilon}{\partial x_j}(x)\omega_{\varepsilon}(x)\right)}{|x-y|^{\mu}}dxdy}_{:=\widehat{D}_{3,\varepsilon}^{(5)}}\notag\\
&+2\cdot2^{*}_{\mu}d\underbrace{\displaystyle{\int_{\Omega}}u_{\varepsilon}^{2^{*}_\mu-1}(x)\phi(x)\left(\sum_{j=1}^{N}a_j\frac{\partial u_\varepsilon}{\partial x_j}(x)\right)\left(\displaystyle{\int_{\Omega}}\frac{u_{\varepsilon}^{2^{*}_\mu-1}(y)\phi(y)\omega_{\varepsilon}(y)}{|x-y|^{\mu}}dy\right)dx}_{:=\widehat{D}_{3,\varepsilon}^{(6)}}.\label{hatD-ep,3}
\end{align}

In the following, we evaluate separately each term in $\widehat{N}_{\varepsilon}$ and $\widehat{D}_{\varepsilon}$, which are needed to compute the quotient in \eqref{eq5-2}. Let $(a_0,a_{1},\cdots,a_{N},d)$ denote a maximizer of $\underset{a_{0},a_{1},\cdots,a_{N},d}{\max}\left(1+\frac{\widehat{N}_{\varepsilon}}{\widehat{D}_{\varepsilon}}\right)$, which is normalized as $a_{0}^{2}+\sum_{j=1}^{N}a_{j}^{2}+d^2=1$. Since $\lambda_{1,\varepsilon}<1/\left(2\cdot2^{*}_\mu-1\right)$ provided $a_{0}=1$, we only consider the case $\sum_{j=1}^{N}a_{j}^{2}+d^2\neq 0$.  As for $\widehat{D}_{2,\varepsilon}^{(1)}$ given in \eqref{hatD-ep,2}, a direct computation yields that
\begin{equation}\label{5-12}
\begin{split}
\widehat{D}_{2,\varepsilon}^{(1)}=&a_{0}d\,\varepsilon\displaystyle{\int_{\Omega}}\phi(x)\left(\sum_{j=1}^{N}\left(x_j-x_{\varepsilon,j}\right)\frac{\partial}{\partial x_{j}u_{\varepsilon}^{2}(x)}\right)dx+(N-2)a_{0}d\,\varepsilon\displaystyle{\int_{\Omega}}\phi(x)u_{\varepsilon}^{2}(x)\\
=&a_{0}d\,\varepsilon\displaystyle{\int_{\Omega}}\left(-\sum_{j=1}^{N}\left(x_j-x_{\varepsilon,j}\right)\frac{\partial\phi(x)}{\partial x_{j}}\right)u_{\varepsilon}^{2}(x)dx-\frac{2a_{0}d\,\varepsilon}{\|u_{\varepsilon}\|_{\infty}^{2^{*}-2}}\displaystyle{\int_{\Omega_{\varepsilon}}}\phi(\tau_{\varepsilon}^{-1}x+x_{\varepsilon})\tilde{u}_{\varepsilon}^{2}(x)dx\\
=&O\left(\frac{\varepsilon}{\|u_{\varepsilon}\|_{\infty}^{2}}\right)+O\left(\frac{\varepsilon}{\|u_{\varepsilon}\|_{\infty}^{2^{*}-2}}\left(\displaystyle{\int_{\R^N}}U_{0,1}^{2}(x)dx+o(1)\right)\right)=O\left(\frac{1}{\|u_{\varepsilon}\|_{\infty}^{2}}\right),
\end{split}
\end{equation}
where the last equality is derived from \eqref{eq 2.4}. For the term $\widehat{D}_{2,\varepsilon}^{(2)}$, we get that
\begin{equation}\label{hat_D,2}
    \begin{split}
\widehat{D}_{2,\varepsilon}^{(2)}=&\underbrace{\displaystyle{\int_{B_{\rho}(x_{\varepsilon})}}\displaystyle{\int_{B_{\rho}(x_{\varepsilon})}}\frac{u_{\varepsilon}^{2^{*}_\mu-1}(x)\phi(x)\omega_{\varepsilon}(x)u_{\varepsilon}^{2^{*}_\mu}(y)}{|x-y|^{\mu}}dxdy}_{:=M_{1}}\\
&+\underbrace{\displaystyle{\int_{\Omega\backslash B_{\rho}(x_{\varepsilon})}}\displaystyle{\int_{B_{\rho}(x_{\varepsilon})}}\frac{u_{\varepsilon}^{2^{*}_\mu-1}(x)\phi(x)\omega_{\varepsilon}(x)u_{\varepsilon}^{2^{*}_\mu}(y)}{|x-y|^{\mu}}dxdy}_{:=M_{2}}\\
&+\underbrace{\displaystyle{\int_{B_{\rho}(x_{\varepsilon})}}\displaystyle{\int_{\Omega\backslash B_{\rho}(x_{\varepsilon})}}\frac{u_{\varepsilon}^{2^{*}_\mu-1}(x)\phi(x)\omega_{\varepsilon}(x)u_{\varepsilon}^{2^{*}_\mu}(y)}{|x-y|^{\mu}}dxdy}_{:=M_{3}}\\
&+\underbrace{\displaystyle{\int_{\Omega\backslash B_{\rho}(x_{\varepsilon})}}\displaystyle{\int_{\Omega\backslash B_{\rho}(x_{\varepsilon})}}\frac{u_{\varepsilon}^{2^{*}_\mu-1}(x)\phi(x)\omega_{\varepsilon}(x)u_{\varepsilon}^{2^{*}_\mu}(y)}{|x-y|^{\mu}}dxdy}_{:=M_{4}}.
    \end{split}
\end{equation}
Note that for any $x\in B_{\rho}(x_{\varepsilon})$, $\phi(x)\equiv1$. Then, using the integration by parts, we have
\begin{equation}\label{m1}
    \begin{split}
     M_{1}=&\displaystyle{\int_{B_{\rho}(x_{\varepsilon})}}\displaystyle{\int_{B_{\rho}(x_{\varepsilon})}}\frac{u_{\varepsilon}^{2^{*}_\mu-1}(x)\left(\left(x-x_\varepsilon\right)\cdot\nabla {u}_{\varepsilon}(x)+\frac{N-2}{2}{u}_{\varepsilon}(x)\right)u_{\varepsilon}^{2^{*}_\mu}(y)}{|x-y|^{\mu}}dxdy\\
    =&\frac{-N}{2^{*}_\mu}\displaystyle{\int_{B_{\rho}(x_{\varepsilon})}}\displaystyle{\int_{B_{\rho}(x_{\varepsilon})}}\frac{u_{\varepsilon}^{2^{*}_\mu}(x)u_{\varepsilon}^{2^{*}_\mu}(y)}{|x-y|^{\mu}}dxdy\\
    &+\frac{\mu}{2^{*}_\mu}\underbrace{\displaystyle{\int_{B_{\rho}(x_{\varepsilon})}}\displaystyle{\int_{B_{\rho}(x_{\varepsilon})}}\frac{(x-y)\cdot(x-x_{\varepsilon})}{|x-y|^{\mu+2}}u_{\varepsilon}^{2^{*}_\mu}(x)u_{\varepsilon}^{2^{*}_\mu}(y)dxdy}_{:=M_{1,1}}\\
    &+\underbrace{\frac{1}{2^{*}_\mu}\int_{\partial B_{\rho}(x_\varepsilon)}\int_{B_{\rho}(x_\varepsilon)}\frac{u_{\varepsilon}^{2^{*}_\mu}(x)u_{\varepsilon}^{2^{*}_\mu}(y)(x-x_\varepsilon)\cdot\nu}{|x-y|^{\mu}}dydS_{x}}_{:=M_{1,2}}\\
    &+\frac{N-2}{2}\displaystyle{\int_{B_{\rho}(x_{\varepsilon})}}\displaystyle{\int_{B_{\rho}(x_{\varepsilon})}}\frac{u_{\varepsilon}^{2^{*}_\mu}(x)u_{\varepsilon}^{2^{*}_\mu}(y)}{|x-y|^{\mu}}dxdy.
    \end{split}
\end{equation}
Due to the symmetry property, we have
\begin{align*}
     &\displaystyle{\int_{B_{\rho}(x_{\varepsilon})}}\displaystyle{\int_{B_{\rho}(x_{\varepsilon})}}\frac{(y-x)\cdot(y-x_{\varepsilon})}{|x-y|^{\mu+2}}\,u_{\varepsilon}^{2^{*}_\mu}(x)u_{\varepsilon}^{2^{*}_\mu}(y)dxdy\\
     =&-\displaystyle{\int_{B_{\rho}(x_{\varepsilon})}}\displaystyle{\int_{B_{\rho}(x_{\varepsilon})}}\frac{(y-x)\cdot(x_{\varepsilon}-y)}{|x-y|^{\mu+2}}\,u_{\varepsilon}^{2^{*}_\mu}(x)u_{\varepsilon}^{2^{*}_\mu}(y)dxdy,
\end{align*}
then a direct calculation yields that
\begin{equation}\label{eq 6.2}
    \begin{split}
      M_{1,1}=&\frac{1}{2}\displaystyle{\int_{B_{\rho}(x_{\varepsilon})}}
      \displaystyle{\int_{B_{\rho}(x_{\varepsilon})}}\frac{(x-y)\cdot(x-x_{\varepsilon})-(y-x)\cdot(x_{\varepsilon}-y)}{|x-y|^{\mu+2}}
      u_{\varepsilon}^{2^{*}_\mu}(x)u_{\varepsilon}^{2^{*}_\mu}(y)dxdy\\
    =&\frac{1}{2}\displaystyle{\int_{B_{\rho}(x_{\varepsilon})}}\displaystyle{\int_{B_{\rho}(x_{\varepsilon})}}\frac{u_{\varepsilon}^{2^{*}_\mu}(x)u_{\varepsilon}^{2^{*}_\mu}(y)}{|x-y|^{\mu}}dxdy.
    \end{split}
\end{equation}
Moreover, 
from \eqref{eq 1.7} and \eqref{eq 2.2}, we deduce that
\begin{equation}\label{5-14}
    \begin{split}
        M_{1,2}\leq&C\displaystyle{\int_{\partial B_{\rho}(x_\varepsilon)}}U_{x_{\varepsilon},\tau_{\varepsilon}}^{2^{*}_\mu}(x)\left(\displaystyle{\int_{B_{\rho}(x_\varepsilon)}}\frac{U_{x_{\varepsilon},\tau_{\varepsilon}}^{2^{*}_\mu}(y)}{|x-y|^{\mu}}dy\right)dS_{x}\leq C\displaystyle{\int_{\partial B_{\rho}(x_\varepsilon)}}U_{x_{\varepsilon},\tau_{\varepsilon}}^{2^{*}}(x)dS_{x}\\
=&O\left(\frac{1}{\|u_\varepsilon\|_{\infty}^{2^{*}}}\right)=o\left(\frac{1}{\|u_\varepsilon\|_{\infty}^{2^{*}_\mu}}\right).
    \end{split}
\end{equation}
Thus, from \eqref{m1}, \eqref{eq 6.2} and \eqref{5-14}, we get that
\begin{equation*}
\begin{split}
  M_{1}=&\left(\frac{-N}{2^{*}_\mu}+\frac{\mu}{2\cdot2^{*}_\mu}+\frac{N-2}{2}\right)\displaystyle{\int_{B_{\rho}(x_{\varepsilon})}}\displaystyle{\int_{B_{\rho}(x_{\varepsilon})}}\frac{u_{\varepsilon}^{2^{*}_\mu}(x)u_{\varepsilon}^{2^{*}_\mu}(y)}{|x-y|^{\mu}}dxdy+o\left(\frac{1}{\|u_\varepsilon\|_{\infty}^{2^{*}_\mu}}\right)\\
  =&0+o\left(\frac{1}{\|u_\varepsilon\|_{\infty}^{2^{*}_\mu}}\right)=o\left(\frac{1}{\|u_\varepsilon\|_{\infty}^{2^{*}_\mu}}\right).
\end{split}
\end{equation*}
As for $M_{2}$, using HLS, H$\ddot{o}$lder inequality and the Sobolev inequality, we obtain that
\begin{equation*}
    \begin{split}
      M_{2}\leq &C\left(\displaystyle{\int_{ B_{\rho}(x_{\varepsilon})}}\left|u_{\varepsilon}^{2^{*}_\mu-1}(x)\phi(x)\omega_{\varepsilon}(x)\right|^{\frac{2N}{2N-\mu}}dx\right)^{\frac{2N-\mu}{2N}}\left(\displaystyle{\int_{\Omega\backslash B_{\rho}(x_{\varepsilon})}}\left|u_{\varepsilon}(y)\right|^{\frac{2N}{N-2}}dy\right)^{\frac{2N-\mu}{2N}}\\
\leq&C\left(\displaystyle{\int_{ B_{\rho}(x_{\varepsilon})}}\left|u_{\varepsilon}(x)\right|^{2^{*}}\right)^{\frac{N-\mu+2}{2N}}\!\left(\displaystyle{\int_{ B_{\rho}(x_{\varepsilon})}}\left|\omega_{\varepsilon}(x)\right|^{2^{*}}\right)^{\frac{N-2}{2N}}\!\left(\displaystyle{\int_{\Omega\backslash B_{\rho}(x_{\varepsilon})}}U_{x_\varepsilon,\tau_{\varepsilon}}^{2^{*}}(y)\right)^{\frac{2N-\mu}{2N}}\\
\leq &C \,\|u_\varepsilon\|_{H^{1}_{0}(\Omega)}^{\frac{N-\mu+2}{N-2}}\,\,\|\omega_\varepsilon\|_{H^{1}_{0}(\Omega)}^{\frac{N-\mu+2}{N-2}}\left(\displaystyle{\int_{\Omega\backslash B_{\rho}(x_{\varepsilon})}}U_{x_\varepsilon,\tau_{\varepsilon}}^{2^{*}}(y)\right)^{\frac{2N-\mu}{2N}}=O\left(\frac{1}{\|u_{\varepsilon}\|_{\infty}^{2^{*}_\mu}}\right).
\end{split}
\end{equation*}
For $M_{3}$, using HLS, we have
\begin{equation*}
    \begin{split}
      M_{3}\leq& C\left(\displaystyle{\int_{ \Omega\backslash B_{\rho}(x_{\varepsilon})}}\left|u_{\varepsilon}^{2^{*}_\mu-1}(x)\omega_{\varepsilon}(x)\right|^{\frac{2N}{2N-\mu}}dx\right)^{\frac{2N-\mu}{2N}}\left(\displaystyle{\int_{ B_{\rho}(x_{\varepsilon})}}\left|u_{\varepsilon}(y)\right|^{\frac{2N}{N-2}}dy\right)^{\frac{2N-\mu}{2N}} \\
      \leq &C\left(\displaystyle{\int_{\Omega_{\varepsilon}\backslash B_{\tau_{\varepsilon}\rho}(0)}}\left|\tilde{u}_{\varepsilon}^{2^{*}_\mu-1}(x)\left(x\cdot\nabla\tilde{u}_{\varepsilon}(x)+\frac{N-2}{2}\tilde{u}_{\varepsilon}(x)\right)\right|^{\frac{2N}{2N-\mu}}dx\right)^{\frac{2N-\mu}{2N}}\\
\leq&C\left(\displaystyle{\int_{\Omega_{\varepsilon}\backslash B_{\tau_{\varepsilon}\rho}(0)}}\left|U_{0,1}^{2^{*}_\mu-1}(x)\left(\frac{\partial U_{0,\tau}(x)}{\partial\tau}\Big|_{\tau=1}\right)\right|^{\frac{2N}{2N-\mu}}dx\right)^{\frac{2N-\mu}{2N}}\\
\leq& C\left(\displaystyle{\int_{\Omega_{\varepsilon}\backslash B_{\tau_{\varepsilon}\rho}(0)}}U_{0,1}^{2^{*}}(x)dx\right)^{\frac{2N-\mu}{2N}}\leq C\left(\displaystyle{\int_{\R^N\backslash B_{\tau_{\varepsilon}\rho}(0)}}U_{0,1}^{2^{*}}(x)dx\right)^{\frac{2N-\mu}{2N}}=O\left(\frac{1}{\|u_{\varepsilon}\|_{\infty}^{2^{*}_\mu}}\right).
    \end{split}
\end{equation*}
Similarly, we have $M_{4}=O\left({1}/{\|u_{\varepsilon}\|_{\infty}^{2\cdot2^{*}_\mu}}\right)$. Thus, from \eqref{hat_D,2} and the estimates of $M_{1}-M_{4}$, we get that
\begin{equation}\label{5-21}
   \widehat{D}_{2,\varepsilon}^{(2)}=O\left(\frac{1}{\|u_{\varepsilon}\|_{\infty}^{2^{*}_\mu}}\right).
\end{equation}
Therefore, from \eqref{5-12}, \eqref{5-21} and the previous estimates of $D_{2,\varepsilon}$ in \eqref{4-21}, we derive that
\begin{align}
\widehat{D}_{2,\varepsilon}=&O\left(\frac{\varepsilon }{\|u_{\varepsilon}\|_{\infty}^{2}}\right)+O\left(\frac{1}{\|u_{\varepsilon}\|_{\infty}^{2^{*}_{\mu}}}\right)+O\left(\frac{1}{\|u_{\varepsilon}\|_{\infty}^{2}}\right)\notag\\
    =&O\left(\frac{1}{\|u_{\varepsilon}\|_{\infty}^{2}}\right),\label{hatd,2}
\end{align}
 provided $0<\mu<4$. As for $\widehat{D}_{3,\varepsilon}$ given in \eqref{hatD-ep,3}, from \eqref{eq 2.4}, a direct calculation yields that
\begin{align}
\widehat{D}_{3,\varepsilon}^{(1)}=&\frac{\varepsilon d^{2}}{\|u_\varepsilon\|_{\infty}^{2^{*}-2}}\displaystyle{\int_{\Omega_{\varepsilon}}}\phi^{2}(\tau_{\varepsilon}^{-1}x+x_{\varepsilon})\left(x\cdot\nabla\tilde{u}_{\varepsilon}(x)+\frac{N-2}{2}\tilde{u}_{\varepsilon}(x)\right)^{2}dx\notag\\
=&\frac{\varepsilon d^{2}}{\|u_\varepsilon\|_{\infty}^{2^{*}-2}}\left(\displaystyle{\int_{\R^N}}\frac{\left(1-|x|^{2}\right)^{2}}{(1+|x|^{2})^{N}}dx+o(1)\right)\notag\\
=&O\left(\frac{\varepsilon }{\|u_{\varepsilon}\|_{\infty}^{2^{*}-2}}\right)=O\left(\frac{1}{\|u_{\varepsilon}\|_{\infty}^{2}}\right).\label{5.19}
\end{align}
Moreover,
\begin{align}
\widehat{D}_{3,\varepsilon}^{(2)}=&\frac{2\varepsilon d\sum_{j=1}^{N}a_{j}}{\|u_\varepsilon\|_{\infty}^{\frac{2}{N-2}}}\displaystyle{\int_{\Omega_{\varepsilon}}}\phi^{2}(\tau_{\varepsilon}^{-1}x+x_{\varepsilon})\,\frac{\partial \tilde{u}_\varepsilon (x)}{\partial x_{j}}\left(x\cdot\nabla\tilde{u}_{\varepsilon}(x)+\frac{N-2}{2}\tilde{u}_{\varepsilon}(x)\right)dx\notag\\
=&\frac{2\varepsilon d\sum_{j=1}^{N}a_{j}}{\|u_\varepsilon\|_{\infty}^{\frac{2}{N-2}}}\left(\displaystyle{\int_{\R^N}}\,\frac{\partial U_{0,1} (x)}{\partial x_{j}}\,\frac{1-|x|^{2}}{(1+|x|^{2})^{N}}\,dx+o(1)\right)\notag\\
=&0+o\left(\frac{\varepsilon}{\|u_{\varepsilon}\|_{\infty}^{\frac{2}{N-2}}}\right)=o\left(\frac{\varepsilon}{\|u_{\varepsilon}\|_{\infty}^{\frac{2}{N-2}}}\right).\label{D-2}
\end{align}
For the other terms in $\widehat{D}_{3,\varepsilon}$, we derive that
\begin{align}
&\widehat{D}_{3,\varepsilon}^{(3)}\notag\\=&(2^{*}_\mu-1)d^{2}\displaystyle{\int_{\Omega_{\varepsilon}}}\displaystyle{\int_{\Omega_{\varepsilon}}}\frac{\tilde{u}_{\varepsilon}^{2^{*}_\mu-2}(x){\phi^{2}(\tau_{\varepsilon}^{-1}x+x_\varepsilon)}\left(x\cdot\nabla \tilde{u}_{\varepsilon}(x)+\frac{N-2}{2}\tilde{u}_{\varepsilon}(x)\right)^{2}\tilde{u}_{\varepsilon}^{2^{*}_\mu}(y)}{|x-y|^{\mu}}dxdy\notag\\
=&(2^{*}_\mu-1)d^{2}\left(\frac{N-2}{2}\right)^{2}\left(\displaystyle{\int_{\R^N}}U_{0,1}^{2^{*}_\mu-2}(x)\frac{\left(1-|x|^{2}\right)^{2}}{(1+|x|^{2})^{N}}\left(\displaystyle{\int_{\R^N}}\frac{U_{0,1}^{2^{*}_\mu}(y)}{|x-y|^{\mu}}dy\right)dx+o(1)\right)\notag\\
=&(2^{*}_\mu-1)d^{2}\left(\frac{N-2}{2}\right)^{2}\left(\displaystyle{\int_{\R^N}}U_{0,1}^{2^{*}-2}(x)\frac{\left(1-|x|^{2}\right)^{2}}{(1+|x|^{2})^{N}}dx+o(1)\right),\label{eq 6.3}
\end{align}
and
\begin{align}
&\widehat{D}_{3,\varepsilon}^{(4)}\notag\\
=&2^{*}_{\mu}d^{2}\Bigg[\displaystyle{\int_{\Omega_{\varepsilon}}}\tilde{u}_{\varepsilon}^{2^{*}_\mu-1}(x)\phi(\tau_{\varepsilon}^{-1}x+x_\varepsilon)\left(x\cdot\nabla \tilde{u}_{\varepsilon}(x)+\frac{N-2}{2}\tilde{u}_{\varepsilon}(x)\right)\notag\\
&\quad\,\,\quad\,\,\quad\left(\displaystyle{\int_{\Omega_{\varepsilon}}}\frac{\tilde{u}_{\varepsilon}^{2^{*}_\mu-1}(y)\phi(\tau_{\varepsilon}^{-1}y+x_\varepsilon)\left(y\cdot\nabla \tilde{u}_{\varepsilon}(y)+\frac{N-2}{2}\tilde{u}_{\varepsilon}(y)\right)}{|x-y|^{\mu}}dy\right)dx\Bigg]\notag\\
=&2^{*}_{\mu}d^{2}\Bigg[\displaystyle{\int_{\R^N}}U_{0,1}^{2^{*}_\mu-1}(x)\left(\frac{N-2}{2}\frac{1-|x|^{2}}{(1+|x|^{2})^{\frac{N}{2}}}\right)
\notag\\
&\quad\,\,\quad\,\,\quad\left(\displaystyle{\int_{\R^N}}\frac{U_{0,1}^{2^{*}_\mu-1}(y)}{|x-y|^{\mu}}
\left(\frac{N-2}{2}\frac{1-|y|^{2}}{(1+|y|^{2})^{\frac{N}{2}}}\right)dy\right)dx+o(1)\Bigg]\notag\\
\underset{(*)}{=}&\left(2^{*}-2^{*}_\mu\right)d^{2}\left(\frac{N-2}{2}\right)^{2}\left(\displaystyle{\int_{\R^N}}U_{0,1}^{2^{*}-2}(x)\frac{\left(1-|x|^{2}\right)^{2}}{(1+|x|^{2})^{N}}dx+o(1)\right) ,\label{eq 6.4}
\end{align}
where $(*)$ can be obtained via the similar calculation as the proof of \eqref{a-2}. In addition, we find
\begin{align}
       &\widehat{D}_{3,\varepsilon}^{(5)}\notag\\=&\|u_{\varepsilon}\|_{\infty}^{\frac{2}{N-2}}\Bigg[\displaystyle{\int_{\Omega_{\varepsilon}}}\tilde{u}_{\varepsilon}^{2^{*}_\mu-2}(x)\phi^{2}(\tau_{\varepsilon}^{-1}x+x_\varepsilon)\left(\sum_{j=1}^{N}a_j\frac{\partial \tilde{u}_\varepsilon}{\partial x_j}(x)\right)\left(x\cdot\nabla \tilde{u}_{\varepsilon}(x)+\frac{N-2}{2}\tilde{u}_{\varepsilon}(x)\right)\notag\\
&\quad\,\,\quad\,\,\quad\,\quad\left(\displaystyle{\int_{\Omega_{\varepsilon}}}\frac{\tilde{u}_{\varepsilon}^{2^{*}_\mu}(y)}{|x-y|^{\mu}}dy\right)dx\Bigg]\notag\\
   =&\|u_{\varepsilon}\|_{\infty}^{\frac{2}{N-2}}\displaystyle{\int_{\R^N}}U_{0,1}^{2^{*}-2}(x)\left(\sum_{j=1}^{N}a_j\frac{(2-N)x_{j}}{(1+|x|^{2})^{\frac{N}{2}}}\right)\left(\frac{N-2}{2}\frac{1-|x|^{2}}{(1+|x|^{2})^{\frac{N}{2}}}\right)dx+o\left(\|u_{\varepsilon}\|_{\infty}^{\frac{2}{N-2}}\right)\notag\\
=&0+o\left(\|u_{\varepsilon}\|_{\infty}^{\frac{2}{N-2}}\right)=o\left(\|u_{\varepsilon}\|_{\infty}^{\frac{2}{N-2}}\right).\label{5-22}
    \end{align}
Moreover, 
we have
\begin{align}
&\widehat{D}_{3,\varepsilon}^{(6)}\notag\\=&\|u_{\varepsilon}\|_{\infty}^{\frac{2}{N-2}}\Bigg[\displaystyle{\int_{\Omega_{\varepsilon}}}\tilde{u}_{\varepsilon}^{2^{*}_\mu-1}(x)\phi(\tau_{\varepsilon}^{-1}x+x_\varepsilon)\left(\sum_{j=1}^{N}a_j\frac{\partial \tilde{u}_\varepsilon}{\partial x_j}(x)\right)\notag\\
&\quad\,\,\quad\,\,\quad\,\quad\left(\displaystyle{\int_{\Omega_{\varepsilon}}}\frac{\tilde{u}_{\varepsilon}^{2^{*}_\mu-1}(y)\phi(\tau_{\varepsilon}^{-1}y+x_\varepsilon)\left(y\cdot\nabla \tilde{u}_{\varepsilon}(y)+\frac{N-2}{2}\tilde{u}_{\varepsilon}(y)\right)}{|x-y|^{\mu}}dy\right)dx\Bigg]\notag\\
=&\|u_{\varepsilon}\|_{\infty}^{\frac{2}{N-2}}\Bigg[\displaystyle{\int_{\R^N}}U_{0,1}^{2^{*}_\mu-1}(x)\left(\sum_{j=1}^{N}a_j\frac{\partial U_{0,1}}{\partial x_j}(x)\right)\notag\\
&\quad\,\,\quad\,\,\quad\,\quad\left(\displaystyle{\int_{\R^N}}\frac{U_{0,1}^{2^{*}_\mu-1}(y)}{|x-y|^{\mu}}\left(\frac{N-2}{2}\frac{1-|y|^{2}}{(1+|y|^{2})^{\frac{N}{2}}}\right)dy\right)dx+o(1)\Bigg]\notag\\
=&\|u_{\varepsilon}\|_{\infty}^{\frac{2}{N-2}}\left(\displaystyle{\int_{\R^N}}U_{0,1}^{2^{*}-2}(x)\left(\sum_{j=1}^{N}a_j\frac{(2-N)x_{j}}{(1+|x|^{2})^{\frac{N}{2}}}\right)\left(\frac{(N-2)\mu}{2\left(2N-\mu\right)}\frac{1-|x|^{2}}{(1+|x|^{2})^{\frac{N}{2}}}\right)dx+o(1)\right)\notag\\
=&0+o\left(\|u_{\varepsilon}\|_{\infty}^{\frac{2}{N-2}}\right)=o\left(\|u_{\varepsilon}\|_{\infty}^{\frac{2}{N-2}}\right)\label{5-23}.
\end{align}
Hence, from \eqref{5.19}-\eqref{5-23} and some estimates of ${D}_{3,\varepsilon}$ in \eqref{4-23}, we have
\begin{align}
\widehat{D}_{3,\varepsilon}=&\varepsilon\left(\sum_{j=1}^{N}a_j^{2}\right)\left(\frac{1}{N}\displaystyle{\int_{\R^N}}\left|\nabla U_{0,1}\right|^{2}dx+o(1)\right)\notag\\
&+\left(2^{*}-1\right)\|u_{\varepsilon}\|_{\infty}^{2^{*}-2}\left(\sum_{j=1}^{N}a_j^{2}\right)\left(\frac{1}{N}\displaystyle{\int_{\R^N}}U_{0,1}^{2^{*}-2}(x)\left|\nabla U_{0,1}\right|^{2}dx+o(1)\right)\notag\\
&+O\left(\frac{1}{\|u_{\varepsilon}\|_{\infty}^{2}}\right)+o\left(\frac{\varepsilon}{\|u_{\varepsilon}\|_{\infty}^{\frac{2}{N-2}}}\right)+o\left(\|u_{\varepsilon}\|_{\infty}^{\frac{2}{N-2}}\right)\notag\\
    &+\left(2^{*}-1\right) d^{2}\left(\frac{N-2}{2}\right)^{2}\left(\displaystyle{\int_{\R^N}}U_{0,1}^{2^{*}-2}(x)\frac{\left(1-|x|^{2}\right)^{2}}{(1+|x|^{2})^{N}}dx+o(1)\right).\label{5-24}
    \end{align}

Therefore, from \eqref{hatD-ep}, \eqref{hatD-ep,1}, \eqref{hatd,2} and \eqref{5-24}, we conclude that
\begin{equation}\label{6-6}
    \begin{split}
     \widehat{D}_{\varepsilon}\geq &\widehat{D}_{2,\varepsilon}+\widehat{D}_{3,\varepsilon}\\
\geq&C_{1}|a|^{2}\|u_\varepsilon\|_{\infty}^{2^{*}-2}+o\left(\|u_{\varepsilon}\|_{\infty}^{\frac{2}{N-2}}\right)+C_{2}^{2}\,d^{2}\\
\geq &\frac{C_{1}}{2}|a|^{2}\|u_\varepsilon\|_{\infty}^{2^{*}-2}+\frac{C_{2}^{2}d^{2}}{2}\geq \delta>0,
    \end{split}
\end{equation}
for some $C_{1},\,C_{2},\delta>0$, and the last inequality holds due to the assumption that $\sum_{j=1}^{N}a_{j}^{2}+d^{2}=1$, which implies that $\sum_{j=1}^{N}a_{j}^{2}$ and $d^{2}$ can not vanish at the same time.

In the following, we estimate each term in the numerator $\widehat{N}_{\varepsilon}$. Combining the estimates of $\widehat{D}_{2,\varepsilon}^{(2)}$ in \eqref{5-21} and the estimates of ${N}_{2,\varepsilon}$ in \eqref{4-25}, we derive that
\begin{equation}\label{5-27}
\widehat{N}_{2,\varepsilon}=O\left(\frac{1}{\|u_{\varepsilon}\|_{\infty}^{2^{*}_{\mu}}}\right).
\end{equation}
Moreover, from \eqref{eq 2.4}, we obtain that the term $\widehat{N}_{4,\varepsilon}$ given in \eqref{hatN-ep,4} can be estimated as follows,
\begin{equation}\label{5-28}
\begin{split}
\widehat{N}_{4,\varepsilon}=&\!-\varepsilon d
\sum_{j=1}^{N}a_{j}\displaystyle{\int_{\Omega}}\frac{\partial\left(\phi^{2}(x)\right)}{\partial x_{j}} u_{\varepsilon}^{2}dx+2\varepsilon d^{2}\!\left(\!\displaystyle{\int_{\Omega}}\left(\nabla\phi \cdot\left(x-x_{\varepsilon}\right)\right)\phi\,u_{\varepsilon}^{2}dx-\!\displaystyle{\int_{\Omega}}\phi^{2}u_{\varepsilon}^{2}dx\right)\\
\leq&O\left(\frac{\varepsilon}{\|u_{\varepsilon}\|_{\infty}^{2}}\right)+\frac{2\varepsilon d^{2}}{\|u_{\varepsilon}\|_{\infty}^{2^{*}-2}}\displaystyle{\int_{\Omega_{\varepsilon}}}\phi^{2}(\tau_{\varepsilon}^{-1}x+x_{\varepsilon})\tilde{u}_{\varepsilon}^{2}(x)dx\\
=&O\left(\frac{\varepsilon}{\|u_{\varepsilon}\|_{\infty}^{2}}\right)+\frac{2\varepsilon d^{2}}{\|u_{\varepsilon}\|_{\infty}^{2^{*}-2}}\left(\displaystyle{\int_{\R^N}}U_{0,1}^{2}(x)dx+o(1)\right)\\
=&O\left(\frac{\varepsilon}{\|u_{\varepsilon}\|_{\infty}^{2}}\right)+O\left(\frac{\varepsilon}{\|u_{\varepsilon}\|_{\infty}^{2^{*}-2}}\right)=O\left(\frac{1}{\|u_{\varepsilon}\|_{\infty}^{2}}\right).
\end{split}
\end{equation}
Similarly, from the definition of the cut-off function $\phi$ in \eqref{eq 4.1} and \eqref{eq 2.3}, we have
\begin{equation}\label{5-6}
\widehat{N}_{3,\varepsilon}^{(1)}=O\left(\frac{1}{\|u_{\varepsilon}\|_{\infty}^{2}}\right),\quad
\widehat{N}_{3,\varepsilon}^{(2)}=O\left(\frac{1}{\|u_{\varepsilon}\|_{\infty}^{2}}\right).
\end{equation}
Moreover, due to the symmetry property of the double integrals and the fact that
\begin{equation}\label{5-7}
    (y-x_{\varepsilon})\cdot\nabla\left(u_\varepsilon^{2^{*}_\mu}(y)\right)
    =\sum_{l=1}^{N}\frac{\partial}{\partial y_{l}}\left((y_l-x_{\varepsilon,l})u_\varepsilon^{2^{*}_\mu}(y)\right)-N u_\varepsilon^{2^{*}_\mu}(y),
\end{equation}
we deduce that
\begin{equation}\label{eq 6.7}
    \begin{split}
&\widehat{N}_{3,\varepsilon}^{(4)}+\widehat{N}_{3,\varepsilon}^{(5)}\\
=&\frac{d}{2^{*}_{\mu}}\!\displaystyle{\int_{\Omega}}\!\displaystyle{\int_{\Omega}}\frac{\left(\phi(x)-\phi(y)\right)^{2}\!\left(\sum_{j=1}^{N}a_j\frac{\partial }{\partial x_j} u_\varepsilon^{2^{*}_\mu}(x)\right)\!\left((y-x_{\varepsilon})\cdot\nabla(u_\varepsilon^{2^{*}_\mu}(y))+\frac{N-2}{2}u_\varepsilon^{2^{*}_\mu}(y)\right)}{|x-y|^{\mu}}dxdy\\
=&\underbrace{\frac{d}{2^{*}_{\mu}}\displaystyle{\int_{\Omega}}\displaystyle{\int_{\Omega}}\frac{\left(\phi(x)-\phi(y)\right)^{2}\left(\sum_{j=1}^{N}a_j\frac{\partial }{\partial x_j} u_\varepsilon^{2^{*}_\mu}(x)\right)\left(\sum_{l=1}^{N}\frac{\partial}{\partial y_l}\left((y_l-x_{\varepsilon,l})u_\varepsilon^{2^{*}_\mu}(y)\right)\right)}{|x-y|^{\mu}}dxdy}_{:=P_{1}}\\
        &-\underbrace{\frac{(N+2)d}{2\cdot2^{*}_{\mu}}\displaystyle{\int_{\Omega}}\displaystyle{\int_{\Omega}}\frac{\left(\phi(x)-\phi(y)\right)^{2}\left(\sum_{j=1}^{N}a_j\frac{\partial }{\partial x_j} u_\varepsilon^{2^{*}_\mu}(x)\right)u_\varepsilon^{2^{*}_\mu}(y)}{|x-y|^{\mu}}dxdy}_{:=P_2}.
    \end{split}
\end{equation}
Then, a direct computation yields that
\begin{align*}
   & {P_1}\\
    =&\frac{d}{2^{*}_{\mu}}\sum_{j,l=1}^{N}a_j\displaystyle{\int_{\Omega}}\displaystyle{\int_{\Omega}}\left(\frac{\partial^{2}}{\partial y_l\partial x_j} \left(\frac{\left(\phi(x)-\phi(y)\right)^{2}}{|x-y|^{\mu}}\right)(y_l-x_{\varepsilon,l})\right)u_\varepsilon^{2^{*}_\mu}(x)u_\varepsilon^{2^{*}_\mu}(y)dxdy\\
    =&\underbrace{\frac{d}{2^{*}_{\mu}}\sum_{j,l=1}^{N}a_j\displaystyle{\int_{\Omega}}\displaystyle{\int_{\Omega}}\left(\left(\frac{\partial^{2}}{\partial y_l\partial x_j} \left(\phi(x)-\phi(y)\right)^{2}\right)(y_l-x_{\varepsilon,l})\right)\frac{u_\varepsilon^{2^{*}_\mu}(y)u_\varepsilon^{2^{*}_\mu}(x)}{|x-y|^{\mu}}dxdy}_{:=P_{1,1}}\\
    &+\underbrace{\frac{d}{2^{*}_{\mu}}\sum_{j,l=1}^{N}a_j\displaystyle{\int_{\Omega}}\displaystyle{\int_{\Omega}}\left(\left(\frac{\partial^{2}}{\partial y_l\partial x_j}\left(\frac{1}{|x-y|^{\mu}}\right)\right)(y_l-x_{\varepsilon,l}) \right)\left(\phi(x)-\phi(y)\right)^{2}{u_\varepsilon^{2^{*}_\mu}(y)u_\varepsilon^{2^{*}_\mu}(x)}dxdy}_{:=P_{1,2}}\\
    &+\underbrace{\frac{d}{2^{*}_{\mu}}\sum_{j,l=1}^{N}a_j\displaystyle{\int_{\Omega}}\displaystyle{\int_{\Omega}}\left(\frac{\partial}{\partial y_l}\left(\frac{1}{|x-y|^{\mu}}\right)\frac{\partial}{\partial x_j}\Big(\phi(x)-\phi(y)\Big)^{2}(y_l-x_{\varepsilon,l}) \right){u_\varepsilon^{2^{*}_\mu}(y)u_\varepsilon^{2^{*}_\mu}(x)}dxdy}_{:=P_{1,3}}\\
    &+\underbrace{\frac{d}{2^{*}_{\mu}}\sum_{j,l=1}^{N}a_j\displaystyle{\int_{\Omega}}\displaystyle{\int_{\Omega}}\left(\frac{\partial}{\partial x_j}\left(\frac{1}{|x-y|^{\mu}}\right)\frac{\partial}{\partial y_l}\Big(\phi(x)-\phi(y)\Big)^{2}(y_l-x_{\varepsilon,l}) \right){u_\varepsilon^{2^{*}_\mu}(y)u_\varepsilon^{2^{*}_\mu}(x)}dxdy}_{:=P_{1,4}}.
\end{align*}
Since for any $y\in\Omega$, $|y_l-x_{\varepsilon,l}|\leq |y-x_\varepsilon|\leq diam\left(\Omega\right)\leq C$, for $l=1,\cdots,N$. Similar arguments as  the estimates of $\widehat{K_{1}}$ and $\widehat{K_{2}}$ in $N_{3,\varepsilon}^{(2)}$ in \eqref{eqq4.35} can be used to yield that
$$P_{1,1}=O\left(\frac{1}{\|u_\varepsilon\|^{2\cdot2^{*}_\mu}_{\infty}}\right),\quad P_{1,2}=O\left(\frac{1}{\|u_\varepsilon\|^{2^{*}_\mu}_{\infty}}\right).$$
As for $P_{1,3}$, using the similar arguments as the estimates of $\widehat{K_{3}}$ and $\widehat{K_{4}}$ in $N_{3,\varepsilon}^{(2)}$ in \eqref{eqq4.35}, we have
\begin{align*}
    &P_{1,3}\\
    =&\frac{2d\mu}{2^{*}_\mu}\sum_{j,l=1}^{N}a_{j}\displaystyle{\int_{B_{\rho}(x_\varepsilon)}}\displaystyle{\int_{B_{\rho}(x_\varepsilon)}}\left(\frac{(x_j-y_j)}{|x-y|^{\mu+2}}\frac{\partial \phi(y)}{\partial y_{l}}\,\left(y_l-x_{\varepsilon,l}\right)\right)\left(\phi(x)-\phi(y)\right)u_{\varepsilon}^{2^{*}_\mu}(x)u_{\varepsilon}^{2^{*}_\mu}(y)dxdy\\
+&\frac{d}{2^{*}_\mu}\sum_{j,l=1}^{i-1}a_{j}\!\displaystyle{\int}\displaystyle{\int_{\left(\Omega\times\Omega\right)\backslash\left(B_{\rho}(x_\varepsilon)\times B_{\rho}(x_\varepsilon)\right)}}\!\left(\frac{\partial}{\partial x_{j}}\left(\!\frac{1}{|x-y|^{\mu}}\!\right)\!\right)\!\left(\!\frac{\partial}{\partial y_{l}}\Big(\phi(x)-\phi(y)\Big)^{2}\right)\!\left(y_l-x_{\varepsilon,l}\!\right)u_{\varepsilon}^{2^{*}_\mu}(x)u_{\varepsilon}^{2^{*}_\mu}(y)\\
    =&0+O\left(\frac{1}{\|u_{\varepsilon}\|_{\infty}^{2^{*}_\mu}}\right)=O\left(\frac{1}{\|u_{\varepsilon}\|_{\infty}^{2^{*}_\mu}}\right).
\end{align*}
Similarly, we can deduce that $P_{1,4}=O\left({1}/{\|u_{\varepsilon}\|_{\infty}^{2^{*}_\mu}}\right)$.
Hence, we derive that
\begin{equation}\label{5-32}
    P_{1}=O\left(\frac{1}{\|u_{\varepsilon}\|_{\infty}^{2^{*}_\mu}}\right).
\end{equation}
Also, a direct calculation yields that
\begin{align*}
    P_{2}
    =&\frac{(N+2)d}{2\cdot2^{*}_{\mu}}\sum_{j=1}^{N}a_j\displaystyle{\int_{\Omega}}\displaystyle{\int_{\Omega}}\,\frac{\partial }{\partial x_j} \left(\frac{\left(\phi(x)-\phi(y)\right)^{2}}{|x-y|^{\mu}}\right)u_\varepsilon^{2^{*}_\mu}(x)u_\varepsilon^{2^{*}_\mu}(y)dxdy\\
    =&\underbrace{\frac{(N+2)d}{2^{*}_{\mu}}\sum_{j=1}^{N}a_j\displaystyle{\int_{\Omega}}\displaystyle{\int_{\Omega}}\, \frac{\frac{\partial \phi(x)}{\partial x_j}\left(\phi(x)-\phi(y)\right)u_\varepsilon^{2^{*}_\mu}(x)u_\varepsilon^{2^{*}_\mu}(y)}{|x-y|^{\mu}}dxdy}_{:=P_{2,1}}\\
    &+\underbrace{\frac{(N+2)d}{2\cdot2^{*}_{\mu}}\sum_{j=1}^{N}a_j\displaystyle{\int_{\Omega}}\displaystyle{\int_{\Omega}}\,\frac{\partial }{\partial x_j} \left(\frac{1}{|x-y|^{\mu}}\right)\left(\phi(x)-\phi(y)\right)^{2}u_\varepsilon^{2^{*}_\mu}(x)u_\varepsilon^{2^{*}_\mu}(y)dxdy}_{:=P_{2,2}}.
\end{align*}
Using HLS, we have
\begin{align*}
P_{2,1}\leq&C\left(\displaystyle{\displaystyle{\int_{\Omega}}}\left|\left(\frac{\partial \phi(x)}{\partial x_j} \right)u_\varepsilon^{2^{*}_\mu}(x)\right|^{\frac{2N}{2N-\mu}}dx\right)^{\frac{2N-\mu}{2N}}\left(\displaystyle{\int_{\Omega}}\left|(\phi(x)-\phi(y))u_\varepsilon^{2^{*}_\mu}(y)\right|^{\frac{2N}{2N-\mu}}dy\right)^{\frac{2N-\mu}{2N}}\\
    =&O\left(\frac{1}{\|u_\varepsilon\|^{2^{*}_\mu}_{\infty}}\right).
\end{align*}
As for $P_{2,2}$, due to the definition of the cut-off function $\phi$ in \eqref{eq 4.1}, we find that $\phi(x)-\phi(y)=0$ provided $x,y\in B_{\rho}(x_{\varepsilon})$, which implies that
\begin{align*}
    P_{2,2}=&\frac{(N+2)d}{2\cdot2^{*}_{\mu}}\displaystyle{\int_{\Omega\backslash B_{\rho}(x_\varepsilon) }\int_{B_{\rho}(x_\varepsilon)}}\left(\sum_{j=1}^{N}a_j\frac{\partial }{\partial x_j} \left(\frac{1}{|x-y|^{\mu}}\right)\right)\left(\phi(x)-\phi(y)\right)^{2}u_\varepsilon^{2^{*}_\mu}(x)u_\varepsilon^{2^{*}_\mu}(y)dxdy\\
    &+\frac{(N+2)d}{2\cdot2^{*}_{\mu}}\displaystyle{\int_{\Omega\backslash B_{\rho}(x_\varepsilon) }\int_{B_{\rho}(x_\varepsilon)}}\left(\sum_{j=1}^{N}a_j\frac{\partial }{\partial x_j} \left(\frac{1}{|x-y|^{\mu}}\right)\right)\left(\phi(x)-\phi(y)\right)^{2}u_\varepsilon^{2^{*}_\mu}(x)u_\varepsilon^{2^{*}_\mu}(y)dydx\\
    &+\frac{(N+2)d}{2\cdot2^{*}_{\mu}}\displaystyle{\int_{\Omega\backslash B_{\rho}(x_\varepsilon) }\int_{\Omega\backslash B_{\rho}(x_\varepsilon)}}\left(\sum_{j=1}^{N}a_j\frac{\partial }{\partial x_j} \left(\frac{1}{|x-y|^{\mu}}\right)\right)\left(\phi(x)-\phi(y)\right)^{2}u_\varepsilon^{2^{*}_\mu}(x)u_\varepsilon^{2^{*}_\mu}(y)dxdy\\
    =&0+O\left(\frac{1}{\|u_\varepsilon\|^{2\cdot2^{*}_\mu}_{\infty}}\right)=O\left(\frac{1}{\|u_\varepsilon\|^{2\cdot2^{*}_\mu}_{\infty}}\right),
\end{align*}
where we utilized the fact that
\begin{equation*}
    \begin{split}
   &\displaystyle{\int_{\Omega\backslash B_{\rho}(x_\varepsilon) }\int_{B_{\rho}(x_\varepsilon)}}\left(\sum_{j=1}^{N}a_j\frac{\partial }{\partial x_j} \left(\frac{1}{|x-y|^{\mu}}\right)\right)\left(\phi(x)-\phi(y)\right)^{2}u_\varepsilon^{2^{*}_\mu}(x)u_\varepsilon^{2^{*}_\mu}(y)dydx\\
   =&\displaystyle{\int_{\Omega\backslash B_{\rho}(x_\varepsilon) }\int_{B_{\rho}(x_\varepsilon)}}\left(\sum_{j=1}^{N}a_j\frac{\partial }{\partial y_j} \left(\frac{1}{|x-y|^{\mu}}\right)\right)\left(\phi(y)-\phi(x)\right)^{2}u_\varepsilon^{2^{*}_\mu}(y)u_\varepsilon^{2^{*}_\mu}(x)dxdy\\
=&\displaystyle{\int_{\Omega\backslash B_{\rho}(x_\varepsilon) }\int_{B_{\rho}(x_\varepsilon)}}\left(-\sum_{j=1}^{N}a_j\frac{\partial }{\partial x_j} \left(\frac{1}{|x-y|^{\mu}}\right)\right)\left(\phi(y)-\phi(x)\right)^{2}u_\varepsilon^{2^{*}_\mu}(y)u_\varepsilon^{2^{*}_\mu}(x)dxdy
    \end{split}
\end{equation*}
and similar arguments as  the estimates of $\widehat{H_{2,4}}$ in $\widehat{H_{2}}$ given in \eqref{hatH2}. Hence, we have
\begin{equation}\label{5-33}
    P_{2}=O\left(\frac{1}{\|u_\varepsilon\|^{2^{*}_\mu}_{\infty}}\right).
\end{equation}
Therefore, from \eqref{eq 6.7}, \eqref{5-32} and \eqref{5-33}, we conclude that
\begin{equation}\label{5-10}
\widehat{N}_{3,\varepsilon}^{(4)}+\widehat{N}_{3,\varepsilon}^{(5)}=O\left(\frac{1}{\|u_\varepsilon\|^{2^{*}_\mu}_{\infty}}\right).
\end{equation}

Now, we estimate the remaining term $\widehat{N}_{3,\varepsilon}^{(3)}$ in $\widehat{N}_{3,\varepsilon}$. Due to the symmetry property and \eqref{5-7}, we obtain that
\begin{equation*}\label{eq 6.9}
    \begin{split}
&\widehat{N}_{3,\varepsilon}^{(3)}\\
=&\underbrace{\frac{d^{2}}{2\cdot2^{*}_\mu }\displaystyle{\int_{\Omega}}\displaystyle{\int_{\Omega}}\frac{\left(\phi(x)-\phi(y)\right)^{2}\left(\sum_{j=1}^{N}\frac{\partial}{\partial x_{j}}\left((x_j-x_{\varepsilon,j})u_\varepsilon^{2^{*}_\mu}(x)\right)\right)\left(\sum_{l=1}^{N}\frac{\partial}{\partial y_{l}}\left((y_l-x_{\varepsilon,l})u_\varepsilon^{2^{*}_\mu}(y)\right)\right)}{|x-y|^{\mu}}dydx}_{:=A_{1}}\\
&-\underbrace{\frac{d^{2}(N+2)}{2\cdot2^{*}_\mu }\displaystyle{\int_{\Omega}}\displaystyle{\int_{\Omega}}\frac{\left(\phi(x)-\phi(y)\right)^{2}\left(\sum_{j=1}^{N}\frac{\partial}{\partial x_{j}}\left((x_j-x_{\varepsilon,j})u_\varepsilon^{2^{*}_\mu}(x)\right)\right)u_\varepsilon^{2^{*}_\mu}(y)}{|x-y|^{\mu}}dxdy}_{:=A_{2}}\\
&+\underbrace{\frac{d^{2}}{2\cdot2^{*}_\mu }\left(\frac{N+2}{2}\right)^{2}\displaystyle{\int_{\Omega}}\displaystyle{\int_{\Omega}}\frac{\left(\phi(x)-\phi(y)\right)^{2}u_\varepsilon^{2^{*}_\mu}(x)u_\varepsilon^{2^{*}_\mu}(y)}{|x-y|^{\mu}}dxdy}_{:=A_{3}}.
    \end{split}
\end{equation*}
Then, using similar discussions in the proof of $P_{1}$ and $P_{2}$, we obtain that
    \begin{align*}
        A_{1}=O\left(\frac{1}{\|u_\varepsilon\|^{2^{*}_\mu}_{\infty}}\right),\quad A_{2}=O\left(\frac{1}{\|u_\varepsilon\|^{2^{*}_\mu}_{\infty}}\right).
    \end{align*}
As for $A_{3}$, since $\phi(x)=\phi(y)=1$ for $x,y\in B_{\rho}(x_\varepsilon)$,  by HLS, it holds
\begin{align*}
    A_{3}
    \leq&C\displaystyle{\int_{B_{\rho}(x_\varepsilon)}}\displaystyle{\int_{\Omega\backslash B_{\rho}(x_\varepsilon)}}\frac{U_{x_\varepsilon,\tau_{\varepsilon}}^{2^{*}_\mu}(x)U_{x_\varepsilon,\tau_{\varepsilon}}^{2^{*}_\mu}(y)}{|x-y|^{\mu}}dydx\\
    &+C\displaystyle{\int_{\Omega\backslash B_{\rho}(x_\varepsilon)}}\displaystyle{\int_{\Omega\backslash B_{\rho}(x_\varepsilon)}}\frac{U_{x_\varepsilon,\tau_{\varepsilon}}^{2^{*}_\mu}(x)U_{x_\varepsilon,\tau_{\varepsilon}}^{2^{*}_\mu}(y)}{|x-y|^{\mu}}dydx\\
    =&O\left(\frac{1}{\|u_\varepsilon\|^{2^{*}_\mu}_{\infty}}\right).
\end{align*}
Thus, from above estimates of $A_{1}-A_{3}$, we obtain that
\begin{equation}\label{5-35}
    \widehat{N}_{3,\varepsilon}^{(3)}=O\left(\frac{1}{\|u_\varepsilon\|^{2^{*}_\mu}_{\infty}}\right).
\end{equation}
Then, from \eqref{hatN-ep,3}, \eqref{5-6}, \eqref{5-10}, \eqref{5-35} and the previous estimates of $N_{3,\varepsilon}^{(1)}$ and $N_{3,\varepsilon}^{(2)}$ in \eqref{4-34} and \eqref{N_{3,2}}, we obtain that
\begin{equation}\label{5-36}
  \widehat{N}_{3,\varepsilon}=O\left(\frac{1}{\|u_\varepsilon\|^{2}_{\infty}}\right).
\end{equation}

Therefore, from \eqref{hatN-ep}, \eqref{hatN-ep,1}, \eqref{5-27}, \eqref{5-28} and \eqref{5-36}, we conclude that
\begin{equation}\label{5-13}
    \begin{split}
        \widehat{N}_{\varepsilon}=& \widehat{N}_{1,\varepsilon}+O\left(\frac{1}{\|u_{\varepsilon}\|_{\infty}^{2}}\right)\\
        =&a_{0}^{2}(2-2\cdot2^{*}_\mu)\displaystyle{\int_{\Omega}}\displaystyle{\int_{\Omega}}\frac{u_{\varepsilon}^{2^{*}_\mu}(x)u_{\varepsilon}^{2^{*}_\mu}(y)}{|x-y|^{\mu}}dxdy+O\left(\frac{1}{\|u_{\varepsilon}\|_{\infty}^{2}}\right)\\
        \leq&\, O\left(\frac{1}{\|u_{\varepsilon}\|_{\infty}^{2}}\right).
    \end{split}
\end{equation}
Moreover, from \eqref{eq5-2}, \eqref{6-6} and \eqref{5-13}, we have
\begin{align*}
\underset{\varepsilon\rightarrow0}{
\limsup}\,\lambda_{N+2,\varepsilon}\leq\underset{\varepsilon\rightarrow0}{
\limsup}\left(1+\frac{\widehat{N}_{\varepsilon}}{\widehat{D}_{\varepsilon}}\right)\leq 1+\underset{\varepsilon\rightarrow0}{
    \lim}\,\frac{O\left(\frac{1}{\|u_{\varepsilon}\|_{\infty}^{2}}\right)}{\delta}=1.
\end{align*}
Thus, we obtain the desired result \eqref{eq 6.1}.
\end{proof}
\end{lemma}

In the following, we prove other claims in Theorem 1.3.
\begin{proof}[\textbf{Proof of \eqref{eq 1.24} and \eqref{eq 1.23}}] Since we have obtained \eqref{eq 6.1}, from Lemma \ref{lema 2.8}, we know that for $\varepsilon>0$ sufficiently small,
\begin{align}\label{eq 6.10}
    \tilde{v}_{N+2,\varepsilon}(x)\longrightarrow\sum_{k=1}^{N}
\frac{a_{N+2,k}\,x_k}{\big(1+|x|^{2}\big)^{\frac{N}{2}}}+b_{N+2}\frac{1-|x|^{2}}{\big(1+|x|^{2}\big)^{\frac{N}{2}}},\quad\text{in $C^{1}_{loc}(\R^N)$,}
\end{align}
where $\left(a_{N+2,1},\cdots,a_{N+2,N},b_{N+2}\right)\neq\vec{0}$. To obtain the desired result \eqref{eq 1.24}, we only need to prove $a_{N+2,k}=0$, for $k=1,\cdots,N$. In fact, for fixed $\varepsilon>0$, we know that $v_{N+2,\varepsilon}$ is orthogonal to $v_{i,\varepsilon}$ in the sense of \eqref{eqq1-11}, then
\begin{equation}\label{eq 6.11}
    \begin{split}
0
=&(2^{*}_\mu-1)\lambda_{N+2,\varepsilon}\displaystyle{\int_{\Omega_{\varepsilon}}}\tilde{u}^{2^{*}_\mu-2}_{\varepsilon}\tilde{v}_{i,\varepsilon}\tilde{v}_{N+2,\varepsilon}\left(\displaystyle{\int_{\Omega_{\varepsilon}}}\frac{\tilde{u}_{\varepsilon}^{2^{*}_\mu}(y)}{|x-y|^{\mu}}dy\right)dx+\frac{\varepsilon\lambda_{N+2,\varepsilon}}{\|u_\varepsilon\|_{\infty}^{2^{*}-2}} \displaystyle{\int_{\Omega_{\varepsilon}}}\tilde{v}_{i,\varepsilon}\tilde{v}_{N+2,\varepsilon}dx\\
&+2^{*}_\mu \lambda_{N+2,\varepsilon}\displaystyle{\int_{\Omega_{\varepsilon}}}\tilde{u}^{2^{*}_\mu-1}_{\varepsilon}\tilde{v}_{i,\varepsilon}\left(\displaystyle{\int_{\Omega_{\varepsilon}}}\frac{\tilde{u}^{2^{*}_\mu-1}_{\varepsilon}(y)\tilde{v}_{N+2,\varepsilon}(y)}{|x-y|^{\mu}}dy\right)dx,\quad\text{for $i=2,\cdots, N+1$}.
    \end{split}
\end{equation}
From \eqref{eq 1.9}, \eqref{eq 1.21} and \eqref{eq 6.10}, we derive that as $\varepsilon\rightarrow0$,
\begin{align*}
&\text{RHS of \eqref{eq 6.11}}
   \longrightarrow \\
   &\!\underbrace{{\!(2^{*}_\mu-1)}\displaystyle{\!\int_{\R^N}}\!U_{0,1}^{2^{*}_\mu-2}
   \!\left(\!\sum_{k=1}^{N}\!\frac{a_{i,k}\,x_{k}}{\left(1+|x|^{2}\!\right)^{\frac{N}{2}}}\right)\!\left(\!\sum_{j=1}^{N}
\!\frac{a_{N+2,j}\,x_j}{\big(1+|x|^{2}\big)^{\frac{N}{2}}}+\!b_{N+2}\frac{1-|x|^{2}}{\big(1+|x|^{2}\big)
^{\frac{N}{2}}}\!\right)\!\left(\!\displaystyle{\int_{\R^N}}\!\frac{U_{0,1}^{2^{*}_\mu}(y)}{|x-y|^{\mu}}dy\!\right)\!dx}_{:=G_{1}}\\
+&\!\underbrace{{2^{*}_\mu }\displaystyle{\int_{\R^N}}U_{0,1}^{2^{*}_\mu-1}\!\left(\sum_{k=1}^{N}\frac{a_{i,k}\,x_{k}}{\left(1+|x|^{2}\right)
^{\frac{N}{2}}}\right)\!\left(\displaystyle{\int_{\R^N}}\frac{U_{0,1}^{2^{*}_\mu-1}(y)}{|x-y|^{\mu}}\!\left(\sum_{j=1}^{N}
\frac{a_{N+2,j}\,y_j}{\big(1+|y|^{2}\big)^{\frac{N}{2}}}+\!b_{N+2}\frac{1-|y|^{2}}{\big(1+|y|^{2}\big)^{\frac{N}{2}}}\right)dy\!\right)\!dx}_{:=G_{2}},
\end{align*}
where we have used the fact \eqref{eq 2.4} to derive that
\begin{align*}
&\underset{\varepsilon\rightarrow0}{\lim}\frac{\varepsilon}{\|u_\varepsilon\|_{\infty}^{2^{*}-2}} \displaystyle{\int_{\Omega_{\varepsilon}}}\left(\sum_{k=1}^{N}\frac{a_{i,k}\,x_{k}}{\left(1+|x|^{2}\right)^{\frac{N}{2}}}\right)\left(\sum_{j=1}^{N}
\frac{a_{N+2,j}\,x_j}{\big(1+|x|^{2}\big)^{\frac{N}{2}}}+b_{N+2}\frac{1-|x|^{2}}{\big(1+|x|^{2}\big)^{\frac{N}{2}}}\right)dx\\
=&\underset{\varepsilon\rightarrow0}{\lim}\frac{\varepsilon}{\|u_\varepsilon\|_{\infty}^{2^{*}-2}} \displaystyle{\int_{\Omega_{\varepsilon}}}\left(\sum_{k=1}^{N}\frac{a_{i,k}\,a_{N+2,k}\,x_{k}^{2}}{\left(1+|x|^{2}\right)^{N}}\right)dx=0.
\end{align*}
From \eqref{eq 1.7}, we have
\begin{align*}
    G_{1}=&{(2^{*}_\mu-1)}\displaystyle{\int_{\R^N}}U_{0,1}^{2^{*}-2}\left(\sum_{k=1}^{N}\frac{a_{i,k}\,x_{k}}{\left(1+|x|^{2}\right)^{\frac{N}{2}}}\right)\left(\sum_{j=1}^{N}
\frac{a_{N+2,j}\,x_j}{\big(1+|x|^{2}\big)^{\frac{N}{2}}}+b_{N+2}\frac{1-|x|^{2}}{\big(1+|x|^{2}\big)^{\frac{N}{2}}}\right)dx\\
=&{(2^{*}_\mu-1)}\displaystyle{\int_{\R^N}}U_{0,1}^{2^{*}-2}\left(\sum_{k=1}^{N}\frac{a_{i,k}\,a_{N+2,k}\,x_{k}^{2}}{\left(1+|x|^{2}\right)^{N}}\right)dx.
\end{align*}
Moreover, using similar arguments as  the proof of \eqref{a-1} and \eqref{a-2}, we derive
\begin{align*}
      G_{2}=&{2^{*}_\mu }\displaystyle{\int_{\R^N}}U_{0,1}^{2^{*}_\mu-1}\left(\sum_{k=1}^{N}\frac{a_{i,k}x_{k}}{\left(1+|x|^{2}\right)
      ^{\frac{N}{2}}}\right)\left(\displaystyle{\int_{\R^N}}\frac{U_{0,1}^{2^{*}_\mu-1}(y)}{|x-y|^{\mu}}\left(\sum_{j=1}^{N}
\frac{a_{N+2,j}\,y_j}{\big(1+|y|^{2}\big)^{\frac{N}{2}}}\right)dy\right)dx\\
&+{2^{*}_\mu }\displaystyle{\int_{\R^N}}U_{0,1}^{2^{*}_\mu-1}\left(\sum_{k=1}^{N}\frac{a_{i,k}\,x_{k}}{\left(1+|x|^{2}\right)^{\frac{N}{2}}}\right)
\left(\displaystyle{\int_{\R^N}}\frac{U_{0,1}^{2^{*}_\mu-1}(y)}{|x-y|^{\mu}}\left(b_{N+2}\frac{1-|y|^{2}}{\big(1+|y|^{2}\big)^{\frac{N}{2}}}\right)dy\right)dx\\
=&{\left(2^{*}-2^{*}_\mu\right)}\displaystyle{\int_{\R^N}}U_{0,1}^{2^{*}-2}\left(\sum_{k=1}^{N}\frac{a_{i,k}\,x_{k}}{\left(1+|x|^{2}\right)^{\frac{N}{2}}}\right)\left(\sum_{j=1}^{N}a_{N+2,j}\frac{x_{j}}{\left(1+|x|^{2}\right)^{\frac{N}{2}}}\right)dx\\
&+{\left(2^{*}-2^{*}_\mu\right)b_{N+2} }\displaystyle{\int_{\R^N}}U_{0,1}^{2^{*}-2}\left(\sum_{k=1}^{N}\frac{a_{i,k}\,x_{k}}{\left(1+|x|^{2}\right)^{\frac{N}{2}}}\right)\left(\frac{1-|x|^{2}}{\big(1+|x|^{2}\big)^{\frac{N}{2}}}\right)dx\\
=&{\left(2^{*}-2^{*}_\mu\right)}\displaystyle{\int_{\R^N}}U_{0,1}^{2^{*}-2}\left(\sum_{k=1}^{N}\frac{a_{i,k}\,a_{N+2,k}\,x_{k}^{2}}{\left(1+|x|^{2}\right)^{N}}\right)dx.
\end{align*}

Therefore, we deduce that
\begin{equation*}
    \begin{split}
      0=&{\left(2^{*}-1\right)}\displaystyle{\int_{\R^N}}U_{0,1}^{2^{*}-2}\left(\sum_{k=1}^{N}\frac{a_{i,k}\,a_{N+2,k}\,x_{k}^{2}}{\left(1+|x|^{2}\right)^{N}}\right)dx\\
      =&\langle \vec{a}_{i},\vec{a}_{N+2}\rangle\left(\frac{\left(2^{*}-1\right)}{N}\displaystyle{\int_{\R^N}}U_{0,1}^{2^{*}-2}\,\frac{|x|^{2}}{\left(1+|x|^{2}\right)^{N}}\,dx\right),
    \end{split}
\end{equation*}
which implies that $\langle \vec{a}_{i},\vec{a}_{N+2}\rangle=0$, for $i=2,\cdots,N+1$. Moreover, since the vectors $\{\vec{a}_{i}\}_{i=2}^{N+1}$ are mutually orthogonal in $\R^N$, we have $\vec{a}_{N+2}=\vec{0}$. Hence, we conclude that \eqref{eq 1.24} holds. Furthermore, since $b_{N+2}\neq 0$, utilizing Lemma \ref{lema 2.11} with $i=N+2$, we can obtain \eqref{eq 1.23}. In addition, we claim that $\lambda_{N+2,\varepsilon}$ is simple via a similar discussion as in Section 3. This ends the proof of Theorem 1.3.
\end{proof}
\vskip 0.2cm
\begin{proof}[\textbf{Proof of Theorem \ref{cor1}.}] Suppose that $x_0\in\Omega$ is a non-degenerate critical point of Robin function $R(x)$, then all the eigenvalues $\gamma_{i}$ of the Hessian matrix $D^{2}R(x_0)$ are nonzero, for $i=1,\cdots,N$. On the other hand, from \eqref{eq 1.20}, we find that for $\varepsilon>0$ small, any $\gamma_{i}<0$ implies that $\lambda_{i+1,\varepsilon}<1$. Then, the Morse index of $u_\varepsilon$ is bigger or equal to $m(x_0)+1$ since $\lambda_{1,\varepsilon}<1$, where $m(x_0)$ denotes as the Morse index of $D^{2}R(x)$ at $x_0$. However, from \eqref{eq 1.23}, we know that for $\varepsilon>0$ small, $\lambda_{N+2,\varepsilon}$ is bigger than one. Thus, the Morse index of $u_\varepsilon$ is precisely equal to $m(x_0)+1$.
\end{proof}

\noindent
\section*{Appendix}
\appendix
\renewcommand{\theequation}{A.\arabic{equation}}

\setcounter{equation}{0}

\section{ Some preliminaries}
\renewcommand{\theequation}{A.\arabic{equation}}

\setcounter{equation}{0}

\vskip 0.2cm

\renewcommand{\theequation}{A.\arabic{equation}}

\setcounter{equation}{0}

In this Appendix, we outline some well-known results. First of all, to understand the critical growth of problem \eqref{eq 1.1}, we recall the Hardy-Littlewood-Sobolev inequality (HLS for short).

\smallskip
\begin{lemma} \cite{L2} \label{lema 2.1} Suppose $\mu\in(0,N)$ and $\theta,\,r>1$ with $\frac{1}{\theta}+\frac{1}{r}+\frac{\mu}{N}=2$. Let $f\in L^{\theta}(\R^N)$ and $g\in L^{r}(\R^N)$, there exists a sharp constant $C(\theta,r,\mu,N)$, independent of $f$ and $g$, such that
\begin{align}\label{eq1.11}
\displaystyle{\int_{\R^N}}\displaystyle{\int_{\R^N}}\frac{f(x)g(y)}{|x-y|^{\mu}}dxdy\leq C(\theta,r,\mu,N)\|f\|_{L^{\theta}(\R^N)}\|g\|_{L^{r}(\R^N)}.
\end{align}
If $\theta=r=\frac{2N}{2N-\mu}$, then
$$C(\theta,r,\mu,N)=C_{N,\mu}:=\pi^{\frac{\mu}{2}}\frac{\Gamma\left(\frac{N-\mu}{2}\right)}{\Gamma\left(N-\frac{\mu}{2}\right)}\left(\frac{\Gamma(N)}{\Gamma\left(\frac{N}{2}\right)}\right)^{\frac{N-\mu}{N}}.$$
In this case, the equality in \eqref{eq1.11} holds if and only if $f\equiv (const.)\, g$, where
$$g(x)=A\left(\frac{1}{\tilde{\tau}^{2}+|x-\tilde{a}|^{2}}\right)^{\frac{2N-\mu}{2}},\quad \text{for some $A\in \mathbb{C}$, $0\neq\tilde{\tau}\in\R$ and $\tilde{a}\in\R^N$.}$$
\end{lemma}
\begin{Rem}\label{Rem1.1}
  From \eqref{eq1.11}, the integral
$$\displaystyle{\int_{\R^N}}\displaystyle{\int_{\R^N}}\frac{|u(x)|^{q}|u(y)|^{q}}{|x-y|^{\mu}}dxdy$$
is well-defined in $H^{1}(\R^N)\times H^{1}(\R^N)$ if $\frac{2N-\mu}{N}\leq q\leq\frac{2N-\mu}{N-2}$. Moreover, for any $u\in D^{1,2}(\R^{N})$, we know that
$$\left(\displaystyle{\int_{\R^N}}\displaystyle{\int_{\R^N}}\frac{|u(x)|^{\frac{2N-\mu}{N-2}}|u(y)|^{\frac{2N-\mu}{N-2}}}{|x-y|^{\mu}}dxdy\right)^{\frac{N-2}{2N-\mu}}\leq \left(C_{N,\mu}\right)^{\frac{N-2}{2N-\mu}}\left(\displaystyle{\int_{\R^N}}|u(x)|^{2^{*}}dx\right)^{\frac{2}{2^{*}}}.$$
\end{Rem}
\vskip 0.2cm
In the following, we present several results, which play a key role in our proofs in Section 2.
\begin{lemma}\cite[Lemma 2.5]{G1}\label{lema 2.1}
    Let $u\in H^{1}_{0}(\Omega)$ be a smooth solution of
    $$-\Delta u=a(x)u,$$
with $a(x)\in L^{\frac{N}{2}}(\Omega)$. If there exists $\varepsilon_{0}>0$ such that
$$\displaystyle{\int_{\Omega}}|a(x)|^{\frac{N}{2}}dx<\varepsilon_0,$$
then for any $Q\in \R^N$,
$$\underset{x\in\Omega\cap B_{R}(Q)}{\sup}|u(x)|\leq C\left(\frac{1}{R^N}\displaystyle{\int_{\Omega\cap B_{R}(Q)}}|u|^{p}dx\right)^{\frac{1}{p}},\quad \text{$\forall\,p>1$}.$$
\end{lemma}
\begin{lemma} \cite[Proposition 6.2]{GY}\label{lema 2.2} Assume that $N\geq 3$ and $u_\varepsilon$ is a solution of \eqref{eq 1.1}, then the following identity holds true
\begin{align}\label{eq 2.1}
\varepsilon\displaystyle{\int_{\Omega}}u_\varepsilon^{2}(x)\,dx=\frac{1}{2}\displaystyle{\int_{\partial\Omega}}\left(\frac{\partial u_{\varepsilon}}{\partial \nu }\right)^{2}(x-y)\cdot {\nu}\,dS_{x},\quad\text{$\forall\,y\in\R^N$,}
\end{align}
 where $\nu$ denotes the unit  outward normal to $\partial\Omega$.
\end{lemma}

\begin{lemma}\cite[Lemma 2]{H}\label{lema 2.3}
    Let $u$ be the solution of
    \begin{equation*}
        \begin{cases}
            -\Delta u=f,\quad&\text{in $\Omega$},\\
            u=0,\quad  &\text{on $\partial\Omega$}.
        \end{cases}
    \end{equation*}
Assume that $\omega$ is a neighborhood of $\partial\Omega$, then
\begin{align*}
    \|u\|_{W^{1,q}(\Omega)}+\|\nabla u\|_{C^{0,\alpha}(\omega^{\prime})}\leq C\left(\|f\|_{L^{1}(\Omega)}+\|f\|_{L^{\infty}(\omega)}\right)
\end{align*}
for $q<\frac{N}{N-1}$, $\alpha\in(0,1)$ and $\omega^{\prime}\subset\subset\omega$ is a strict subdomain of $\omega$.
\end{lemma}
\begin{lemma}\label{lema 2.5}\cite{WY} For any constant $0<\sigma\leq N-2$, there exists a constant $C>0$ such that
    \begin{equation*}
     \displaystyle{\displaystyle{\int_{\R^N}}}\frac{1}{|y-x|^{N-2}}\frac{1}{(1+|x|)^{2+\sigma}}dx\leq
     \begin{cases}
       \,\frac{C}{(1+|y|)^{\sigma}},\quad&\text{if $\sigma< N-2$,}\\
      \,\frac{C}{(1+|y|)^{\sigma}}\ln|y|,\quad&\text{if $\sigma= N-2$.}
    \end{cases}
\end{equation*}
\end{lemma}
\begin{lemma}\label{lema 2.6} \cite[Lemma 3.6]{SYZ} For any constant $\sigma\geq N-2-\frac{\mu}{2}$ and $\mu\in(0,4]$, there exists a constant $C>0$ such that
\begin{equation*}
\displaystyle{\int_{\R^N}}\frac{1}{|y-x|^{\frac{2N(N-2)}{2N-\mu}}}\frac{1}{(1+|x|)^{\frac{2N(2+\sigma)}{2N-\mu}}}dx\leq
     \begin{cases}
      \, \frac{C}{(1+|y|)^{\frac{N(2\sigma+\mu)}{2N-\mu}}},\,\,&\text{if $\sigma> N-2-\frac{\mu}{2}$ and $\mu\in(0,4)$,}\\
       \,\frac{C}{(1+|y|)^{\frac{N(2\sigma+\mu)}{2N-\mu}}}\ln|y|,\,\,&\text{if $\sigma= N-2-\frac{\mu}{2}$ and $\mu=4$.}
    \end{cases}
\end{equation*}
\end{lemma}

\vskip0.2cm
\section{ Some basic estimates}
\renewcommand{\theequation}{B.\arabic{equation}}

\setcounter{equation}{0}

\vskip 0.2cm
In this Appendix, we present some proofs of the basic estimates. Firstly, we show two computational techniques, which are useful in the estimates of the double integrals involving the Talenti bubbles $U_{\xi,\tau}(x)$, $\forall \,x,\,\xi\in\R^N$ and $\tau\in\R^{+}$.
\begin{lemma}\label{integral}
As for the Talenti bubbles
$$U_{0,1}(x)=\frac{1}{\left(1+x^{2}\right)^{\frac{N-2}{2}}},$$
we derive the following two equations
\begin{equation}\label{a-1}
    \begin{split}
        &2^{*}_\mu\displaystyle{\int_{\R^N}}U_{0,1}^{2^{*}_\mu-1}(x)\left(\sum_{j=1}^{i-1}a_j\frac{\partial U_{0,1}(x)}{\partial x_j}\right)\left(\displaystyle{\int_{\R^N}}\frac{U^{2^{*}_\mu-1}_{0,1}(y)\left(\sum_{l=1}^{i-1}a_l\frac{\partial U_{0,1}(y)}{\partial y_l}\right)}{|x-y|^{\mu}}dy\right)dx  \\
=&\left(2^{*}-2^{*}_\mu\right)\displaystyle{\int_{\R^N}}U_{0,1}^{2^{*}-2}(x)\left(\sum_{j=1}^{i-1}a_j\frac{\partial {U_{0,1}(x)}}{\partial x_j}\right)\left(\sum_{l=1}^{i-1}a_l\frac{\partial U_{0,1}(x)}{\partial x_l}\right)dx,
    \end{split}
\end{equation}
and
\begin{equation}\label{a-2}
    \begin{split}
       &\frac{2^{*}_\mu(N-2)}{2} \int_{\R^N}U_{0,1}^{2^{*}_\mu-1}(x)\,\frac{1-|x|^{2}}{(1+x^{2})^{\frac{N}{2}}}\left(\int_{\R^N}\frac{U_{0,1}^{2^{*}_\mu-1}(y)\frac{1-|y|^{2}}{(1+y^{2})^{\frac{N}{2}}}}{|x-y|^{\mu}}dy\right)dx\\
       =&\frac{\mu}{2}\int_{\R^N}U_{0,1}^{2^{*}-2}(x)\,\frac{\left(1-|x|^{2}\right)^{2}}{\left(1+x^{2}\right)^N}\,dx.
    \end{split}
\end{equation}
\end{lemma}
\begin{proof}
As for \eqref{a-1}, since
$$\frac{\partial U_{0,1}(y)}{\partial y_{l}}=-\frac{\partial U_{\xi,1}(y)}{\partial \xi_{l}}\Bigg|_{\xi=0}, \quad \text{for $l=1,2,\cdots,N$,}$$
and \eqref{eq 1.7} holds with $\tau=1$, i.e.
$$\displaystyle{\int_{\R^N}}\frac{U_{\xi,1}^{2^{*}_{\mu}}(y)}{|x-y|^{\mu}}dy=U_{\xi,1}^{2^{*}-2^{*}_{\mu}}(x),\quad \forall\, x,\,\xi \in \R^{N},$$
a direct simple computation yields that
\begin{align*}
&2^{*}_\mu\displaystyle{\int_{\R^N}}\frac{U_{0,1}^{2^{*}_\mu-1}(y)\left(\sum_{l=1}^{i-1}a_l\frac{\partial U_{0,1}(y)}{\partial y_l}\right)}{|x-y|^{\mu}}dy=\sum_{l=1}^{i-1}a_l\displaystyle{\int_{\R^N}}\frac{\frac{\partial }{\partial y_l}\,\left(U^{2^{*}_\mu}_{0,1}(y)\right)}{|x-y|^{\mu}}\,dy\\
=&\sum_{l=1}^{i-1}a_l\,\displaystyle{\int_{\R^N}}\frac{\left(-\frac{\partial }{\partial \xi_l}\,\left(U^{2^{*}_\mu}_{\xi,1}(y)\right)\right)\Big|_{\xi=0}}{|x-y|^{\mu}}\,dy=-\sum_{l=1}^{i-1}a_l\left(\frac{\partial}{\partial \xi_l}\left(\displaystyle{\int_{\R^N}}\frac{U_{\xi,1}^{2^{*}_\mu}(y)}{|x-y|^{\mu}}dy\right)\Bigg|_{\xi=0}\right)\\
=&-\sum_{l=1}^{i-1}a_l\frac{\partial}{\partial \xi_l}\left(U_{\xi,1}^{2^{*}-2^{*}_\mu}(x)\right)\Big|_{\xi=0}=\left(2^{*}-2^{*}_\mu\right)U_{0,1}^{2^{*}-2^{*}_\mu-1}(x)\left(\sum_{l=1}^{i-1}a_l\frac{\partial U_{0,1}(x)}{\partial x_l}\right).
\end{align*}
This gives \eqref{a-1}.
As for $\eqref{a-2}$, remark that
$$\frac{\partial U_{0,\tau}(y)}{\partial\tau}\Big|_{\tau=1}
=\frac{N-2}{2}\,\frac{1-|y|^{2}}{(1+y^{2})^{\frac{N}{2}}}=y\cdot\nabla U_{0,1}(y)+\frac{N-2}{2}U_{0,1}(y),$$
and \eqref{eq 1.7} holds with $\xi=0$, i.e.
$$\displaystyle{\int_{\R^N}}\frac{U_{0,\tau}^{2^{*}_{\mu}}(y)}{|x-y|^{\mu}}dy=U_{0,\tau}^{2^{*}-2^{*}_{\mu}}(x),\quad\forall\,\tau \in \R^{+},$$
then we have
\begin{align*}
    & \frac{2^{*}_\mu(N-2)}{2}\int_{\R^N}\frac{U_{0,1}^{2^{*}_\mu-1}(y)\,\frac{1-|y|^{2}}{(1+y^{2})^{\frac{N}{2}}}}{|x-y|^{\mu}}dy=\int_{\R^N}\frac{\frac{\partial }{\partial\tau}\left(U_{0,\tau}^{2^{*}_\mu}(y)\right)\Big|_{\tau=1}}{|x-y|^{\mu}}\,dy\\
     =&\frac{\partial}{\partial\tau}\left(\int_{\R^N}\frac{U_{0,\tau}^{2^{*}_\mu}(y)}{|x-y|^{\mu}}dy\right)\Bigg|_{\tau=1}=\frac{\partial}{\partial\tau}\left(U_{0,\tau}^{2^{*}-2^{*}_\mu}(x)\right)\Big|_{\tau=1}\\
     =&(2^{*}-2^{*}_\mu)\,U_{0,1}^{2^{*}-2^{*}_\mu-1}(x)\,\left(\frac{\partial U_{0,\tau}(x)}{\partial\tau}\Big|_{\tau=1}\right)=\frac{\mu}{2}\,U_{0,1}^{2^{*}-2^{*}_\mu-1}(x)\,\frac{1-|x|^{2}}{(1+x^{2})^{\frac{N}{2}}}.
\end{align*}
Thus, we get \eqref{a-2}.
\end{proof}

\begin{lemma}\label{lma A-3} For any $x\in\Omega$, we denote the test-functions as follows:
$$\frac{\partial u_\varepsilon}{\partial x_{i}}\,(i=1,2,\cdots N)\quad\text{and}\quad(x-x_{\varepsilon})\cdot\nabla u_{\varepsilon}(x)+\frac{N-2}{2}u_{\varepsilon}(x),$$
then a direct computation yields that
\begin{equation}\label{eqA-4}
    \begin{split}
        -\Delta\left(\frac{\partial u_{\varepsilon}}{\partial x_i}\right)
        =&\varepsilon \left(\frac{\partial u_{\varepsilon}}{\partial x_i}\right)+(2^{*}_\mu-1)u_\varepsilon^{2^{*}_\mu-2}\left(\frac{\partial u_{\varepsilon}}{\partial x_i}\right)\left(\displaystyle{\int_{\Omega}}\frac{u_\varepsilon^{2^{*}_{\mu}}(y)}{|x-y|^{\mu}}dy\right)\\
        &+2^{*}_\mu u_\varepsilon^{2^{*}_\mu-1}(x)\left(\displaystyle{\int_{\Omega}}\frac{u_\varepsilon^{2^{*}_\mu-1}(y)\frac{\partial u_\varepsilon(y) }{\partial y_{i}}}{|x-y|^{\mu}}dy\right),\quad\text{for $i=1,\cdots,N$,}
    \end{split}
\end{equation}
and
 \begin{equation}\label{eqA-5}
     \begin{split}
        &-\Delta\Big((x-x_{\varepsilon})\cdot\nabla u_{\varepsilon}(x)+\frac{N-2}{2}u_{\varepsilon}(x)\Big)\\
     =&\varepsilon\left((x-x_\varepsilon)\cdot\nabla u_{\varepsilon}(x)+\frac{N-2}{2}u_{\varepsilon}(x)\right)\\
    &+(2^{*}_\mu-1)u_\varepsilon^{2^{*}_\mu-2}(x)\left((x-x_\varepsilon)\cdot\nabla_{x} u_{\varepsilon}(x)+\frac{N-2}{2}u_{\varepsilon}(x)\right)\left(\displaystyle{\int_{\Omega}}\frac{u_\varepsilon^{2^{*}_{\mu}}(y)}{|x-y|^{\mu}}dy\right)\\
     &+2^{*}_\mu u_\varepsilon^{2^{*}_\mu-1}(x)\left(\displaystyle{\int_{\Omega}}\frac{u_\varepsilon^{2^{*}_{\mu}-1}(y)\big((y-x_{\varepsilon})\cdot\nabla_{y} u_\varepsilon(y)+\frac{N-2}{2} u_\varepsilon(y)\big)}{|x-y|^{\mu}}dy\right)+2\varepsilon u_\varepsilon(x).
     \end{split}
 \end{equation}
\end{lemma}
\begin{proof}
    Since $u_\varepsilon$ is a solution of \eqref{eq 1.1}, differentiating it with respect to $x_i$, for $i=1,2,\cdots,N$, we derive that
\begin{align*}
-\Delta\left(\frac{\partial u_{\varepsilon}}{\partial x_i}\right)
         =&\varepsilon \left(\frac{\partial u_{\varepsilon}}{\partial x_i}\right)+(2^{*}_\mu-1)u_\varepsilon^{2^{*}_\mu-2}\left(\frac{\partial u_{\varepsilon}}{\partial x_i}\right)\left(\displaystyle{\int_{\Omega}}\frac{u_\varepsilon^{2^{*}_{\mu}}(y)}{|x-y|^{\mu}}dy\right)\\
&+u_\varepsilon^{2^{*}_\mu-1}(x)\left(\displaystyle{\int_{\Omega}}\frac{\partial}{\partial x_{i}}\left(\frac{1}{|x-y|^{\mu}}\right)u_\varepsilon^{2^{*}_\mu}(y)dy\right)\\
        =&\varepsilon \left(\frac{\partial u_{\varepsilon}}{\partial x_i}\right)+(2^{*}_\mu-1)u_\varepsilon^{2^{*}_\mu-2}\left(\frac{\partial u_{\varepsilon}}{\partial x_i}\right)\left(\displaystyle{\int_{\Omega}}\frac{u_\varepsilon^{2^{*}_{\mu}}(y)}{|x-y|^{\mu}}dy\right)\\
        &+2^{*}_\mu u_\varepsilon^{2^{*}_\mu-1}(x)\left(\displaystyle{\int_{\Omega}}\frac{u_\varepsilon^{2^{*}_\mu-1}(y)\frac{\partial u_\varepsilon(y) }{\partial y_{i}}}{|x-y|^{\mu}}dy\right),
    \end{align*}
where we have used the fact that
$$\frac{\partial}{\partial x_{i}}\left(\frac{1}{|x-y|^{\mu}}\right)=-\frac{\partial}{\partial y_{i}}\left(\frac{1}{|x-y|^{\mu}}\right).$$
Hence, we derive \eqref{eqA-4}. On the other hand, define $\omega_{\varepsilon}(x):=(x-x_{\varepsilon})\cdot\nabla u_{\varepsilon}(x)+\frac{N-2}{2}u_{\varepsilon}(x)$. Then, a direct computation yields that
\begin{align*}
   -\Delta \omega_{\varepsilon}=&-2\Delta u_\varepsilon+(x-x_\varepsilon)\cdot\nabla(-\Delta u_\varepsilon)+\frac{N-2}{2}(-\Delta u_\varepsilon)\\
=&~\varepsilon\omega_{\varepsilon}+(2^{*}_\mu-1)u_\varepsilon^{2^{*}_\mu-2}(x)\omega_{\varepsilon}(x)\left(\displaystyle{\int_{\Omega}}\frac{u_\varepsilon^{2^{*}_{\mu}}(y)}{|x-y|^{\mu}}dy\right)+2\varepsilon u_\varepsilon+\frac{\mu}{2}u_\varepsilon^{2^{*}_\mu-1}(x)\left(\displaystyle{\int_{\Omega}}\frac{u_\varepsilon^{2^{*}_{\mu}}(y)}{|x-y|^{\mu}}dy\right)\\
    \\
&+\underbrace{u_\varepsilon^{2^{*}_\mu-1}(x)\left((x-x_\varepsilon)\cdot\nabla_{x} \left(\displaystyle{\int_{\Omega}}\frac{u_\varepsilon^{2^{*}_{\mu}}(y)}{|x-y|^{\mu}}dy\right)\right)}_{:=Q}\\
    =&~\varepsilon\omega_{\varepsilon}+(2^{*}_\mu-1)u_\varepsilon^{2^{*}_\mu-2}(x)\omega_{\varepsilon}(x)\left(\displaystyle{\int_{\Omega}}\frac{u_\varepsilon^{2^{*}_{\mu}}(y)}{|x-y|^{\mu}}dy\right)+2^{*}_\mu u_\varepsilon^{2^{*}_\mu-1}(x)\left(\displaystyle{\int_{\Omega}}\frac{u_\varepsilon^{2^{*}_{\mu}-1}(y)\omega_{\varepsilon}(y)}{|x-y|^{\mu}}dy\right)\\
    &+2\varepsilon u_\varepsilon,
\end{align*}
since
    \begin{align*}
Q:=&~u_\varepsilon^{2^{*}_\mu-1}(x)\left(\displaystyle{\int_{\Omega}}u_\varepsilon^{2^{*}_{\mu}}(y)\left(\sum_{i=1}^{N}(x_i-x_{\varepsilon,i})\,\frac{\partial}{\partial x_i}\left(\frac{1}{|x-y|^{\mu}}\right)\right)dy\right)\\
  =& u_\varepsilon^{2^{*}_\mu-1}(x)\left(\displaystyle{\int_{\Omega}}\left(\sum_{i=1}^{N}(y_{i}-x_{\varepsilon,i})\frac{\partial \left( u_\varepsilon^{2^{*}_{\mu}}(y)\right)}{\partial y_i}\right)\frac{1}{|x-y|^{\mu}}dy\right)\\
  &+ u_\varepsilon^{2^{*}_\mu-1}(x)\left(\displaystyle{\int_{\Omega}}\left(\sum_{i=1}^{N}(x_i-y_{i})\frac{\partial\left( u_\varepsilon^{2^{*}_{\mu}}(y)\right)}{\partial y_i}\right)\frac{1}{|x-y|^{\mu}}dy\right)\\
  =&2^{*}_\mu u_\varepsilon^{2^{*}_\mu-1}(x)\left(\displaystyle{\int_{\Omega}}\frac{u_\varepsilon^{2^{*}_{\mu}-1}(y)\omega_{\varepsilon}(y)}{|x-y|^{\mu}}dy\right)-\frac{2N-\mu}{2}u_\varepsilon^{2^{*}_\mu-1}(x)\left(\displaystyle{\int_{\Omega}}\frac{u_\varepsilon^{2^{*}_{\mu}}(y)}{|x-y|^{\mu}}dy\right)\\
  &- u_\varepsilon^{2^{*}_\mu-1}(x)\left(\displaystyle{\int_{\Omega}}\left(\sum_{i=1}^{N}\frac{\partial }{\partial y_i}\left(\frac{x_i-y_i}{|x-y|^{\mu}}\right) \right)u_\varepsilon^{2^{*}_{\mu}}(y)dy\right)\\
  =&2^{*}_\mu u_\varepsilon^{2^{*}_\mu-1}(x)\left(\displaystyle{\int_{\Omega}}\frac{u_\varepsilon^{2^{*}_{\mu}-1}(y)\omega_{\varepsilon}(y)}{|x-y|^{\mu}}dy\right)-\frac{\mu}{2}u_\varepsilon^{2^{*}_\mu-1}(x)\left(\displaystyle{\int_{\Omega}}\frac{u_\varepsilon^{2^{*}_{\mu}}(y)}{|x-y|^{\mu}}dy\right).
\end{align*}
Therefore, we conclude that \eqref{eqA-5} holds true.
\end{proof}

In the following, we give the proof of \eqref{eq4.18} in Lemma \ref{lema 4.2}.
\begin{proof}[\textbf{Proof of \eqref{eq4.18}}] 
It follows from \eqref{eq 4.16}, the denominator of $\lambda_{i,\varepsilon}$ can be computed as follows
\begin{equation}\label{eqA-1}
 \begin{split}
&\displaystyle{\int_{\Omega}}\varepsilon v^{2}dx+(2^{*}_\mu-1)\displaystyle{\int_{\Omega}}\displaystyle{\int_{\Omega}}\frac{u_{\varepsilon}^{2^{*}_\mu-2}(x)v^{2}(x)u_{\varepsilon}^{2^{*}_\mu}(y)}{|x-y|^{\mu}}dxdy\\
&+2^{*}_\mu \displaystyle{\int_{\Omega}}\displaystyle{\int_{\Omega}}\frac{u_{\varepsilon}^{2^{*}_\mu-1}(x)v(x)u_{\varepsilon}^{2^{*}_\mu-1}(y)v(y)}{|x-y|^{\mu}}dxdy\\
  =&a_{0}^{2}\bigg(\varepsilon \displaystyle{\int_{\Omega}}u_{\varepsilon}^{2}(x)dx+(2\cdot2^{*}_\mu-1)\displaystyle{\int_{\Omega}}\displaystyle{\int_{\Omega}}\frac{u_{\varepsilon}^{2^{*}_\mu}(x)u_{\varepsilon}^{2^{*}_\mu}(y)}{|x-y|^{\mu}}dxdy\bigg)\\
  &+2a_{0}\bigg(\varepsilon \displaystyle{\int_{\Omega}}
 u_{\varepsilon}(x)\phi(x) z_{\varepsilon}(x)dx+(2^{*}_\mu-1)\displaystyle{\int_{\Omega}}\displaystyle{\int_{\Omega}}\frac{u_{\varepsilon}^{2^{*}_\mu-1}(x)\phi(x)z_{\varepsilon}(x)u_{\varepsilon}^{2^{*}_\mu}(y)}{|x-y|^{\mu}}dxdy\bigg)\\
&+2^{*}_\mu\displaystyle{\int_{\Omega}}\displaystyle{\int_{\Omega}}\frac{u_{\varepsilon}^{2^{*}_\mu-1}(x)u_{\varepsilon}^{2^{*}_\mu-1}(y)\big(a_{0}u_{\varepsilon}(x)\phi(y)z_{\varepsilon}(y)+a_{0}u_{\varepsilon}(y)\phi(x)z_{\varepsilon}(x)\big)}{|x-y|^{\mu}}dxdy\\
&+\varepsilon \displaystyle{\int_{\Omega}}\phi^{2}(x)z_{\varepsilon}^{2}(x)dx+(2^{*}_\mu-1)\displaystyle{\int_{\Omega}}\displaystyle{\int_{\Omega}}\frac{u_{\varepsilon}^{2^{*}_\mu-2}(x)\phi^{2}(x)z_{\varepsilon}^{2}(x)u_{\varepsilon}^{2^{*}_\mu}(y)}{|x-y|^{\mu}}dxdy\\
&+2^{*}_\mu\displaystyle{\int_{\Omega}}
\displaystyle{\int_{\Omega}}\frac{u_{\varepsilon}^{2^{*}_\mu-1}(x)\phi(x)z_{\varepsilon}(x)u_{\varepsilon}^{2^{*}_\mu-1}(y)\phi(y)z_{\varepsilon}(y)}{|x-y|^{\mu}}dxdy.
\end{split}
\end{equation}
Due to the symmetry, we get that for any $x,y\in \Omega$,
\begin{align*}
\displaystyle{\int_{\Omega}}\displaystyle{\int_{\Omega}}\frac{u_{\varepsilon}^{2^{*}_\mu}(x)u_{\varepsilon}^{2^{*}_\mu-1}(y)\phi(y)z_{\varepsilon}(y)}{|x-y|^{\mu}}dxdy=\displaystyle{\int_{\Omega}}\displaystyle{\int_{\Omega}}\frac{u_{\varepsilon}^{2^{*}_\mu}(y)u_{\varepsilon}^{2^{*}_\mu-1}(x)\phi(x)z_{\varepsilon}(x)}{|x-y|^{\mu}}dxdy,
\end{align*}
then \eqref{eqA-1} reduces to
\begin{equation}\label{eq 4.17}
    \begin{split}
&\underbrace{a_{0}^{2}\bigg(\varepsilon \displaystyle{\int_{\Omega}}u_{\varepsilon}^{2}(x)dx+(2\cdot2^{*}_\mu-1)\displaystyle{\int_{\Omega}}\displaystyle{\int_{\Omega}}\frac{u_{\varepsilon}^{2^{*}_\mu}(x)u_{\varepsilon}^{2^{*}_\mu}(y)}{|x-y|^{\mu}}dxdy\bigg)}_{:=D_{1,\varepsilon}}\\
  +&\underbrace{2a_{0}\bigg(\varepsilon \displaystyle{\int_{\Omega}}
 u_{\varepsilon}(x)\phi(x) z_{\varepsilon}(x)dx+(2\cdot2^{*}_\mu-1)\displaystyle{\int_{\Omega}}\displaystyle{\int_{\Omega}}\frac{u_{\varepsilon}^{2^{*}_\mu-1}(x)\phi(x)z_{\varepsilon}(x)u_{\varepsilon}^{2^{*}_\mu}(y)}{|x-y|^{\mu}}dxdy\bigg)}_{:=D_{2,\varepsilon}}\\
 +&\underbrace{\varepsilon \displaystyle{\int_{\Omega}}\phi^{2}(x)z_{\varepsilon}^{2}(x)dx+(2^{*}_\mu-1)\displaystyle{\int_{\Omega}}\displaystyle{\int_{\Omega}}\frac{u_{\varepsilon}^{2^{*}_\mu-2}(x)\phi^{2}(x)z_{\varepsilon}^{2}(x)u_{\varepsilon}^{2^{*}_\mu}(y)}{|x-y|^{\mu}}dxdy}_{:=\widehat{D_{1}}}\\
+&\underbrace{2^{*}_\mu\displaystyle{\int_{\Omega}}\displaystyle{\int_{\Omega}}\frac{u_{\varepsilon}^{2^{*}_\mu-1}(x)\phi(x)z_{\varepsilon}(x)u_{\varepsilon}^{2^{*}_\mu-1}(y)\phi(y)z_{\varepsilon}(y)}{|x-y|^{\mu}}dxdy}_{:=\widehat{D_{2}}}\\
:=&D_{1,\varepsilon}+D_{2,\varepsilon}+D_{3,\varepsilon}=D_{\varepsilon},
    \end{split}
\end{equation}
where $D_{3,\varepsilon}:=\widehat{D_{1}}+\widehat{D_{2}}$. Next, we evaluate the numerator of $\lambda_{i,\varepsilon}$. From \eqref{eq4.15} and \eqref{eq4.16}, a direct calculation yields that
\begin{equation}\label{eq 4.21}
    \begin{split}
\displaystyle{\int_{\Omega}}|\nabla v|^{2}dx=&\displaystyle{\int_{\Omega}}|\nabla (a_0 u_\varepsilon+\phi z_\varepsilon)|^{2}dx\\
=&a_0^{2}\displaystyle{\int_{\Omega}}|\nabla u_\varepsilon|^{2}dx+2a_0\displaystyle{\int_{\Omega}}\nabla u_\varepsilon\nabla(\phi z_\varepsilon)dx+\displaystyle{\int_{\Omega}}|\nabla\phi|^{2}|z_\varepsilon|^{2}dx+\varepsilon\displaystyle{\int_{\Omega}} \phi^{2}(x)z_\varepsilon^{2}(x)dx\\
&+(2^{*}_\mu-1)\displaystyle{\int_{\Omega}} \displaystyle{\int_{\Omega}}\frac{u_{\varepsilon}^{2^{*}_\mu-2}(x)\phi^{2}(x)z_\varepsilon^{2}(x) u_{\varepsilon}^{2^{*}_\mu}(y)}{|x-y|^{\mu}}dxdy\\
    &+2^{*}_\mu \displaystyle{\int_{\Omega}}\displaystyle{\int_{\Omega}}\frac{u_{\varepsilon}^{2^{*}_\mu-1}(x)\phi^{2}(x)z_\varepsilon (x)u_{\varepsilon}^{2^{*}_\mu-1}(y)z_\varepsilon(y)}{|x-y|^{\mu}}dxdy.
    \end{split}
\end{equation}
Since $u_{\varepsilon}$ is the solution of \eqref{eq 1.1}, we derive that
\begin{align*}
  & a_{0}^{2} \displaystyle{\int_{\Omega}}|\nabla u_\varepsilon(x)|^{2}dx
   =a_{0}^{2} \left(\displaystyle{\int_{\Omega}}\varepsilon u_\varepsilon^{2}(x)dx+\displaystyle{\int_{\Omega}}\displaystyle{\int_{\Omega}}\frac{u_{\varepsilon}^{2^{*}_\mu}(x)
   u_{\varepsilon}^{2^{*}_\mu}(y)}{|x-y|^{\mu}}dxdy\right)\\
   =&\underbrace{a_{0}^{2} \left(\displaystyle{\int_{\Omega}}\varepsilon u_\varepsilon^{2}(x)dx+(2\cdot2^{*}_\mu-1)\displaystyle{\int_{\Omega}}
   \displaystyle{\int_{\Omega}}\frac{u_{\varepsilon}^{2^{*}_\mu}(x)u_{\varepsilon}^{2^{*}_\mu}(y)}{|x-y|^{\mu}}dxdy\right)}
   _{D_{1,\varepsilon}}+\underbrace{a_{0}^{2}(2-2\cdot2^{*}_\mu) \displaystyle{\int_{\Omega}}\displaystyle{\int_{\Omega}}\frac{u_{\varepsilon}^{2^{*}_\mu}(x)u_{\varepsilon}^{2^{*}_\mu}(y)}{|x-y|^{\mu}}dxdy}_{:=N_{1,\varepsilon}}.
\end{align*}
Moreover, from \eqref{eq4.17}, we find that
\begin{align*}
&2a_{0}\displaystyle{\int_{\Omega}}\nabla u_{\varepsilon}\nabla(\phi z_\varepsilon)dx
    =\underbrace{2a_{0}(2-2\cdot2^{*}_\mu)\displaystyle{\int_{\Omega}}\displaystyle{\int_{\Omega}}
    \frac{u_{\varepsilon}^{2^{*}_\mu-1}(x)\phi(x)z_{\varepsilon}(x)u_{\varepsilon}^{2^{*}_\mu}(y)}{|x-y|^{\mu}}dxdy}_{:=N_{2,\varepsilon}}\\
    &+\underbrace{2a_{0}\bigg(\varepsilon \displaystyle{\int_{\Omega}}
 u_{\varepsilon}(x)\phi(x) z_{\varepsilon}(x)dx+(2\cdot2^{*}_\mu-1)\displaystyle{\int_{\Omega}}\displaystyle{\int_{\Omega}}
 \frac{u_{\varepsilon}^{2^{*}_\mu-1}(x)\phi(x)z_{\varepsilon}(x)u_{\varepsilon}^{2^{*}_\mu}(y)}{|x-y|^{\mu}}dxdy\bigg)}_{D_{2,\varepsilon}}.
\end{align*}
As for the remaining terms in \eqref{eq 4.21}, we have
\begin{align*}
&\quad\displaystyle{\int_{\Omega}}|\nabla\phi|^{2}|z_\varepsilon|^{2}+\varepsilon\displaystyle{\int_{\Omega}} \phi^{2}(x)z_\varepsilon^{2}(x)dx+(2^{*}_\mu-1)\displaystyle{\int_{\Omega}} \displaystyle{\int_{\Omega}}\frac{u_{\varepsilon}^{2^{*}_\mu-2}(x)\phi^{2}(x)z_\varepsilon^{2}(x) u_{\varepsilon}^{2^{*}_\mu}(y)}{|x-y|^{\mu}}dxdy\\
   &+2^{*}_\mu \displaystyle{\int_{\Omega}}\displaystyle{\int_{\Omega}}\frac{u_{\varepsilon}^{2^{*}_\mu-1}(x)\phi^{2}(x)z_\varepsilon (x)u_{\varepsilon}^{2^{*}_\mu-1}(y)z_\varepsilon(y)}{|x-y|^{\mu}}dxdy   \\
=&\underbrace{\varepsilon \displaystyle{\int_{\Omega}}\phi^{2}(x)z_{\varepsilon}^{2}(x)dx+(2^{*}_\mu-1)
\displaystyle{\int_{\Omega}}\displaystyle{\int_{\Omega}}
\frac{u_{\varepsilon}^{2^{*}_\mu-2}(x)\phi^{2}(x)z_{\varepsilon}^{2}(x)u_{\varepsilon}^{2^{*}_\mu}(y)}{|x-y|^{\mu}}dxdy}_{\widehat{D_{1}}}\\
&+\underbrace{2^{*}_\mu\displaystyle{\int_{\Omega}}\displaystyle{\int_{\Omega}}
\frac{u_{\varepsilon}^{2^{*}_\mu-1}(x)\phi(x)z_{\varepsilon}(x)u_{\varepsilon}^{2^{*}_\mu-1}(y)\phi(y)z_{\varepsilon}(y)}{|x-y|^{\mu}}dxdy}_{\widehat{D_{2}}}\\
&+\underbrace{\displaystyle{\int_{\Omega}}|\nabla\phi|^{2}|z_\varepsilon|^{2}dx+2^{*}_\mu \displaystyle{\int_{\Omega}}\displaystyle{\int_{\Omega}}\frac{u_{\varepsilon}^{2^{*}_\mu-1}(x)\big(\phi(x)-\phi(y)\big)\phi(x)z_\varepsilon (x)u_{\varepsilon}^{2^{*}_\mu-1}(y)z_\varepsilon(y)}{|x-y|^{\mu}}dxdy}_{:=N_{3,\varepsilon}}\\
:=&D_{3,\varepsilon}+N_{3,\varepsilon}.
\end{align*}
Then, from above estimates, we conclude that the numerator \eqref{eq 4.21} is equal to
\begin{align}\label{B.8}
\displaystyle{\int_{\Omega}}|\nabla v|^{2}dx=D_{1,\varepsilon}+D_{2,\varepsilon}+D_{3,\varepsilon}+N_{1,\varepsilon}+N_{2,\varepsilon}+N_{3,\varepsilon}=D_{\varepsilon}+N_{\varepsilon},
\end{align}
Thus, from \eqref{eq 4.17} and \eqref{B.8}, a simple computation yields that
\begin{align*}
\lambda_{i,\varepsilon}\leq\underset{a_0,\,a_1,\cdots,\,a_{i-1}}{\max}{\left\{1+\frac{N_\varepsilon}{D_\varepsilon}\right\}}.
\end{align*}
\end{proof}

\end{document}